\documentclass{article}
\usepackage{amsmath}
\usepackage{amsthm}
\usepackage{amsfonts}
\usepackage{amssymb}
\usepackage{graphicx}
\usepackage{longtable}
\usepackage{amscd,enumerate}%
\usepackage{cite}
\usepackage{color}
\providecommand{\U}[1]{\protect\rule{.1in}{.1in}}
\newtheorem{theorem}{Theorem}[section]

\newtheorem{corollary}[theorem]{Corollary}

\newtheorem{definition}[theorem]{Definition}

\newtheorem{lemma}[theorem]{Lemma}

\newtheorem{remark}[theorem]{Remark}

\makeatletter\@addtoreset{equation}{section}\makeatother
\evensidemargin=\oddsidemargin
\newdimen\dummy
\dummy=\oddsidemargin
\newcommand{\hatPicurlcom}{\widehat \Pi^{\operatorname*{curl},3d}_p}
\newcommand{\hatPicurlcomtwod}{\widehat \Pi^{\operatorname*{curl},2d}_p}

\newcommand{\hatPigradcom}{\widehat\Pi^{\operatorname*{grad},3d}_{p+1}}
\newcommand{\hatPigradcomtwod}{\widehat\Pi^{\operatorname*{grad},2d}_{p+1}}

\newcommand{\hatPidivcom}{\widehat\Pi^{\operatorname*{div},3d}_{p}}

\def\XXint#1#2#3{{\setbox0=\hbox{$#1{#2#3}{\int}$ }
\vcenter{\hbox{$#2#3$ }}\kern-.6\wd0}}

\newcommand{\eremk}{\hbox{}\hfill\rule{0.8ex}{0.8ex}}
\begin{document}
\title{On commuting $p$-version projection-based interpolation on tetrahedra}
\author{J.M. Melenk\thanks{\{(melenk@tuwien.ac.at, claudio.rojik@tuwien.ac.at)\} Institut f\"{u}r Analysis und
Scientific Computing, Technische Universit\"{a}t Wien, Wiedner Hauptstrasse
8-10, A-1040 Wien, Austria.}
\and C. Rojik\footnotemark[1]}
\maketitle
%\author{C.~Rojik}
%\address{Institut f\"ur Analysis und Scientific Computing,
%       Technische Universit\"at Wien, 
%       Wiedner Hauptstra\ss{}e 8-10,
%       A-1040 Wien, Austria}
%       \email{\{melenk,\,claudio.rojik\}@tuwien.ac.at}

\begin{abstract}
On the reference tetrahedron $\widehat K$, we define three projection-based interpolation operators 
on $H^2(\widehat K)$, ${\mathbf H}^1(\widehat K,\operatorname{\mathbf{curl}})$, 
and ${\mathbf H}^1(\widehat K,\operatorname{div})$.  
These operators are projections onto spaces of polynomials, they have the commuting diagram
property and feature the optimal convergence rate as the polynomial degree increases
in $H^{1-s}(\widehat K)$, $\widetilde{\mathbf{H}}^{-s}(\widehat K,\operatorname{\mathbf{curl}})$, 
$\widetilde{\mathbf{H}}^{-s}(\widehat K,\operatorname{div})$ for $0 \leq s \leq 1$. 
\end{abstract}
%\keywords{$p$-version FEM, commuting diagram, Maxwell equations}
%\subjclass{65N30, 65N35}

\maketitle

\section{Introduction} 
Operators that approximate a given function by a (piecewise) polynomial are 
fundamental tools in numerical analysis. The case of scalar functions is 
rather well-understood and many such approximation operators exist both for 
fixed order approximation where accuracy is achieved by refining the mesh, 
the so-called $h$-version, and the $p$-version, 
where accuracy is obtained by increasing the polynomial degree $p$; for the $p$-version
in an $H^1$-conforming setting 
we refer to \cite{babuska-suri94,SchwabhpBook,apel-melenk17} and references therein. 
For the approximation of vector-valued functions, specifically, the approximation 
in the spaces ${\mathbf H}(\operatorname{\mathbf{curl}})$ and 
${\mathbf H}(\operatorname{div})$, the situation is less mature since the approximation 
operators should satisfy, in addition to having certain approximation properties, 
also the requirement to be projections and to have a commuting diagram property. 
While various operators with all these desirable properties were developed for the $h$-version, 
optimal results in the $p$-version are missing in the literature. The present paper is devoted to the 
analysis of a $p$-version projection-based interpolation operator that has the optimal 
polynomial approximation properties under suitable regularity assumptions. 

High order polynomial projection-based interpolation operators --- a phrase that first appeared in \cite[Sec.~{3}]{hiptmair-acta} --- with the projection 
and commuting diagram properties have been developed by L.~Demkowicz and several coworkers,  
\cite{demkowicz-babuska03,demkowicz-buffa05,demkowicz-cao05,demkowicz-monk-vardapetyan-rachowicz00};
a very nice and comprehensive presentation of these results can be found in \cite{demkowicz08}, which will also 
be the basis for the present work. 
The projection-based interpolation operators presented in \cite{demkowicz08} are a) projections, b) have the commuting diagram
property, and c) admit element-by-element construction. The last point means that 
the operators are defined elementwise by specifying them on the reference element and that the appropriate
interelement continuity is ensured by defining the interpolant in terms of pertinent traces: 
for scalar functions, the projection-based interpolant interpolates at the vertices and its restriction to 
an edge or a face is fully  determined by the restriction of the function to that edge or face; 
for the ${\mathbf H}(\operatorname{\mathbf{curl}})$-conforming interpolant, its tangential component on an edge or face 
is completely determined by the tangential trace of the function on that edge or face; 
for the ${\mathbf H}(\operatorname{div})$-conforming interpolant, 
the normal component on a face is fully dictated by the normal component of the function on that face. 
Such a construction is only possible under additional regularity assumptions beyond the minimal ones
(which would be $H^1$, ${\mathbf H}(\operatorname{\mathbf{curl}})$ 
or ${\mathbf H}(\operatorname{div})$). Indeed, in 3D, the 
construction described in \cite{demkowicz08} requires the regularity $H^{1+s}$ with $s > 1/2$ for scalar functions, 
${\mathbf H}^s(\operatorname{\mathbf{curl}})$ with $s > 1/2$ 
and ${\mathbf H}^s(\operatorname{div})$ with $s > 0$ for the vectorial ones. 
Under these regularity assumptions, it is shown in \cite[Thm.~{5.3}]{demkowicz08} that the projection-based
interpolation operator has, up to logarithmic factors, the optimal algebraic convergence properties 
(as $p \rightarrow \infty$), for functions with finite Sobolev regularity as measured by $s$. 
In the present work, we remove the logarithmic
factors, i.e., we show optimal rates of convergence, under the more stringent regularity assumption $s \ge 1$
(cf.~Theorem~\ref{thm:projection-based-interpolation} for the case of tetrahedra and Theorem~\ref{thm:projection-based-interpolation-2d} for the case of triangles). 

The projection-based interpolation operator analyzed in the present work is of the type studied 
in \cite{demkowicz08}. Correspondingly, many tools used in \cite{demkowicz08} are also employed here, 
most notably, the polynomial lifting operators developed for tetrahedra in 
\cite{demkowicz-gopalakrishnan-schoeberl-I,demkowicz-gopalakrishnan-schoeberl-II,demkowicz-gopalakrishnan-schoeberl-III}
and for the simpler case of triangles in 
\cite{ainsworth-demkowicz09}; we mention in passing that suitable polynomial lifting operators are also 
available for the case of the cube \cite{costabel-dauge-demkowicz07}. 
Another tool that \cite{demkowicz08} uses are right inverses of the gradient, curl and div operators 
(``Poincar\'e maps''). Here, we utilize a more recent and powerful variant, namely, the regularized right inverses 
of \cite{costabel-mcintosh10}. This breakthrough paper \cite{costabel-mcintosh10} allows for stable decompositions of functions 
in ${\mathbf H}(\operatorname{\mathbf{curl}})$ 
and ${\mathbf H}(\operatorname{div})$ with appropriate mapping properties 
in scales of Sobolev spaces and is an essential component in the analysis of the $p$-version in ${\mathbf H}(\operatorname{\mathbf{curl}})$, 
\cite{bespalov-heuer09, hiptmair08,boffi-costabel-dauge-demkowicz-hiptmair11}. The distinguishing
technical difference between \cite{demkowicz08} and the present work, which is responsible for the removal 
of the logarithmic factor, is the treatment of the non-local norms on the boundary. 
Non-local norms on the boundary are written in \cite{demkowicz08} (following \cite{demkowicz-buffa05}) 
as a sum of contributions over the boundary parts (that is, faces in 3D and edges in 2D); 
in finite-dimensional spaces of piecewise polynomials, this localization procedure is possible at the price 
of logarithmic factors. Instead of localizing a non-local norm, the approach taken here is to 
realize the non-local norm by interpolating between two norms related to integer order Sobolev norms, which 
both can be localized, i.e., written as sums of contributions over boundary parts. In turn, this requires to analyze the 
error of the projection-based interpolation in two norms instead of a single one. 
The estimate in the stronger norm is obtained by a best approximation argument as done in \cite{demkowicz08}, 
the estimate in the weaker norm is obtained by a duality argument. 

We close the introduction with some notation. 
The gradient operator $\nabla$ for scalar functions $u$ and the divergence operator $\operatorname{div}$ for $\mathbb{R}^d$-valued functions $\mathbf{u}$ are defined in the usual way: $\nabla u = (\partial_{x_1} u,\ldots,\partial_{x_d} u)^\top$ and $\operatorname{div} {\mathbf u} = \sum_{i=1}^d \partial_{x_i} {\mathbf u}_i$. 
For $d = 3$ and ${\mathbb R}^3$-valued functions ${\mathbf u}$
the $\operatorname{\mathbf{curl}}$-operator is defined by 
$\operatorname{\mathbf{curl}} {\mathbf u} := 
(\partial_{x_2} {\mathbf u}_3 - \partial_{x_3} {\mathbf u}_2, 
-(\partial_{x_1} {\mathbf u}_3 - \partial_{x_3} {\mathbf u}_1), 
\partial_{x_1} {\mathbf u}_2 - \partial_{x_2} {\mathbf u}_1)^\top$. 
For $d = 2$ we distinguish between the scalar-valued and vector-valued curl operator: 
for a scalar function $u$, we define 
$\operatorname{\mathbf{curl}} u := (\partial_{x_2} u,-\partial_{x_1} u)^\top$
and for an ${\mathbb R}^2$-valued function ${\mathbf u}$ 
we set $\operatorname{{curl}} {\mathbf u} := \partial_{x_1} {\mathbf u}_2-\partial_{x_2} {\mathbf u}_1$. 
For Lipschitz domains $\omega \subset {\mathbb R}^d$ ($d \in \{1,2,3\}$) and 
scalar functions, we employ the usual Sobolev spaces $H^s(\omega)$, $s \ge 0$, as defined, e.g., in
%\cite{adams-fournier03,Mclean00}. 
\cite{Mclean00}. 
The $H^s(\omega)$-seminorm, $s \ge 0$, will be denoted 
$|\cdot|_{H^s(\omega)}$ and the full norm 
$\|\cdot\|_{H^s(\omega)}$. The dual space of $H^s(\omega)$, $s \ge 0$, 
is denoted $\widetilde H^{-s}(\omega)$ and equipped with the norm 
%The spaces $\widetilde H^s(\omega)$, $s \ge 0$, are defined as the closure
%of $C^\infty_0(\omega)$ under the norms $\|\cdot\|_{H^s(\omega)}$. 
%For $s \geq 0$ and $u\in L^2(\omega)$, we define the norm
%the space $\widetilde H^{-s}(\omega):=\left(H^s(\omega)\right)^\prime$ is the dual space of $H^s(\omega)$ characterized by the norm 
\begin{equation}
\label{eq:def-negative-norm}
\|u\|_{\widetilde H^{-s}(\omega)}:= \sup_{v \in H^s(\omega)} \frac{(u,v)_{L^2(\omega)}}{\|v\|_{H^s(\omega)}},
\end{equation}
where $(\cdot,\cdot)_{L^2(\omega)}$ denotes the (extended) $L^2(\omega)$-scalar product. 
Vector-valued analogs ${\mathbf H}^s(\omega)$ are defined to be elements of 
$H^s(\omega)$ componentwise and also the dual norm 
$\|\cdot\|_{\widetilde {\mathbf H}^{-s}(\omega)}$ is defined analogously to 
(\ref{eq:def-negative-norm}). 
We will make use of the fact that the Sobolev scales $H^s(\omega)$ 
and $\widetilde H^{-s}(\omega)$ are interpolation spaces, \cite{Mclean00}.
For $s \ge 0$ and $d = 3$, we set 
${\mathbf H}^s(\omega,\operatorname{\mathbf{curl}}) = \{{\mathbf u} \in {\mathbf H}^s(\omega)\,|\, 
\operatorname{\mathbf{curl}} {\mathbf u} \in {\mathbf H}^s(\omega)\}$ and 
${\mathbf H}^s(\omega,\operatorname{div}) = \{{\mathbf u} \in {\mathbf H}^s(\omega)\,|\, 
\operatorname{div} {\mathbf u} \in H^s(\omega)\}$; for $d = 2$ we have 
${\mathbf H}^s(\omega,\operatorname{curl}) = \{{\mathbf u} \in {\mathbf H}^s(\omega)\,|\, 
\operatorname{curl} {\mathbf u} \in H^s(\omega)\}$. For $s \ge 0$, we define 
$$
\|{\mathbf u}\|^2_{\widetilde {\mathbf H}^{-s}(\omega,\operatorname{\mathbf{curl}})} := 
\|{\mathbf u}\|^2_{\widetilde {\mathbf H}^{-s}(\omega)} + 
\|\operatorname{\mathbf{curl}} {\mathbf u}\|^2_{\widetilde{\mathbf{H}}^{-s}(\omega)}
$$ 
and analogously the norms 
$\|{\mathbf u}\|^2_{\widetilde {\mathbf H}^{-s}(\omega,\operatorname{div})}$ and 
$\|{\mathbf u}\|^2_{\widetilde {\mathbf H}^{-s}(\omega,\operatorname{curl})}$.
%\clr{We set $\|u\|^2_{\widetilde{H}^{-s}(\omega,\nabla)} := \|u\|^2_{\widetilde{H}^{-s}(\omega)} + \|\nabla u\|^2_{\widetilde{\mathbf{H}}^{-s}(\omega)}$ for $d\in\{2,3\}$.}
The space $H^{1/2}(\partial\omega)$ will be understood as the trace space of $H^1(\omega)$
and $H^{-1/2}(\partial\omega)$ denotes its dual. The spaces 
${\mathbf H}_0(\omega,\operatorname{\mathbf{curl}})$ and 
${\mathbf H}_0(\omega,\operatorname{div})$ are the subspaces of 
${\mathbf H}(\omega,\operatorname{\mathbf{curl}}):={\mathbf H}^0(\omega,\operatorname{\mathbf{curl}})$ and 
${\mathbf H}(\omega,\operatorname{div}) := {\mathbf H}^0(\omega,\operatorname{div})$ with vanishing tangential or normal trace, defined as the 
closure of $(C^\infty_0(\omega))^d$ under the norms 
$\|\cdot\|_{{\mathbf H}(\omega,\operatorname{\mathbf{curl}})}$ and 
$\|\cdot\|_{{\mathbf H}(\omega,\operatorname{div})}$. 
$H^1_0(\omega)$ is the subspace of $H^1(\omega)$ of functions with vanishing
trace. 

%---------------------------------------------------------------------------
\section{Projection based interpolation}

\label{sec:thm:projection-based-interpolation}
%---------------------------------------------------------------------------
Throughout the paper $\widehat{K} \subset {\mathbb R}^3$ denotes a reference tetrahedron.
The sets $\mathcal{F}(\widehat{K})$, $\mathcal{E}(\widehat{K})$ and $\mathcal{V}(\widehat{K})$ denote 
the sets of faces, edges, and vertices of $\widehat{K}$, respectively. In the two-dimensional case, 
we use the notation $\widehat{f}\subset {\mathbb R}^2$ for a reference triangle, and $\mathcal{E}(\widehat{f})$ 
and $\mathcal{V}(\widehat{f})$ for the sets of edges and vertices of $\widehat{f}$. We set
\begin{align}
\label{eq:maximal-angle}
\widehat{s}&:=\frac{\pi}{\omega_{max}(\widehat f)}, \quad \pi > \omega_{max}(\widehat f):=\text{maximal interior angle of } \widehat{f}. 
\end{align}
$\widehat{e} = (-1,1) \subset \mathbb{R}$ denotes the reference edge.
We also need the tangential trace and tangential component operators: 
For a sufficiently smooth 
function ${\mathbf u}$ on $\widehat K$ we set 
$\Pi_{\tau} {\mathbf u} := ({\mathbf n} \times {\mathbf u}|_{\partial \widehat K}) \times {\mathbf n}$ and 
$\gamma_\tau {\mathbf u} := {\mathbf n} \times {\mathbf u}|_{\partial\widehat K}$, 
where ${\mathbf n}$ denotes the outer normal
vector of $\widehat K$. We fix the orientation of each face 
$f \in {\mathcal F}(\widehat K)$ by defining its normal vector ${\mathbf n}_f$ to coincide with 
the (outer) normal vector of $\widehat K$ on $f$. In turn, this fixes the orientation of 
the boundary $\partial f$ and for each face 
$f \in {\mathcal F}(\widehat K)$ we will write 
$\Pi_{\tau,f}$ for the (in-plane) tangential trace on $\partial f$, i.e., for each edge 
$e \in {\mathcal E}(f)$ and its unit tangential vector ${\mathbf t}_e$ (with orientation matching the orientation
of $f$) we set $(\Pi_{\tau,f} {\mathbf u})|_{e} =  {\mathbf u}|_e \cdot {\mathbf t}_e$
for sufficiently smooth tangential fields ${\mathbf u}$.

We have the integration by parts formula
\begin{equation}
\label{eq:integration-by-parts} 
(\operatorname{\mathbf{curl}} {\mathbf u},{\mathbf v})_{L^2(\widehat K)} = 
(\operatorname{\mathbf{curl}} {\mathbf v},{\mathbf u})_{L^2(\widehat K)} 
- (\Pi_\tau {\mathbf u},\gamma_\tau {\mathbf v})_{L^2(\partial \widehat K)} 
\qquad \forall {\mathbf u}, {\mathbf v} \in {\mathbf H}^1(\widehat K), 
\end{equation}
which actually extends to 
${\mathbf u}$, ${\mathbf v} \in {\mathbf H}(\widehat K,\operatorname{\mathbf{curl}})$, 
\cite[Thm.~{3.29}]{Monkbook}. (In this case, the $L^2(\partial\widehat K)$-inner product is
replaced with a duality pairing, which we will still denote by $(\cdot,\cdot)_{L^2(\partial\widehat K)}$.)
In 2D, we have the integration by parts formula (Stokes formula)
\begin{equation}
\label{eq:2d-stokes}\int_{\widehat f} \operatorname{\mathbf{curl}} v
\cdot{\mathbf{F}} = \int_{\widehat f} v \operatorname{curl}{\mathbf{F}} -
\int_{\partial \widehat f} v {\mathbf{F}} \cdot{\mathbf{t}},%
\end{equation}
where the boundary $\partial\widehat f$ and, therefore, tangential vectors, are oriented counterclockwise.

For each face $f \in {\mathcal F}(\widehat K)$ and $s \ge 0$ we define the Sobolev spaces $H^s(f)$
%$\widetilde H^{-s}(f)$
as well as ${\mathbf H}_T^s(f,\operatorname{curl})$
%and $\widetilde {\mathbf H}^{-s}_T(f,\operatorname{curl})$
by identifying the face $f$ with a subset of 
${\mathbb R}^2$ via an affine congruence map. The subscript $T$ indicates that \emph{tangential} fields are considered. 
Also the spaces $H^s(e)$
%and $\widetilde H^{-s}(e)$
on an edge $e \in {\mathcal E}(\widehat K)$ are defined by such
an identification.

%-----------------------------------------------------------
\subsection{Function spaces on the reference element}
%-----------------------------------------------------------
For 
$\nu  \in \{\widehat K, \widehat f, \widehat e, \mathbb{R}^3, \mathbb{R}^2, \mathbb{R}\}$, we denote by 
${\mathcal P}_p(\nu)$ the space of polynomials of (total) degree $p$ 
in $d$ variables, where $d$ is the dimension of the manifold $\nu$. For faces $f \in {\mathcal F}(\widehat K)$ or 
edges $e \in {\mathcal E}(\widehat K)$, we define ${\mathcal P}_p(f)$ or 
${\mathcal P}_p(e)$ by identifying $f$ or $e$ with $\widehat f$ or $\widehat e$
via an affine map. Further polynomial spaces are defined in the following
subsection. 

%-----------------------------------------------------------
\subsubsection{Spaces on the reference tetrahedron and reference triangle}
On $\widehat{K}$ we introduce the classical
N\'{e}d\'{e}lec type I and Raviart-Thomas elements of degree $p\geq0$ (see,
e.g., \cite{Monkbook,hiptmair-acta,nedelec80}):
\begin{align}
W_{p}(\widehat{K}) &  :=\operatorname*{span}\{x^{\alpha
}\,|\,|\alpha|\leq p\},\label{eq:RTp}\\
\mathbf{Q}_{p}(\widehat{K}) &
:=\{\mathbf{p}(\mathbf{x})+\mathbf{x}\times\mathbf{q}(\mathbf{x})\,|\,\mathbf{p},\mathbf{q}%
\in({\mathcal{P}}_{p}(\widehat{K}))^{3}\},\\
\mathbf{V}_{p}(\widehat{K}) &  :=\{\mathbf{p}(\mathbf{x})+{q}(\mathbf{x})\mathbf{x}\,|\,\mathbf{p}%
\in({\mathcal{P}}_{p}(\widehat{K}))^{3},{q}\in{\mathcal{P}}_{p}(\widehat{K}%
)\}.
\end{align}
Recall the exact sequences on the continuous level
(cf., e.g., \cite{costabel-mcintosh10} and Lemma~\ref{lemma:mcintosh})
\begin{align}
\minCDarrowwidth15pt
\begin{CD} \mathbb{R} @> \operatorname*{id} >> H^2(\widehat K) @> \nabla >> \mathbf{H}^1(\widehat K, \operatorname*{\mathbf{curl}}) @> \operatorname*{\mathbf{curl}} >> \mathbf{H}^1(\widehat K, \operatorname*{div}) @> \operatorname*{div} >> H^1(\widehat K) @> \operatorname*{0} >> \{0\} \end{CD}
\end{align}
and on the discrete level (see, e.g., \cite[(57)]{demkowicz08})
\begin{align}
\minCDarrowwidth15pt
\begin{CD} \mathbb{R} @> \operatorname*{id} >> W_{p+1}(\widehat{K}) @> \nabla >> \mathbf{Q}_{p}(\widehat{K}) @> \operatorname*{\mathbf{curl}} >> \mathbf{V}_{p}(\widehat{K}) @> \operatorname*{div} >> W_{p}(\widehat{K}) @> \operatorname*{0} >> \{0\} \end{CD}.
\end{align}
%Using the notation
%\[
%W_{p+1}(\widehat{K}):={\mathcal{P}}_{p+1}(\widehat{K}),\qquad{\mathbf{Q}}%
%_{p}(\widehat{K}):=\boldsymbol{\mathcal{N}}_{p}^{\operatorname*{I}%
%}(\widehat{K}),\qquad{\mathbf{V}}_{p}(\widehat{K}):=\mathbf{RT}_{p}%
%(\widehat{K}),
%\]
In this paper, we present projection operators $\hatPigradcom$, $\hatPicurlcom$, $\hatPidivcom$,
$\widehat{\Pi}_{p}^{L^{2}}$ that enjoy the commuting diagram property
\begin{equation}
\minCDarrowwidth14pt
\begin{CD} \mathbb{R} @> \operatorname*{id} >> H^2(\widehat K) @> \nabla >> \mathbf{H}^1(\widehat K, \operatorname*{\mathbf{curl}}) @> \operatorname*{\mathbf{curl}} >> \mathbf{H}^1(\widehat K, \operatorname*{div}) @> \operatorname*{div} >> H^1(\widehat K) @> \operatorname*{0} >> \{0\}\\ @. @VV \hatPigradcom V @VV \hatPicurlcom V @VV \hatPidivcom V @VV \widehat \Pi^{L^2}_p V \\ \mathbb{R} @> \operatorname*{id} >> W_{p+1}(\widehat K) @> \nabla >> {\mathbf Q}_p(\widehat K) @> \operatorname*{\mathbf{curl}} >> {\mathbf V}_p(\widehat K) @> \operatorname*{div} >> W_p(\widehat K) @> \operatorname*{0} >>\{0\}\end{CD} \label{eq:commuting-diagram-hinten}
\end{equation}
In the two-dimensional setting, the N\'{e}d\'{e}lec type I elements are defined by
\begin{align}
\label{eq:2d-nedelec}
\mathbf{Q}_p(\widehat{f}) := \{\mathbf{p}(\mathbf{x}) + q({\mathbf x}) (y,-x)^T \, | \, \mathbf{p} \in (\mathcal{P}_p(\widehat{f}))^2, q\in \widetilde{\mathcal{P}}_p(\widehat{f})\},
\end{align}
where $\widetilde{\mathcal{P}}_p(\widehat{f})$ denotes the homogeneous polynomials of degree $p$. Here we have shorter exact sequences of the forms
(cf., e.g., \cite{costabel-mcintosh10} and Lemma~\ref{lemma:mcintosh-2d})
\begin{align}
\begin{CD} \mathbb{R} @> \operatorname*{id} >> H^{3/2}(\widehat f) @> \nabla >> \mathbf{H}^{1/2}(\widehat f, \operatorname*{curl}) @> \operatorname*{curl} >> H^{1/2}(\widehat f) @> \operatorname*{0} >> \{0\} \end{CD}
\end{align}
on the continuous level and
\begin{align}
\begin{CD} \mathbb{R} @> \operatorname*{id} >> W_{p+1}(\widehat f) @> \nabla >> \mathbf{Q}_p(\widehat{f}) @> \operatorname*{{curl}} >> W_p(\widehat f) @> \operatorname*{0} >> \{0\} \end{CD}
\end{align}
on the discrete level (see, e.g., \cite[(30)]{demkowicz08}). We present projection operators $\hatPigradcomtwod$, $\hatPicurlcomtwod$, $\widehat{\Pi}_p^{L^2}$ that satisfy the commuting diagram property
\begin{equation}
\begin{CD} \mathbb{R} @> \operatorname*{id} >> H^{3/2}(\widehat f) @> \nabla >> \mathbf{H}^{1/2}(\widehat f, \operatorname*{curl}) @> \operatorname*{curl} >> H^{1/2}(\widehat f) @> \operatorname*{0} >> \{0\}\\ @. @VV \hatPigradcomtwod V @VV \hatPicurlcomtwod V @VV \widehat \Pi^{L^2}_p V \\ \mathbb{R} @> \operatorname*{id} >> W_{p+1}(\widehat f) @> \nabla >> {\mathbf Q}_p(\widehat f) @> \operatorname*{curl} >> W_p(\widehat f) @> \operatorname*{0} >>\{0\}\end{CD} \label{eq:commuting-diagram-2d}
\end{equation}

\subsubsection{Trace spaces on the boundary}
\label{sec:trace_spaces}

%---------------------------------------------------------------------------

We will need the traces of the spaces $W_{p+1}(\widehat{K})$, $\mathbf{Q}_p(\widehat{K})$ and $\mathbf{V}_p(\widehat{K})$ on various parts of the boundary. For faces $f\in{\mathcal{F}}(\widehat{K})$ the corresponding
spaces are defined by trace operations:
\begin{equation*}
W_{p+1}(f):=W_{p+1}(\widehat{K})|_{f},\quad{\mathbf{Q}}_{p}(f):=(\Pi_{\tau
}{\mathbf{Q}}_{p}(\widehat{K}))|_{f},\quad{V}_{p}(f):=\mathbf{V}%
_{p}(\widehat{K})\cdot\mathbf{n}_{f},
\end{equation*}
where $\Pi_{\tau}$ is the tangential component and $\mathbf{n}_{f}$ the outer normal
vector of $f$. These trace spaces are well-known objects: Identifying a face $f$
with the reference triangle $\widehat f$ via an affine bijection, the space $W_{p+1}(f)$
is isomorphic to the space ${\mathcal{P}}_{p+1}({\mathbb{R}}^{2})$ of bivariate
polynomials of (total) degree $p+1$; the space ${\mathbf{Q}}_{p}(f)$ turns out
to be the space of type-I
N\'{e}d\'{e}lec elements on triangles; ${V}_{p}(f)$ is isomorphic to the space
${\mathcal{P}}_{p}({\mathbb{R}}^{2})$. Lowering the dimension even further, we
introduce for each edge $e\in{\mathcal{E}}(\widehat{K})$ the spaces
\[
W_{p+1}(e):=W_{p+1}(\widehat{K})|_{e},\quad{Q}_{p}(e):=\mathbf{Q}%
_{p}(\widehat{K})\cdot\mathbf{t}_{e},
\]
where $\mathbf{t}_{e}$ is the tangential vector of the edge $e$. 
Similar to the case of the faces, the space $W_{p+1}(e)$ can be identified with the
univariate polynomials of degree $p+1$ and ${Q}_{p}(e)$ with the univariate
polynomials of degree $p$.
We will require spaces of functions vanishing on the boundary in the
appropriate sense and set
\begin{align*}
\mathring{W}_{p+1}(\widehat{K})& :=W_{p+1}(\widehat{K})\cap H_{0}^{1}%
(\widehat{K}), & 
\mathring{\mathbf{Q}}_{p}(\widehat{K})&:=\{\mathbf{u}%
\in{\mathbf{Q}}_{p}(\widehat{K})\,|\,\Pi_{\tau}\mathbf{u}=0\},\\
\mathring{\mathbf{V}}_{p}(\widehat{K})&:=\{\mathbf{u}\in{\mathbf{V}}_{p}(\widehat{K}%
)\,|\,\mathbf{n}\cdot\mathbf{u}=0\},
& W^{aver}_p(\widehat K)&:= \{u \in W_{p+1}(\widehat K)\,|\, \int_{\widehat K} u = 0\}.
\end{align*}
%We also need $W^{aver}_p(\widehat K):= \{u \in W_{p+1}(\widehat K)\,|\, \int_{\widehat K} u = 0\}$. 
Corresponding spaces on lower-dimensional manifolds are defined as follows:
\begin{align*}
\mathring{W}_{p+1}(f) &:=W_{p+1}(f)\cap H_{0}^{1}(f),
& \mathring{\mathbf{Q}%
}_{p}(f)& :=\{\mathbf{u}\in{\mathbf{Q}}_{p}(f)\,|\,\Pi_{\tau,f}\mathbf{u}%
=0\},\\
\mathring{V}_{p}(f) &:=\{u \in V_{p}(f)\,|\, \int_{f} u = 0\}.
\end{align*}
Finally, we set for edges $e\in{\mathcal{E}%
}(\widehat{K})$
\begin{align*}
\mathring{W}_{p+1}(e):=W_{p+1}(e)\cap H_{0}^{1}(e),
\qquad\mathring{Q}_{p}(e):=\{u \in Q_{p}(e)\,|\, \int_e u = 0\}. 
\end{align*}
By
%e.g., \cite{demkowicz08} or
\cite[eq.~(4.16)]{hiptmair08} (actually, \cite{hiptmair08} uses the tangential trace
operator $\gamma_{\tau}$ instead of $\Pi_{\tau}$ in the definition of the
spaces ${\mathbf{Q}}_{p}(f)$ and correspondingly identifies the space
${\mathbf{Q}}_{p}(f)$ with a Raviart-Thomas space instead of a N\'{e}d\'{e}lec
space) or \cite{demkowicz-buffa05} we have the following diagrams for $\widehat K$, 
its faces $f\in{\mathcal{F}}(\widehat{K})$, and edges $e\in{\mathcal{E}}(\widehat{K})$
\begin{equation}
\minCDarrowwidth15pt
\begin{CD} \{0\} @>\operatorname{Id} >> \mathring{W}_{p+1}(\widehat K) @> \nabla >> \mathring{\mathbf{Q}}_p(\widehat K) @> \operatorname*{\mathbf{curl}} >> \mathring{\mathbf{V}}_p(\widehat K) @> \operatorname*{div} >> W^{aver}_p(\widehat K)  @> 0 >> \{0\} \\ \{0\} @>\operatorname{Id} >> \mathring{W}_{p+1}(f) @> \nabla_f >> \mathring{\mathbf{Q}}_p(f) @> \operatorname*{curl}_f >> \mathring{V}_p(f)  @> 0 >> \{0\} \\ \{0\} @> \operatorname{Id}>> \mathring{W}_{p+1}(e) @> \nabla_e >> \mathring{Q}_p(e) @> 0 >> \{0\} \end{CD} \label{eq:commuting-diagram-bc}%
\end{equation}
In this diagram (and in what follows), the operators $\nabla_{f}$, $\nabla
_{e}$ represent surface gradients on a face $f$ and tangential differentiation
on an edge $e$, respectively. The operator $\operatorname*{curl}_{f}$ is the
surface curl on face $f$. (Recall that the orientation of $f$ has been fixed to match that of 
$\partial \widehat K$.)
For the reference triangle $\widehat f$ we set
\begin{align*}
\mathring{W}_{p+1}(\widehat{f}) &:= W_{p+1}(\widehat{f}) \cap H_0^1(\widehat{f}), &\mathring{\mathbf{Q}}_p(\widehat f) &:= \{\mathbf{u} \in \mathbf{Q}_p(\widehat f) \, | \, \mathbf{u} \cdot \mathbf{t}_e = 0 \, \forall e\in\mathcal{E}(\widehat{f})\}, \\
\mathring{V}_{p}(\widehat{f}) &:=\{u \in V_{p}(\widehat{f})\,|\, \int_{\widehat{f}} u = 0\}.
\end{align*}
One again looks at shortened sequences, namely, \cite[(33)]{demkowicz08}, 
\begin{equation}
\begin{CD} \{0\} @>\operatorname{Id} >> \mathring{W}_{p+1}(\widehat f) @> \nabla >> \mathring{\mathbf{Q}}_p(\widehat f) @> \operatorname*{{curl}} >> \mathring{V}_p(\widehat f) @> 0 >> \{0\} \\ \{0\} @>\operatorname{Id} >> \mathring{W}_{p+1}(e) @> \nabla_e >> \mathring{Q}_p(e) @> 0 >> \{0\} \end{CD} \label{eq:commuting-diagram-bc-2d}%
\end{equation}
%--------------------------------------------------------
\subsection{Definition of the operators $\hatPigradcom$, $\hatPicurlcom$, $\hatPidivcom$}
\label{sec:def-operators}
%--------------------------------------------------------
The definition of the operators $\hatPigradcom$, $\hatPicurlcom$, $\hatPidivcom$ is very similar
to the definition in \cite{demkowicz-buffa05,demkowicz08}, the difference being that we 
perform all projections using $L^2$-based inner products whereas \cite{demkowicz-buffa05,demkowicz08} 
employs inner products for fractional Sobolev spaces. 

The interpolants may be thought of as being defined 
through a sequence of constrained optimizations in which the value on a 
$j+1$-dimensional subsimplex $S$ 
is determined as the  solution of a minimization problem where the values on the $j$-dimensional boundary 
subsimplices of $S$ have already been fixed. For example, the operator $\hatPigradcom$ starts with
fixing the vertex values, then determines the value on all edges, then on all faces, and finally in the 
interior of $\widehat K$. 

We refer the reader to Section~\ref{sec:commuting} for the proof that
the regularity requirements on the functions $u$, ${\mathbf u}$ 
in the following definitions are indeed sufficient to render 
Definitions~\ref{def:hatPigradcom}--\ref{def:hatPidivcom} meaningful.
%------------------
\subsubsection{The operators in 3D}
\label{sec:3d-operators}
%------------------
In the following definition of $\hatPigradcom$ we can interpret the sequence of conditions 
as first fixing the values in the vertices in (\ref{eq:Pi_grad-d}), 
then on the edges by (\ref{eq:Pi_grad-c}), 
then on the faces by (\ref{eq:Pi_grad-b}), 
and finally in the interior by (\ref{eq:Pi_grad-a}):
\begin{definition}[$\hatPigradcom$]
\label{def:hatPigradcom}
$\hatPigradcom%
:H^{2}(\widehat{K})\rightarrow W_{p+1}(\widehat{K})$ is given by
\begin{subequations}
\label{eq:Pi_grad}
\begin{align}
u(V)-\hatPigradcom u(V)  &  =0\quad \forall
V\in{\mathcal{V}}(\widehat{K}), 
\label{eq:Pi_grad-d}\\
(\nabla_{e}(u-\hatPigradcom u),\nabla
_{e}v)_{L^{2}(e)}  &  =0\quad\forall v\in\mathring{W}_{p+1}(e)\quad\forall
e\in{\mathcal{E}}(\widehat{K}),
\label{eq:Pi_grad-c}\\
(\nabla_{f}(u-\hatPigradcom u),\nabla
_{f}v)_{L^{2}(f)}  &  =0\quad\forall v\in\mathring{W}_{p+1}(f)\quad\forall
f\in{\mathcal{F}}(\widehat{K}),
\label{eq:Pi_grad-b}\\
(\nabla(u-\hatPigradcom u),\nabla v)_{L^{2}%
(\widehat{K})}  &  =0\quad\forall v\in\mathring{W}_{p+1}(\widehat{K}).
\label{eq:Pi_grad-a}
\end{align}
\end{subequations}
\end{definition}

In the following definition of $\hatPicurlcom$ we can interpret the sequence of conditions 
as fixing 
the interpolant first on the edges by (\ref{eq:Pi_curl-f}), (\ref{eq:Pi_curl-e}); 
then on the faces by (\ref{eq:Pi_curl-d}), (\ref{eq:Pi_curl-c}); 
and finally in the interior by (\ref{eq:Pi_curl-b}), (\ref{eq:Pi_curl-a}). 
\begin{definition}[$\hatPicurlcom$]
\label{def:hatPicurlcom}
$\hatPicurlcom:\mathbf{H}%
^{1}(\widehat{K},\operatorname*{\mathbf{curl}})\rightarrow\mathbf{Q}_{p}(\widehat{K})$
is given by
\begin{subequations}
\label{eq:Pi_curl}
\begin{align}
({\mathbf{t}}_{e}\cdot(\mathbf{u}-\hatPicurlcom\mathbf{u}),1)_{L^{2}(e)}  &  =0\quad\forall e\in{\mathcal{E}}%
(\widehat{K}), 
\label{eq:Pi_curl-f}\\
(\mathbf{t}_{e}\cdot(\mathbf{u}-\hatPicurlcom\mathbf{u}),\nabla_{e}v)_{L^{2}(e)}  &  =0\quad\forall v\in\mathring
{W}_{p+1}(e)\quad\forall e\in{\mathcal{E}}(\widehat{K}),
\label{eq:Pi_curl-e}\\
(\Pi_{\tau}(\mathbf{u}-\hatPicurlcom\mathbf{u}),\nabla_{f}v)_{L^{2}(f)}  &  =0\quad\forall v\in\mathring{W}%
_{p+1}(f)\quad\forall f\in{\mathcal{F}}(\widehat{K}),
\label{eq:Pi_curl-d}\\
(\operatorname{curl}_{f}\Pi_{\tau}(\mathbf{u}-\hatPicurlcom\mathbf{u}),\operatorname{curl}_{f}\mathbf{v}%
)_{L^{2}(f)}  &  =0\quad\forall\mathbf{v}\in\mathring{\mathbf{Q}}%
_{p}(f)\quad\forall f\in{\mathcal{F}}(\widehat{K}),
\label{eq:Pi_curl-c}\\
((\mathbf{u}-\hatPicurlcom\mathbf{u}),\nabla
v)_{L^{2}(\widehat{K})}  &  =0\quad\forall v\in\mathring{W}_{p+1}(\widehat{K}),
\label{eq:Pi_curl-b}\\
(\operatorname*{\mathbf{curl}}(\mathbf{u}-\hatPicurlcom \mathbf{u}),\operatorname*{\mathbf{curl}}\mathbf{v})_{L^{2}(\widehat{K})}  &  =0\quad
\forall\mathbf{v}\in\mathring{\mathbf{Q}}_{p}(\widehat{K}).
\label{eq:Pi_curl-a}
\end{align}
\end{subequations}
\end{definition}

In the following definition of $\hatPidivcom$ we can interpret the sequence of conditions 
as fixing 
the interpolant first on the faces by (\ref{eq:Pi_div-d}), (\ref{eq:Pi_div-c}) 
and then in the interior by (\ref{eq:Pi_div-b}), (\ref{eq:Pi_div-a}). 
\begin{definition}[$\hatPidivcom$]
\label{def:hatPidivcom}
$\hatPidivcom:\mathbf{H}%
^{1/2}(\widehat{K},\operatorname*{div})\rightarrow\mathbf{V}_{p}(\widehat{K})$
is given by
\begin{subequations}
\label{eq:Pi_div}
\begin{align}
(\mathbf{n}_{f}\cdot(\mathbf{u}-\hatPidivcom\mathbf{u}),1)_{L^{2}(f)}  &  =0\quad\forall f\in{\mathcal{F}}%
(\widehat{K}), 
\label{eq:Pi_div-d}\\
(\mathbf{n}_{f}\cdot(\mathbf{u}-\hatPidivcom\mathbf{u}),{v})_{L^{2}(f)}  &  =0\quad\forall v\in\mathring{V}_{p}%
(f)\quad\forall f\in{\mathcal{F}}(\widehat{K}),
\label{eq:Pi_div-c}\\
((\mathbf{u}-\hatPidivcom\mathbf{u}%
),\operatorname*{\mathbf{curl}}\mathbf{v})_{L^{2}(\widehat{K})}  &  =0\quad
\forall\mathbf{v}\in\mathring{\mathbf{Q}}_{p}(\widehat{K}),
\label{eq:Pi_div-b}\\
(\operatorname*{div}(\mathbf{u}-\hatPidivcom\mathbf{u}),\operatorname*{div}\mathbf{v})_{L^{2}(\widehat{K})}  &
=0\quad\forall\mathbf{v}\in\mathring{\mathbf{V}}_{p}(\widehat{K}).
\label{eq:Pi_div-a}
\end{align}
\end{subequations}
\end{definition}

\begin{definition}[$\widehat{\Pi}_p^{L^2}$]
$\widehat{\Pi}_p^{L^{2}}:L^{2}(\widehat{K}) \rightarrow
W_{p}(\widehat{K})$ is given by 
\begin{equation}
\label{eq:Pi_L^2}(u - \widehat{\Pi}_p^{L^{2}} u,v)_{L^{2}(\widehat{K})} = 0
\qquad\forall v \in W_{p}(\widehat{K}).
\end{equation}
\end{definition}

%-----------------------------------------------------------------
\subsubsection{The operators in 2D}
%-----------------------------------------------------------------
For the reference triangle $\widehat f$, the 2D-operators are defined 
as follows: 
%We define the projection operators following the lines of Section~\ref{sec:3d-operators}. The operators are then well-defined by the following equations, which can be shown by checking the numbers of conditions the same way as in Section~\ref{sec:3d-operators}. 
\begin{definition}[$\hatPigradcomtwod$]
$\hatPigradcomtwod:H^{3/2}(\widehat{f}) \rightarrow W_{p+1}(\widehat{f})$ is given by
\begin{subequations}
\label{eq:Pi_grad-2d}
\begin{align}
u(V)-\hatPigradcomtwod u(V)  &  =0\quad\forall
V\in{\mathcal{V}}(\widehat{f}), 
\label{eq:Pi_grad-2d-c}\\
(\nabla_{e}(u-\hatPigradcomtwod u),\nabla
_{e}v)_{L^{2}(e)}  &  =0\quad\forall v\in\mathring{W}_{p+1}(e)\quad\forall
e\in{\mathcal{E}}(\widehat{f}),
\label{eq:Pi_grad-2d-b}\\
(\nabla(u-\hatPigradcomtwod u),\nabla v)_{L^{2}%
(\widehat{f})}  &  =0\quad\forall v\in\mathring{W}_{p+1}(\widehat{f}).
\label{eq:Pi_grad-2d-a}
\end{align}
\end{subequations}
\end{definition}

\begin{definition}[$\hatPicurlcomtwod$]
$\hatPicurlcomtwod:\mathbf{H}^{1/2}(\widehat{f},\operatorname*{curl}) \rightarrow \mathbf{Q}_p(\widehat{f})$ is given by
\begin{subequations}
\label{eq:Pi_curl-2d}
\begin{align}
({\mathbf{t}}_{e}\cdot(\mathbf{u}-\hatPicurlcomtwod\mathbf{u}),1)_{L^{2}(e)}  &  =0\quad\forall e\in{\mathcal{E}}%
(\widehat{f}), 
\label{eq:Pi_curl-2d-d}\\
(\mathbf{t}_{e}\cdot(\mathbf{u}-\hatPicurlcomtwod\mathbf{u}),\nabla_{e}v)_{L^{2}(e)}  &  =0\quad\forall v\in\mathring
{W}_{p+1}(e)\quad\forall e\in{\mathcal{E}}(\widehat{f}),
\label{eq:Pi_curl-2d-c}\\
((\mathbf{u}-\hatPicurlcomtwod\mathbf{u}),\nabla
v)_{L^{2}(\widehat{f})}  &  =0\quad\forall v\in\mathring{W}_{p+1}(\widehat{f}),
\label{eq:Pi_curl-2d-b}\\
(\operatorname*{curl}(\mathbf{u}-\hatPicurlcomtwod\mathbf{u}),\operatorname*{curl}\mathbf{v})_{L^{2}(\widehat{f})}  &  =0\quad
\forall\mathbf{v}\in\mathring{\mathbf{Q}}_{p}(\widehat{f}).
\label{eq:Pi_curl-2d-a}
\end{align}
\end{subequations}
\end{definition}

\begin{definition}[$\widehat{\Pi}_p^{L^2}$]
The operator $\widehat{\Pi}_p^{L^2}:L^{2}(\widehat{f}) \rightarrow W_p(\widehat{f})$ is defined by
\begin{align}
(u-\widehat{\Pi}_p^{L^2}u,v)_{L^2(\widehat{f})} = 0 \qquad \forall v\in W_p(\widehat{f}).
\end{align}
\end{definition}

\begin{remark}
[relation of 2D and 3D]
Up to identifying a face $f \in {\mathcal F}(\widehat K)$ with a reference
triangle $\widehat f \subset {\mathbb R}^2$ via an affine \emph{congruence} map, 
the 2D operators $\hatPigradcomtwod$, $\hatPicurlcomtwod$ coincide with the 
restrictions to the face $f$ of $\hatPigradcom$, $\hatPicurlcom$.
\eremk
\end{remark}

The above defined projection-based interpolation operators have the special
feature that the trace of the interpolant on a $j$-dimensional subsimplex of $\widehat K$
is fully determined by the trace on the subsimplex of the function ${\bf u}$. This observation
implies that the above operators admit so-called ``element-by-element''
constructions: 

\begin{remark}
\label{item:thm:projection-based-interpolation-ii}
As described in more detail in \cite[Sec.~{8}]{melenk-sauter18}, the operators 
$\hatPigradcom$, $\hatPicurlcom$, $\hatPidivcom$,
$\widehat{\Pi}_{p}^{L^{2}}$ admit element-by-element constructions. That is,
given a regular triangulation ${\mathcal T}$, one can construct 
operators $\Pi
_{p+1}^{\operatorname*{grad}}$, $\Pi_{p}^{\operatorname*{curl}}$, $\Pi
_{p}^{\operatorname*{div}}$, $\Pi_{p}^{L^{2}}$, which are defined elementwise 
using the operators
$\hatPigradcom$, $\hatPicurlcom$, $\hatPidivcom$,
$\widehat{\Pi}_{p}^{L^{2}}$, that map into the standard spaces of piecewise 
polynomials $W_{p+1}({\mathcal T})$, N\'ed\'elec spaces ${\mathbf Q}_p({\mathcal T})$, or Raviart-Thomas spaces ${\mathbf V}_p({\mathcal T})$. 
These operators are linear projection operators and, for 
smooth $\partial\Omega$, enjoy  the commuting diagram property 
\begin{equation*}
\minCDarrowwidth15pt
\begin{CD} \mathbb{R} @> \operatorname*{id} >> H^2(\Omega) @> \nabla >> \mathbf{H}^1(\Omega, \operatorname*{\mathbf{curl}}) @> \operatorname*{\mathbf{curl}} >> \mathbf{H}^1(\Omega, \operatorname*{div}) @> \operatorname*{div} >> H^1(\Omega) @> \operatorname*{0} >> \{0\}\\ @. @VV \Pi^{\operatorname*{grad}}_{p+1} V @VV \Pi^{\operatorname*{curl}}_p V @VV \Pi^{\operatorname*{div}}_{p} V @VV \Pi^{L^2}_p V \\ \mathbb{R} @> \operatorname*{id} >> W_{p+1}(\mathcal{T}) @> \nabla >> {\mathbf Q}_p(\mathcal{T}) @> \operatorname*{\mathbf{curl}} >> {\mathbf V}_p(\mathcal{T}) @> \operatorname*{div} >> W_p(\mathcal{T}) @> \operatorname*{0} >>\{0\}\end{CD}  
\label{eq:commuting-diagram-global}%
\end{equation*}
This is a direct consequence of Theorem~\ref{thm:projection-based-interpolation}, (\ref{item:thm:projection-based-interpolation-i}) and the fact that the operators are constructed element by element.
An analogous statement about element-by-element constructions holds for meshes in 2d. 
\eremk
\end{remark}
%-------------------------------------------------------------------
\subsection{Main results}

%-------------------------------------------------------------------
%---------------------------------------------------------

The following Theorems~\ref{thm:projection-based-interpolation}, 
\ref{thm:projection-based-interpolation-2d} are our main results. 
\begin{theorem}[Projection-based interpolation in 3D]
\label{thm:projection-based-interpolation} 
Assume that all interior angles of the 4 faces of the reference tetrahedron
$\widehat K$ are smaller than $2\pi/3$. Then  there are constants $C_s$, $C_{s,k}$ 
(depending only on $s$, $k$, and $\widehat K$) such that: 

\begin{enumerate}
[(i)]
\item 
\label{item:thm:projection-based-interpolation-i} 
The operators $\hatPigradcom$, $\hatPicurlcom$, $\hatPidivcom$, $\widehat{\Pi}^{L^2}_p$ are well-defined, 
projections, and the diagram (\ref{eq:commuting-diagram-hinten}) commutes.

\item 
\label{item:thm:projection-based-interpolation-iii} 
For all $\varphi\in H^{2}(\widehat{K})$ there holds
\begin{align*}
\Vert\varphi-\hatPigradcom\varphi\Vert
_{H^{1-s}(\widehat{K})}&\leq C_{s}p^{-(1+s)}\inf_{v\in W_{p+1}(\widehat{K})}\Vert\varphi-v\Vert_{H^{2}(\widehat{K})},\qquad
s\in\lbrack0,1], \\
\Vert \nabla(\varphi-\hatPigradcom \varphi)\Vert_{\widetilde{\mathbf{H}}^{-s}(\widehat{K})}&\leq
C_{s}p^{-(1+s)} \inf_{v \in W_{p+1}(\widehat K)} \Vert \varphi -v\Vert_{H^{2}(\widehat{K})}, \qquad s\in [0,1].
\end{align*}

\item 
\label{item:thm:projection-based-interpolation-iv} 
For all ${\mathbf{u}}\in{\mathbf{H}}^{1}(\widehat{K},\operatorname{\mathbf{curl}})$ there holds
\[
\Vert{\mathbf{u}}-\hatPicurlcom{\mathbf{u}}%
\Vert_{\widetilde{\mathbf{H}}^{-s}(\widehat{K},\operatorname{\mathbf{curl}})}\leq C_s p^{-(1+s)}%
\inf_{\mathbf{v}\in\mathbf{Q}_p(\widehat{K})}\Vert{\mathbf{u}}-{\mathbf{v}}\Vert_{\mathbf{H}^{1}(\widehat{K},\operatorname{\mathbf{curl}})}, \qquad s\in [0,1].
\]

\item 
\label{item:thm:projection-based-interpolation-v} 
For all $k\geq1$ and
all ${\mathbf{u}}\in{\mathbf{H}}^{k}(\widehat{K})$ with 
$\operatorname*{\mathbf{curl}}%
{\mathbf{u}}\in {\mathbf V}_p(\widehat K) \supset
({\mathcal{P}}_{p}(\widehat{K}))^{3}$ there holds
\begin{equation}
\Vert{\mathbf{u}}-\hatPicurlcom{\mathbf{u}}%
\Vert_{\widetilde{\mathbf{H}}^{-s}(\widehat{K},\operatorname{\mathbf{curl}})}\leq C_{s,k}p^{-(k+s)}\Vert{\mathbf{u}}\Vert
_{\mathbf{H}^{k}(\widehat{K})}, \qquad s\in [0,1].
\label{eq:lemma:projection-based-interpolation-approximation-10}%
\end{equation}
If $p\geq k-1$, then $\Vert{\mathbf{u}}\Vert_{{\mathbf{H}}%
^{k}(\widehat{K})}$ can be replaced with the seminorm $|{\mathbf{u}%
}|_{{\mathbf{H}}^{k}(\widehat{K})}$.

\item 
\label{item:thm:projection-based-interpolation-vi} 
For all ${\mathbf{u}%
}\in{\mathbf{H}}^{1/2}(\widehat{K},\operatorname{div})$ there holds
\[
\Vert{\mathbf{u}}-\hatPidivcom{\mathbf{u}}%
\Vert_{\widetilde{\mathbf{H}}^{-s}(\widehat{K},\operatorname{div})}\leq C_s p^{-(1/2+s)}%
\inf_{{\mathbf{v}}\in\boldsymbol{V}_{p}(\widehat{K})}\Vert{\mathbf{u}}-{\mathbf{v}}\Vert_{{\mathbf{H}}%
^{1/2}(\widehat{K},\operatorname{div})}, \qquad s\in [0,1].
\]

\item 
\label{item:thm:projection-based-interpolation-vii} 
For all $k\geq1$ and
all ${\mathbf{u}}\in{\mathbf{H}}^{k}(\widehat{K})$ with $\operatorname*{div}%
{\mathbf{u}}\in {\mathcal{P}}_{p}(\widehat{K})$ there holds
\begin{equation}
\Vert{\mathbf{u}}-\hatPidivcom{\mathbf{u}}%
\Vert_{\widetilde{\mathbf{H}}^{-s}(\widehat{K},\operatorname{div})}\leq C_{s,k}p^{-(k+s)}\Vert{\mathbf{u}}\Vert
_{\mathbf{H}^{k}(\widehat{K})}, \qquad s\in [0,1].
\label{eq:lemma:projection-based-interpolation-approximation-15}%
\end{equation}
If $p\geq k-1$, then $\Vert{\mathbf{u}}\Vert_{{\mathbf{H}}%
^{k}(\widehat{K})}$ can be replaced with the seminorm $|{\mathbf{u}%
}|_{{\mathbf{H}}^{k}(\widehat{K})}$.

\end{enumerate}
\end{theorem}

\begin{remark}[arbitrary reference tetrahedra $\widehat K$]
Inspection of the proof of Theorem~\ref{thm:projection-based-interpolation} shows that 
the following holds, if the condition on the angles 
of the faces of $\widehat K$are dropped:  
Items 
(\ref{item:thm:projection-based-interpolation-i}),
(\ref{item:thm:projection-based-interpolation-vi})
still hold and 
items 
(\ref{item:thm:projection-based-interpolation-iii}), 
(\ref{item:thm:projection-based-interpolation-iv}), 
(\ref{item:thm:projection-based-interpolation-v}), (\ref{item:thm:projection-based-interpolation-vii}) hold for 
$s \in [0, s')$ where $s' > 0$ depends on the angles of the faces of $\widehat K$. 
In particular, the choice $s = 0$ is always admissible. 
\eremk
\end{remark}
\begin{proof}[Proof (of Theorem~\ref{thm:projection-based-interpolation})]
Item (\ref{item:thm:projection-based-interpolation-i}) is shown 
in Lemmas~\ref{lemma:Pi_grad-well-defined}, 
          \ref{lemma:Pi_curl-well-defined}, 
         \ref{lemma:Pi_div-well-defined} and Theorem~\ref{thm:diagram-commutes}.
For 
(\ref{item:thm:projection-based-interpolation-iii})
%see Corollary~\ref{cor:higher-regularity-grad}.
see Theorem~\ref{lemma:demkowicz-grad-3D}.
Item 
(\ref{item:thm:projection-based-interpolation-iv})
%is shown in Corollary~\ref{cor:Picurl-approximation} and 
is shown in Theorem~\ref{thm:duality-again} and 
(\ref{item:thm:projection-based-interpolation-v}) in 
Lemma~\ref{lemma:better-regularity}.
Statement 
(\ref{item:thm:projection-based-interpolation-vi})
is given in Theorem~\ref{thm:duality-again-div}, and 
(\ref{item:thm:projection-based-interpolation-vii}) in
Lemma~\ref{lemma:better-regularity-div}.
\end{proof}

The projection property of the operators $\hatPigradcom$, $\hatPicurlcom$, $\hatPidivcom$ 
together with the best approximation property of Lemma~\ref{lemma:Pgrad1d} implies:
 
\begin{corollary}
\label{cor:thm:projection-based-interpolation}
Assume that all interior angles of the $4$ faces of $\widehat  K$ are smaller than $2\pi/3$.
For $k \ge 1$ and $s\in [0,1]$ there are constants $C_{s,k}$ depending only on $k$, $s$, and the 
choice of $\widehat K$ such that 
\begin{align}
\label{eq:cor:thm:projection-based-interpolation-1}
\|\varphi - \hatPigradcom \varphi\|_{H^{1-s}(\widehat K)} &\leq C_{s,k} p^{-(k+s)}\|\varphi\|_{H^{k+1}(\widehat K)}, \\
\label{eq:cor:thm:projection-based-interpolation-2}
\|{\mathbf u} - \hatPicurlcom {\mathbf u}\|_{\widetilde{\mathbf H}^{-s}(\widehat K,\operatorname{\mathbf{curl}})} 
&\leq C_{s,k} p^{-(k+s)}\|{\mathbf u}\|_{{\mathbf H}^{k}(\widehat K,\operatorname{\mathbf{curl}})}, \\
\label{eq:cor:thm:projection-based-interpolation-3}
\|{\mathbf u} - \hatPidivcom {\mathbf u}\|_{\widetilde{\mathbf H}^{-s}(\widehat K,\operatorname{div})} 
&\leq C_{s,k} p^{-(k+s)}\|{\mathbf u}\|_{{\mathbf H}^{k}(\widehat K,\operatorname{div})}.  
\end{align}
\end{corollary}

\begin{proof}
The estimate (\ref{eq:cor:thm:projection-based-interpolation-1})
follows directly from Theorem~\ref{thm:projection-based-interpolation}, (\ref{item:thm:projection-based-interpolation-iii}) and the best approximation result Lemma~\ref{lemma:Pgrad1d}. 
For the proof of the estimate (\ref{eq:cor:thm:projection-based-interpolation-2})
%we use Theorem~\ref{thm:projection-based-interpolation}, (\ref{item:thm:projection-based-interpolation-iv}) and Lemma~\ref{lemma:Pgrad1d} in the  following way: With 
we write, with Lemma~\ref{lemma:helmholtz-like-decomp}, ${\mathbf u} = \nabla \varphi + {\mathbf z}$ with 
$\|\varphi\|_{H^{k+1}(\widehat K)} \lesssim \|{\mathbf u}\|_{{\mathbf H}^{k}(\widehat K,{\mathbf{curl}})}$ 
and $\|{\mathbf z}\|_{{\mathbf H}^{k+1}(\widehat K)} \lesssim 
\|\operatorname{\mathbf{curl}} {\mathbf u}\|_{{\mathbf H}^{k}(\widehat K)}$. From 
Theorem~\ref{thm:projection-based-interpolation}, (\ref{item:thm:projection-based-interpolation-iv}) 
and Lemma~\ref{lemma:Pgrad1d} we infer 
\begin{align*}
\|{\mathbf u} &- \hatPicurlcom{\mathbf u}\|_{\widetilde {\mathbf H}^{-s}(\widehat K,\operatorname{\mathbf{curl}}) }
 \lesssim p^{-(1+s)}\!\!\!\!\!\! \inf_{\substack{v \in W_{p+1}(\widehat K), \\{\mathbf q} \in {\mathbf Q}_p(\widehat K)} }
\! \|\nabla \varphi + {\mathbf z} - (\nabla v +{\mathbf q})\|_{{\mathbf H}^{1}(\widehat K,\operatorname{\mathbf{curl}})}
\\
& \quad \lesssim p^{-(1+s)} \left[\inf_{v \in W_{p+1}(\widehat K)} \|\varphi - v\|_{H^{2}(\widehat K)} + 
\inf_{{\mathbf q} \in {\mathbf Q}_p(\widehat K)} 
 \|{\mathbf z} - {\mathbf q}\|_{{\mathbf H}^{2}(\widehat K)}\right]
\\
&\stackrel{\text{Lem.~\ref{lemma:Pgrad1d}}}{\lesssim} 
 p^{-(1+s)-(k+1-2)} \left[ \|\varphi\|_{H^{k+1}(\widehat K)} + \|{\mathbf z}\|_{{\mathbf H}^{k+1}(\widehat K)}\right]
\lesssim p^{-(s+k)} \|{\mathbf u}\|_{{\mathbf H}^{k}(\widehat K,\operatorname{\mathbf{curl}})}.
\end{align*}
The bound (\ref{eq:cor:thm:projection-based-interpolation-3}) is shown 
similarly using for 
${\mathbf u} \in {\mathbf H}^k(\widehat K,\operatorname{div})$ the  
decomposition ${\mathbf u} = \operatorname{\mathbf{curl}} {\boldsymbol \varphi} + {\mathbf z}$ 
with 
$\|{\boldsymbol \varphi} \|_{{\mathbf H}^{k+1}(\widehat K)} 
\lesssim \|{\mathbf u}\|_{{\mathbf H}^{k}(\widehat K,\operatorname{div})}$
and $ \| {\mathbf z}\|_{{\mathbf H}^{k+1}(\widehat K)} 
\lesssim \|\operatorname{div} {\mathbf u} \|_{H^{k}(\widehat K)}$ of 
Lemma~\ref{lemma:helmholtz-decomposition-div} and arguing with 
Theorem~\ref{thm:projection-based-interpolation}, (\ref{item:thm:projection-based-interpolation-vi}) 
and Lemma~\ref{lemma:Pgrad1d} to get
\begin{align*}
& \|{\mathbf u} - \hatPidivcom{\mathbf u}\|_{\widetilde {\mathbf H}^{-s}(\widehat K,\operatorname{div})} \\
& \quad \lesssim p^{-(1/2+s)} \inf_{{\mathbf q} \in {\mathbf Q}_p(\widehat K) , {\mathbf v} \in {\mathbf V}_p(\widehat K)}
\|\operatorname{\mathbf{curl}} {\boldsymbol\varphi} + {\mathbf z} - (\operatorname{\mathbf{curl}} {\mathbf q} +{\mathbf v})\|_{{\mathbf H}^{1/2}(\widehat K,\operatorname{div})}
\\
&\quad \lesssim p^{-(1/2+s)} \left[\inf_{{\mathbf q} \in {\mathbf Q}_p(\widehat K)} 
\|{\boldsymbol\varphi} - {\mathbf q}\|_{{\mathbf H}^{3/2}(\widehat K)} + 
 \inf_{{\mathbf v} \in {\mathbf V}_p(\widehat K)} 
 \|{\mathbf z} - {\mathbf v}\|_{{\mathbf H}^{3/2}(\widehat K)}\right] \\
&\quad \stackrel{\text{Lem.~\ref{lemma:Pgrad1d}}}{\lesssim} 
p^{-(s+k)} \|{\mathbf u}\|_{{\mathbf H}^{k}(\widehat K,\operatorname{div})}. 
\qedhere
\end{align*}
\end{proof}

\begin{theorem}[Projection-based interpolation in 2D]
\label{thm:projection-based-interpolation-2d} 
For a reference triangle $\widehat{f} \subset {\mathbb R}^2$ 
with $\widehat s$ given by \eqref{eq:maximal-angle}
there are constants $C_{s,k}$ depending only on 
$s$, $k$, and the  choice of $\widehat f$ such that the following holds: 

\begin{enumerate}
[(i)]

\item 
\label{item:thm:projection-based-interpolation-i-2d} 
The operators $\hatPigradcomtwod$, $\hatPicurlcomtwod$, $\widehat{\Pi}^{L^2}_p$ are well-defined, projections, 
and the diagram (\ref{eq:commuting-diagram-2d}) commutes.

\item 
\label{item:thm:projection-based-interpolation-ii-2d} 
For all $\varphi\in
H^{3/2}(\widehat{f})$ there holds
\begin{align*}
\Vert\varphi-\hatPigradcomtwod\varphi\Vert
_{H^{1-s}(\widehat{f})}& \leq 
C_{s,k}p^{-(1/2+s)}\inf_{v\in W_{p+1}(\widehat{f})}\Vert\varphi-v\Vert_{H^{3/2}(\widehat{f})},\qquad
s\in [0,1], \\
\Vert\varphi-\hatPigradcomtwod\varphi\Vert
_{\widetilde H^{1-s}(\widehat{f})}& \leq 
C_{s,k}p^{-(1/2+s)}\inf_{v\in W_{p+1}(\widehat{f})}\Vert\varphi-v\Vert_{H^{3/2}(\widehat{f})},\qquad
s\in [1,\widehat{s}), \\
\Vert \nabla(\varphi-\hatPigradcomtwod \varphi)\Vert_{\widetilde{\mathbf{H}}^{-s}(\widehat{f})}& \leq
C_{s,k}p^{-(1/2+s)}\inf_{v \in W_{p+1}(\widehat f)} \Vert \varphi -v \Vert_{H^{3/2}(\widehat{f})} ,
\qquad s \in [0,\widehat{s}).
\end{align*}

\item 
\label{item:thm:projection-based-interpolation-iii-2d} 
For all ${\mathbf{u}}\in{\mathbf{H}}^{1/2}(\widehat{f},\operatorname{curl})$ there holds
\[
\Vert{\mathbf{u}}-\hatPicurlcomtwod{\mathbf{u}}%
\Vert_{\widetilde{\mathbf{H}}^{-s}(\widehat{f},\operatorname{curl})}\leq C_{s,k} p^{-(1/2+s)}%
\inf_{\mathbf{v}\in\mathbf{Q}_p(\widehat{f})}\Vert{\mathbf{u}}-{\mathbf{v}}\Vert_{\mathbf{H}^{1/2}(\widehat{f},\operatorname{curl})}, \qquad s\in [0,\widehat{s}).
\]

\item 
\label{item:thm:projection-based-interpolation-iv-2d} 
For all $k\geq1$ and all ${\mathbf{u}}\in{\mathbf{H}}^{k}(\widehat{f})$ with $\operatorname*{curl}{\mathbf{u}}\in{\mathcal{P}}_{p}(\widehat{f})$ there holds
\begin{equation}
\Vert{\mathbf{u}}-\hatPicurlcomtwod{\mathbf{u}}%
\Vert_{\widetilde{\mathbf{H}}^{-s}(\widehat{f},\operatorname{curl})}\leq C_{s,k}p^{-(k+s)}\Vert{\mathbf{u}}\Vert
_{\mathbf{H}^{k}(\widehat{f})}, \qquad s\in [0,\widehat{s}).
\label{eq:lemma:projection-based-interpolation-approximation-2d-10}%
\end{equation}
If $p\geq k-1$, then $\Vert{\mathbf{u}}\Vert_{{\mathbf{H}}%
^{k}(\widehat{f})}$ can be replaced with the seminorm $|{\mathbf{u}%
}|_{{\mathbf{H}}^{k}(\widehat{f})}$.
\end{enumerate}
\end{theorem}

\begin{proof}
The proof of (\ref{item:thm:projection-based-interpolation-i-2d}) follows by arguments
very similar to those given in \cite{demkowicz08} and the arguments 
for the 3D case that are  worked out in Section~\ref{sec:commuting} below. 
Item 
(\ref{item:thm:projection-based-interpolation-ii-2d})
is shown in Theorem~\ref{lemma:demkowicz-grad-2D} and item 
(\ref{item:thm:projection-based-interpolation-iii-2d}) 
%in Corollary~\ref{cor:Picurl-approximation-2d}. 
in Lemma~\ref{lemma:Picurl-face}. 
For 
(\ref{item:thm:projection-based-interpolation-iv-2d}), see
Lemma~\ref{lemma:better-regularity-2d}.
\end{proof}

The following corollary is the two-dimensional analog of 
Corollary~\ref{cor:thm:projection-based-interpolation}:  
\begin{corollary}
\label{cor:thm:projection-based-interpolation-2d}  
For $k \ge 1$,
\begin{align}
\label{eq:cor:thm:projection-based-interpolation-2d-10}  
\|\varphi - \hatPigradcomtwod \varphi\|_{H^{1-s}(\widehat f)} &\leq C_{s,k} p^{-(k+s)}\|\varphi\|_{H^{k+1}(\widehat f)}, \qquad s\in [0,1],\\
\label{eq:cor:thm:projection-based-interpolation-2d-12}  
\|\varphi - \hatPigradcomtwod \varphi\|_{\widetilde H^{1-s}(\widehat f)} &\leq C_{s,k} p^{-(k+s)}\|\varphi\|_{H^{k+1}(\widehat f)}, \qquad s\in [1,\widehat{s}),\\
\label{eq:cor:thm:projection-based-interpolation-2d-20}  
\|{\mathbf u} - \hatPicurlcomtwod {\mathbf u}\|_{\widetilde{\mathbf H}^{-s}(\widehat f,\operatorname{curl})} 
&\leq C_{s,k} p^{-(k+s)}\|{\mathbf u}\|_{{\mathbf H}^{k}(\widehat f,\operatorname{curl})}, \qquad s\in [0,\widehat{s}).
\end{align}
\end{corollary}
\begin{proof}
The proof follows as in 
Corollary~\ref{cor:thm:projection-based-interpolation}, 
relying on Lemma~\ref{lemma:helmholtz-like-decomp-2d}
for %the proof of 
(\ref{eq:cor:thm:projection-based-interpolation-2d-20}). 
\end{proof}
%-------------------------------------------------

%-------------------------------------------------

%-----------------------------------
\section{Well-definedness of the projection operators and commuting diagram property}
\label{sec:commuting}
%--------------------------------------------
We show that the operators introduced in 
Section~\ref{sec:def-operators} are well-defined. The arguments are mostly
well-established and included for completeness' sake.

\begin{lemma} 
\label{lemma:Pi_grad-well-defined}
For $u \in H^2(\widehat K)$ there holds 
$u|_e \in H^1(e)$ for each $e \in {\mathcal E}(\widehat K)$
and $\|u\|_{H^1(e)} \lesssim \|u\|_{H^2(\widehat K)}$. 
The operator $\hatPigradcom$ is well-defined. 
\end{lemma}
\begin{proof}
We claim that for $u \in H^2(\widehat K)$ one has $u|_e \in H^1(e)$ for 
each edge $e \in {\mathcal E}(\widehat K)$. 
This follows from the trace theorem: 
a two-fold trace estimate (from $\widehat K$ to the faces and then 
from the faces to the edges) shows that the trace operator maps
$H^{2+\varepsilon}(\widehat K)\rightarrow H^{1+\varepsilon}(e)$ 
for sufficiently small $\varepsilon > 0$ and $\varepsilon < 0$. 
Interpolation then asserts $H^2(\widehat K) \rightarrow H^1(e)$. 
For the well-definedness of $\hatPigradcom$, we first see that 
(\ref{eq:Pi_grad}) represents a square linear system: 
We have 
$\operatorname*{dim}W_{p+1}(\widehat{K})   =\frac{1}{6}(p+4)(p+3)(p+2)$
and the number of conditions is
$\frac{1}{6}p(p-1)(p-2)+4\frac{(p-1)p}
{2}+6p+4=\operatorname*{dim}W_{p+1}(\widehat{K}).$
Next, we show uniqueness. 
For $u = 0$ condition 
(\ref{eq:Pi_grad-d}) shows $\hatPigradcom u(V) = 0$ for all vertices $V \in {\mathcal V}(\widehat K)$. 
The conditions (\ref{eq:Pi_grad-c}) then imply that $\hatPigradcom u = 0$ on all edges of $\widehat K$; 
next (\ref{eq:Pi_grad-b}) leads to $\hatPigradcom u = 0$ vanishing on all faces of $\widehat K$ and finally 
(\ref{eq:Pi_grad-a}) shows $\hatPigradcom u = 0$. Thus, $\hatPigradcom$ is well-defined.
\end{proof}

\begin{lemma} 
\label{lemma:Pi_curl-well-defined}
For ${\mathbf u} \in {\mathbf H}^1(\widehat K,\operatorname{\mathbf{curl}})$ 
there holds ${\mathbf u} \cdot {\mathbf t}_e \in L^2(e)$ for each edge 
$e \in {\mathcal E}(\widehat K)$ and $\|{\mathbf u} \cdot {\mathbf t}_e\|_{L^2(e)} 
\lesssim \|{\mathbf u}\|_{{\mathbf H}^1(\widehat K,\operatorname{\mathbf{curl}})}$.  
The operator $\hatPicurlcom$ is well-defined. 
\end{lemma}
\begin{proof}
For ${\mathbf u} \in {\mathbf H}^1(\widehat K,\operatorname{\mathbf{curl}})$ 
the trace theorem gives, for each face $f$, 
$\Pi_\tau {\mathbf u} \in {\mathbf H}^{1/2}(f,\operatorname{curl}_f)$. 
The argument at the outset
of the proof of Lemma~\ref{lemma:Picurl-edge} shows 
${\mathbf t}_e \cdot {\mathbf  u} \in L^2(e)$. To see that 
$\hatPicurlcom$ is well-defined, we assert that (\ref{eq:Pi_curl}) represents
a square linear system and that ${\mathbf u} = 0$ implies $\hatPicurlcom {\mathbf u} = 0$. In order to count the number of conditions in (\ref{eq:Pi_curl}) 
we introduce the notation
%\begin{align*}
$\displaystyle 
\operatorname{ker}\operatorname{\mathbf{curl}}=\{\mathbf{q}\in\mathring{\mathbf{Q}}_p(\widehat{K}):\operatorname{\mathbf{curl}}\mathbf{q}=\mathbf{0}\},
$
%\end{align*}
and observe 
in view of the exactness of the sequence (\ref{eq:commuting-diagram-bc}) that
\begin{align*}
\operatorname*{dim}\mathring{\mathbf{Q}}_{p}(\widehat{K})&=\operatorname*{dim}%
\operatorname*{\mathbf{curl}}\mathring{\mathbf{Q}}_{p}(\widehat{K})\!+\!\operatorname*{dim}%
\operatorname*{ker}\operatorname*{\mathbf{curl}}=\operatorname*{dim}\operatorname*{\mathbf{curl}}\mathring{\mathbf{Q}}%
_{p}(\widehat{K})+\operatorname*{dim}\nabla\mathring{W}_{p+1}(\widehat{K}).
\end{align*}
Therefore, 
the number of conditions in 
(\ref{eq:Pi_curl-b}), (\ref{eq:Pi_curl-a})  equals $\operatorname*{dim}%
\mathring{\mathbf{Q}}_{p}(\widehat{K})$.
Analogously, we argue with the exactness of the second sequence in
(\ref{eq:commuting-diagram-bc}) that for each face 
$f \in {\mathcal F}(\widehat K)$ 
the number of conditions in (\ref{eq:Pi_curl-d}), (\ref{eq:Pi_curl-c}) 
equals $\operatorname*{dim}%
\mathring{\mathbf{Q}}_{p}(f)$.
Finally, we check that for each edge $e \in {\mathcal E}(\widehat K)$
the number of conditions in (\ref{eq:Pi_curl-e}) is $p$
and that the number of conditions in (\ref{eq:Pi_curl-f}) equals $6$.
In total, the number of conditions in (\ref{eq:Pi_curl}) coincides with
$\operatorname*{dim}\mathbf{Q}_{p}(\widehat{K})$. We conclude that (\ref{eq:Pi_curl}) represents a  
square system of equations. As in the case of Lemma~\ref{lemma:Pi_grad-well-defined}, 
see that ${\mathbf u} = 0$ implies $\hatPicurlcom {\mathbf u} = 0$ in the following way: 
(\ref{eq:Pi_curl-f}), (\ref{eq:Pi_curl-e}) imply that the tangential component of 
$\hatPicurlcom {\mathbf u}$ vanishes on all edges of $\widehat K$. From that, 
(\ref{eq:Pi_curl-d}), (\ref{eq:Pi_curl-c}) together with the exact sequence property 
(\ref{eq:commuting-diagram-bc-2d}) gives that the tangential component 
$\Pi_\tau \hatPicurlcom {\mathbf u}$ vanishes on all faces of $\widehat K$. Finally, 
(\ref{eq:Pi_curl-b}), (\ref{eq:Pi_curl-a}) together with again the exact sequence property 
(\ref{eq:commuting-diagram-bc}) yields $\hatPicurlcom {\mathbf u} = 0$. 
\end{proof}

\begin{lemma} 
\label{lemma:Pi_div-well-defined}
For ${\mathbf u} \in {\mathbf H}^{1/2}(\widehat K,\operatorname{div})$ 
there holds ${\mathbf u} \cdot {\mathbf n}_f \in L^2(f)$ for each 
face $f \in {\mathcal F}(\widehat K)$  and 
$\|{\mathbf u}\cdot {\mathbf n}_f\|_{L^2(f)} \lesssim \|{\mathbf u}\|_{{\mathbf H}^{1/2}(\widehat K,\operatorname{div})}$. 
The operator $\hatPidivcom$ is well-defined. 
\end{lemma}
\begin{proof}
We first show that for ${\mathbf u} \in {\mathbf H}^{1/2}(\widehat K,\operatorname{div})$ 
the normal trace ${\mathbf n}_f \cdot {\mathbf u} \in L^2(f)$ 
for each face $f$. To that end, we write with the aid of 
Lemma~\ref{lemma:helmholtz-decomposition-div} 
${\mathbf u}=\operatorname{\mathbf{curl}} {\boldsymbol \varphi} + {\mathbf z}$ 
with ${\boldsymbol \varphi}$, ${\mathbf z} \in {\mathbf H}^{3/2}(\widehat K)$. 
We have ${\mathbf n}_f \cdot {\mathbf z} \in {\mathbf H}^1(f)$. Noting 
${\boldsymbol \varphi}|_f  \in {\mathbf H}^1(f)$ 
(cf.\ also the proof of Lemma~\ref{lemma:traces-2d})
and 
$({\mathbf n}_f \cdot \operatorname{\mathbf{curl}} {\boldsymbol \varphi})|_f  
= \operatorname{curl}_f (\Pi_\tau {\boldsymbol \varphi})|_f$, we conclude that 
$({\mathbf n}_f \cdot \operatorname{\mathbf{curl}} {\boldsymbol \varphi})|_f 
\in L^2(f)$.  To see that $\hatPidivcom$ is well-defined, we 
check the number of conditions in (\ref{eq:Pi_div}) and show uniqueness. 
In view of the
exactness of the sequence in (\ref{eq:commuting-diagram-bc}) we get, 
using the notation
%\begin{align*}
$\displaystyle \operatorname{ker}\operatorname{div}=\{\mathbf{v}\in\mathring{\mathbf{V}}_p(\widehat{K}):\operatorname{div}\mathbf{v}=0\},
$
%\end{align*}
the equality
\begin{align*}
\operatorname*{dim}\mathring{\mathbf{V}}_{p}(\widehat{K})=\operatorname*{dim}%
\operatorname*{div}\mathring{\mathbf{V}}_{p}(\widehat{K})+\operatorname*{dim}%
\operatorname*{ker}\operatorname*{div}=\operatorname*{dim}\operatorname*{div}\mathring{\mathbf{V}}%
_{p}(\widehat{K})+\operatorname*{dim}\operatorname*{\mathbf{curl}}\mathring{\mathbf{Q}}_{p}(\widehat{K}),
\end{align*}
so that the number of conditions in
(\ref{eq:Pi_div-b}), (\ref{eq:Pi_div-a})  equals 
$\operatorname*{dim}\mathring{\mathbf{V}}_{p}(\widehat{K})$.
Furthermore, the number of conditions in 
(\ref{eq:Pi_div-d}), (\ref{eq:Pi_div-c}) is $4\operatorname*{dim}W_{p}(f)$
so that 
\[
\operatorname*{dim}\mathring{\mathbf{V}}_{p}(\widehat{K})+4\operatorname*{dim}%
W_{p}(f)=\frac{1}{2}(p+2)(p+1)p+4\frac{(p+1)(p+2)}{2}=\operatorname*{dim}%
\mathbf{V}_{p}(\widehat{K}).
\]
To see that ${\mathbf u}= 0$ implies $\hatPidivcom {\mathbf u} = 0$, we note
that conditions (\ref{eq:Pi_div-d}), (\ref{eq:Pi_div-c}) produce 
${\mathbf n}_f \cdot \hatPidivcom {\mathbf u} = 0$ for all faces $f \in {\mathcal F}(\widehat K)$. 
The exact sequence property (\ref{eq:commuting-diagram-bc}) and 
conditions (\ref{eq:Pi_div-b}), (\ref{eq:Pi_div-a}) then imply $\hatPidivcom {\mathbf u} = 0$. 
\end{proof}

%---------
\begin{lemma}
\label{lemma:traces-2d}
For $u\in H^{3/2}(\widehat f)$ there holds $\nabla_e u \in L^2(e)$ for 
each edge $e \in {\mathcal E}(\widehat f)$ and 
$\|\nabla_e u\|_{L^2(e)} \lesssim \|u\|_{H^{3/2}(\widehat f)}$. 
For ${\mathbf u} \in {\mathbf H}^{1/2}(\widehat f,\operatorname*{curl})$ there holds 
${\mathbf u} \cdot {\mathbf t}_e \in L^2(e)$ for 
each edge $e \in {\mathcal E}(\widehat f)$ 
and 
$\|{\mathbf u} \cdot {\mathbf t}_e\|_{L^2(e)} \lesssim \|{\mathbf u}\|_{{\mathbf H}^{1/2}(\widehat f,\operatorname*{curl})}$. 
The operators $\hatPigradcomtwod$ and $\hatPicurlcomtwod$ are well-defined. 
\end{lemma}
\begin{proof} We only show the assertions for the traces. For $\varepsilon > 0$ sufficiently small, 
the trace theorem implies $u|_e \in H^{1+\varepsilon}(e)$ for $u \in H^{3/2+\varepsilon}(\widehat f)$
and $u|_e \in H^{1-\varepsilon}(e)$ for $u \in H^{3/2-\varepsilon}(\widehat f)$. By interpolation, 
$u|_e \in H^1(e)$ for $u \in H^{3/2}(\widehat f)$. 
Any ${\mathbf u} \in {\mathbf H}^{1/2}(\widehat f,\operatorname*{curl})$ can be written 
as ${\mathbf u} = \nabla \varphi + {\mathbf z}$ with $\varphi \in H^{3/2}(\widehat f)$ and 
${\mathbf z} \in {\mathbf H}^{3/2}(\widehat f)$ by Lemma~\ref{lemma:helmholtz-like-decomp-2d}. Hence, 
${\mathbf u} \cdot {\mathbf t}_e = \nabla_e \varphi|_e + {\mathbf z}|_e \cdot {\mathbf t}_e \in L^2(e)$. 
\end{proof}
%---------
\begin{theorem}
\label{thm:diagram-commutes} The diagrams
\eqref{eq:commuting-diagram-hinten} and \eqref{eq:commuting-diagram-2d} commute.
\end{theorem}

\begin{proof}
We will only show
the arguments for the 3D case \eqref{eq:commuting-diagram-hinten} as the 2D case 
is shown with similar arguments.
The proof follows the arguments given in \cite{demkowicz08}.  

\emph{Proof of 
$ \hatPicurlcom\nabla= \nabla\hatPigradcom $:} 
Let $\mathbf{u}=\nabla\varphi$ for
some $\varphi\in H^{2}(\widehat{K})$. We first claim 
\begin{equation}
\hatPicurlcom\nabla\varphi=\nabla\varphi_{p}%
\quad\text{for some }\varphi_{p}\in W_{p+1}(\widehat{K}).
\label{eq:thm:diagram-commutes-10}%
\end{equation}
For each edge $e$ with endpoints $V_{1}$, $V_{2}$, we compute $\int%
_{e}\mathbf{u}\cdot\mathbf{t}_{e}=\varphi(V_{1})-\varphi(V_{2})$, so that we
get from (\ref{eq:Pi_curl-f}) for each face $f$ 
(and orienting the tangential vectors 
of the edges $e \in {\mathcal{E}}(f)$ so that $f$ is always ``on the left'')
\begin{equation}
\int_{\partial f}\Pi_{\tau,f}\hatPicurlcom\mathbf{u}=\sum_{e\subset\partial f}\int_{e}\mathbf{u}\cdot\mathbf{t}%
_{e}=0. \label{eq:thm:diagram-commutes-20}%
\end{equation}
We conclude with integration by parts in view of $\operatorname{curl}_{f}%
\Pi_{\tau}\mathbf{u}=\operatorname*{curl}_{f}\Pi_{\tau}\nabla\varphi=0$
\begin{equation}
\int_{f}\operatorname{curl}_{f}\Pi_{\tau}\hatPicurlcom\mathbf{u}=\int_{\partial f}\Pi_{\tau,f}\hatPicurlcom\mathbf{u}%
\overset{\text{(\ref{eq:thm:diagram-commutes-20})}}{=}0.
\label{eq:thm:diagram-commutes-30}%
\end{equation}
Furthermore, the exact sequence property (\ref{eq:commuting-diagram-bc}) gives
us $\operatorname{curl}_{f}\mathring{\mathbf{Q}}_{p}(f)=\mathring{V}_{p}(f)$
so that (\ref{eq:Pi_curl-c}) leads to 
\begin{equation}
\operatorname{curl}_{f}\Pi_{\tau}\hatPicurlcom\mathbf{u}=\operatorname*{const}. \label{eq:thm:diagram-commutes-40}%
\end{equation}
(\ref{eq:thm:diagram-commutes-30}), (\ref{eq:thm:diagram-commutes-40})
together imply $\operatorname*{curl}_{f}\Pi_{\tau}\hatPicurlcom\mathbf{u}=0$ so that on each face $(\Pi_{\tau
}\hatPicurlcom\mathbf{u})|_{f}$ is a gradient of
a polynomial: $(\Pi_{\tau}\hatPicurlcom\mathbf{u})|_{f}=\nabla\varphi_{p,f}$ for some $\varphi_{p,f}\in W_{p+1}(f)$
for each face $f\in{\mathcal{F}}(\widehat{K})$.

We claim that this piecewise polynomial can be chosen to be continuous on
$\partial\widehat{K}$. Fix a vertex $V\in{\mathcal{V}}(\widehat{K})$. By
fixing the constant of the polynomials $\varphi_{p,f}$ we may assume that
$\varphi_{p,f}(V)=0$ for each face $f$ that has $V$ as a vertex. From
(\ref{eq:Pi_curl-f}), (\ref{eq:Pi_curl-e}) we conclude that $\varphi_{p,f}$ is
continuous across all edges $e$ that have $V$ as an endpoint. Hence, the
piecewise polynomial $\varphi_{p}$ given by $\varphi_{p}|_{f}=\varphi_{p,f}$
is continuous in all vertices of $\widehat{K}$. We conclude that $\varphi_{p}$
is continuous on $\partial\widehat{K}$. This continuous, piecewise polynomial
$\varphi_{p}$ has, by \cite{munoz-sola97,demkowicz-gopalakrishnan-schoeberl-I}, a polynomial lifting to
$\widehat{K}$ (again denoted $\varphi_{p}\in W_{p+1}(\widehat{K})$). We note
$
\hatPicurlcom\mathbf{u}-\nabla\varphi_{p}%
\in\mathring{\mathbf{Q}}_{p}(\widehat{K})
$
so that (\ref{eq:Pi_curl-a}) with test function $\mathbf{v}=\hatPicurlcom\mathbf{u}-\nabla\varphi_{p}\in\mathring{\mathbf{Q}}_{p}(\widehat{K})$ implies%
\begin{equation}
\operatorname*{\mathbf{curl}}\hatPicurlcom\mathbf{u}=0.
\label{eq:thm:diagram-commutes-50}%
\end{equation}
Since the second line of (\ref{eq:commuting-diagram-hinten}) expresses an
exact sequence property, we conclude that (\ref{eq:thm:diagram-commutes-10}) holds.

We now show that $\hatPicurlcom\nabla
\varphi=\nabla\hatPigradcom\varphi$. From
(\ref{eq:thm:diagram-commutes-10}) we get $\hatPicurlcom\nabla\varphi=\nabla\varphi_{p}$ for some
$\varphi_{p}\in W_{p+1}(\widehat{K})$. We fix the constant in the function
$\varphi_{p}$ by stipulating $\varphi_{p}(V)=\varphi(V)$ for one selected
vertex $V\in{\mathcal{V}}(\widehat{K})$. From (\ref{eq:Pi_curl-f}), we then
get $\varphi(V^{\prime})=\varphi_{p}(V^{\prime})$ for all vertices $V^{\prime
}\in{\mathcal{V}}(\widehat{K})$. Next, (\ref{eq:Pi_grad-c}) and
(\ref{eq:Pi_curl-e}) imply $\hatPigradcom%
\varphi=\varphi_{p}$ on all edges $e\in{\mathcal{E}}(\widehat{K})$. Comparing
(\ref{eq:Pi_grad-b}) and (\ref{eq:Pi_curl-d}) reveals $\nabla_{f}\hatPigradcom\varphi=\Pi_{\tau}\hatPicurlcom\nabla\varphi$ on each face $f\in{\mathcal{F}%
}(\widehat{K})$. Finally, comparing (\ref{eq:Pi_grad-a}) with
(\ref{eq:Pi_curl-b}) shows $\hatPicurlcom%
\nabla\varphi=\nabla\hatPigradcom\varphi$.

\medskip\emph{Proof of $\operatorname*{\mathbf{curl}}\hatPicurlcom=\hatPidivcom\operatorname*{\mathbf{curl}}$}: 
First, we show
\begin{equation}
\operatorname*{div}\hatPidivcom%
\operatorname*{\mathbf{curl}}\mathbf{u}=0. \label{eq:thm:commuting-diagram-90}%
\end{equation}
To see this, we note from the second line of
(\ref{eq:commuting-diagram-hinten}) that $\operatorname*{div}\hatPidivcom\operatorname*{\mathbf{curl}}{\mathbf{u}}\in W_{p}%
(\widehat{K})$. Additionally,
\begin{equation}
\int_{\partial\widehat{K}}{\mathbf{n}}\cdot\hatPidivcom\operatorname*{\mathbf{curl}}\mathbf{u}%
\overset{\text{(\ref{eq:Pi_div-d})}}{=}\int_{\partial\widehat{K}}{\mathbf{n}%
}\cdot\operatorname*{\mathbf{curl}}\mathbf{u}=\int_{\widehat{K}}\operatorname*{div}%
\operatorname*{\mathbf{curl}}\mathbf{u}=0. \label{eq:thm:commuting-diagram-110}%
\end{equation}
Finally, the exact sequence property of the first line of diagram
(\ref{eq:commuting-diagram-bc}) informs us that $\operatorname*{div}:%
\mathring{\mathbf{V}}_{p}(\widehat{K})\rightarrow W^{aver}_p(\widehat K)$ is surjective. 
%\mathring{\mathbf{V}}_{p}(\widehat{K})\rightarrow W_{p}(\widehat{K}%
%)/{\mathbb{R}}$ is surjective. 
Hence, we get from (\ref{eq:Pi_div-a}) that
$\operatorname*{div}\hatPidivcom\operatorname*{\mathbf{curl}}\mathbf{u}=0$, 
i.e., indeed the claim (\ref{eq:thm:commuting-diagram-90}) holds. 
Next, 
(\ref{eq:thm:commuting-diagram-90}) and the exact sequence property of
(\ref{eq:commuting-diagram-hinten}) imply, 
for some $\mathbf{u}_{p}\in\mathbf{Q}_{p}(\widehat{K})$,
\begin{equation}
\hatPidivcom\operatorname*{\mathbf{curl}}\mathbf{u}%
=\operatorname*{\mathbf{curl}}\mathbf{u}_{p}. \label{eq:thm:commuting-diagram-100}%
\end{equation}
We next claim $\hatPidivcom\operatorname*{\mathbf{curl}}%
\mathbf{u}=\operatorname*{\mathbf{curl}}\hatPicurlcom\mathbf{u}$. 
To that end, we ascertain that $\operatorname*{\mathbf{curl}}%
\hatPicurlcom\mathbf{u}\in{\mathbf{V}}%
_{p}(\widehat{K})$ satisfies the equations (\ref{eq:Pi_div}) for
$\hatPidivcom\operatorname*{\mathbf{curl}}\mathbf{u}$, i.e., that 
\begin{subequations}
\label{eq:Pi_div-10}
\begin{align}
\label{eq:Pi_div-d-10}%
(\mathbf{n}_{f}\cdot(\operatorname*{\mathbf{curl}}\mathbf{u}-\operatorname*{\mathbf{curl}}%
\hatPicurlcom\mathbf{u}),1)_{L^{2}(f)}  &
=0\quad\forall f\in{\mathcal{F}}(\widehat{K}), \\
(\mathbf{n}_{f}\cdot(\operatorname*{\mathbf{curl}}\mathbf{u}-\operatorname*{\mathbf{curl}}%
\hatPicurlcom\mathbf{u}),{v})_{L^{2}(f)}  &
=0\quad\forall v\in\mathring{V}_{p}(f)\quad\forall f\in{\mathcal{F}%
}(\widehat{K}),
\label{eq:Pi_div-c-10}\\
(\operatorname*{\mathbf{curl}}\mathbf{u}-\operatorname*{\mathbf{curl}}\hatPicurlcom\mathbf{u},\operatorname*{\mathbf{curl}}\mathbf{v})_{L^{2}%
(\widehat{K})}  &  =0\quad\forall\mathbf{v}\in\mathring{\mathbf{Q}}%
_{p}(\widehat{K}),
\label{eq:Pi_div-b-10}\\
(\operatorname*{div}(\operatorname*{\mathbf{curl}}\mathbf{u}-\operatorname*{\mathbf{curl}}%
\hatPicurlcom\mathbf{u}),\operatorname*{div}%
\mathbf{v})_{L^{2}(\widehat{K})}  &  =0\quad\forall\mathbf{v}\in
\mathring{\mathbf{Q}}_{p}(\widehat{K}).
\label{eq:Pi_div-a-10}
\end{align}
(\ref{eq:Pi_div-a-10}) is obviously satisfied and (\ref{eq:Pi_div-b-10}) is a
rephrasing of (\ref{eq:Pi_curl-a}). Noting $\mathbf{n}_{f}\cdot
\operatorname*{\mathbf{curl}}=\operatorname{curl}_{f}\Pi_{\tau}$, we rephrase
(\ref{eq:Pi_div-c-10}) as
\end{subequations}
\begin{equation}
(\operatorname{curl}_{f}\Pi_{\tau}(\mathbf{u}-\hatPicurlcom\mathbf{u}),v)_{L^{2}(f)}=0\quad\forall v\in
\mathring{V}_{p}(f)
\quad \forall f \in {\mathcal F}(\widehat{K}).
\label{eq:Pi_div-b-10-10}%
\end{equation}
In view of the exact sequence property of (\ref{eq:commuting-diagram-bc}), we 
have $\mathring{V}_{p}(f) = \operatorname{curl}%
_{f}\mathring{\mathbf{Q}}_{p}(f)$ so that (\ref{eq:Pi_curl-c}) implies
(\ref{eq:Pi_div-b-10-10}). Finally, (\ref{eq:Pi_div-d-10}) is seen by an 
integration by parts:  
\[
(\operatorname{curl}_{f}\Pi_{\tau}(\mathbf{u}-\hatPicurlcom\mathbf{u}),1)_{L^{2}(f)}=\sum_{e\subset\partial
f}(\Pi_{\tau}(\mathbf{u}-\hatPicurlcom%
\mathbf{u}),\mathbf{t}_{e})_{L^{2}(e)}\overset{\text{(\ref{eq:Pi_curl-f})}%
}{=}0.
\]

\emph{Proof of $\operatorname*{div}\hatPidivcom=\widehat{\Pi}_{p}^{L^{2}}\operatorname*{div}$:} Again, this 
follows from the exact sequence property (\ref{eq:commuting-diagram-bc}). We check that
$\operatorname*{div}\hatPidivcom\mathbf{u}$
satisfies (\ref{eq:Pi_L^2}). To that end, we note
\begin{equation*}
%(\operatorname*{div}{\mathbf{u}}-v,1)_{L^{2}(\widehat{K})}%
%=
(\operatorname*{div}{\mathbf{u}}-\operatorname*{div}\hatPidivcom{\mathbf{u}},1)_{L^{2}(\widehat{K})}%
=\int_{\partial\widehat{K}}\mathbf{n}\cdot({\mathbf{u}}-\hatPidivcom{\mathbf{u}})\overset{\text{(\ref{eq:Pi_div-d}%
)}}{=}0.
\end{equation*}
Furthermore, the exact sequence property (\ref{eq:commuting-diagram-bc}) implies 
that every $w\in W^{aver}_{p}(\widehat{K})$ has the form
$w=\operatorname*{div}{\mathbf{w}}$ for some $\mathbf{w}\in\mathring{\mathbf{V}}_{p}(\widehat{K})$. 
Hence,  for any $w \in W^{aver}_p(\widehat{K})$
%for every $w\in W^{aver}_{p}(\widehat{K})$
\[
(\operatorname*{div}{\mathbf{u}}-\operatorname*{div}\hatPidivcom\mathbf{u},w)_{L^{2}(\widehat{K})}%
=(\operatorname*{div}{\mathbf{u}}-\operatorname*{div}\hatPidivcom\mathbf{u},\operatorname*{div}\mathbf{w}%
)_{L^{2}(\widehat{K})}\overset{\text{(\ref{eq:Pi_div-a})}}{=}0. 
\]
for all $w \in W^{aver}_p(\widehat{K})$. 
This proves (\ref{eq:commuting-diagram-hinten}) 
in the three-dimensional setting. 
\end{proof}
%end neu
%-----------------------------------

\section{Stability of the projection operators in one space dimension}
%\label{sec:stability-of-widehatPigrad}

In one dimension, the following result holds true.

\begin{lemma}
\label{lemma:demkowicz-grad-1D} Let $\widehat{e}=(-1,1)$. 
Let $\widehat{\Pi}_p^{\operatorname*{grad},1d}:H^{1}(\widehat{e})\rightarrow{\mathcal{P}}_{p}(\widehat e)$ 
be defined by
\begin{subequations}
\label{eq:lemma:demkowicz-grad-1D-10}
\begin{align}
\label{eq:lemma:demkowicz-grad-1D-10-a}
((u-\widehat{\Pi}_p^{\operatorname*{grad},1d}u)^{\prime},v^{\prime})_{L^{2}(\widehat{e})}  &
=0\qquad\forall v\in{\mathcal{P}_{p}(\widehat e)}\cap H_{0}^{1}(\widehat{e}),\\
\label{eq:lemma:demkowicz-grad-1D-10-b}
u(\pm1)  &  =(\widehat{\Pi}_p^{\operatorname*{grad},1d}u)(\pm1).
\end{align}
\end{subequations}
Then for every $s\ge 0$ there is a constant $C_{s}>0$ such that 
\begin{subequations}
\begin{align}
\label{eq:lemma:demkowicz-grad-1D-20-a}
\Vert u-\widehat{\Pi}_p^{\operatorname*{grad},1d}u\Vert_{H^{1-s}(\widehat{e})}
&\leq C_{s} p^{-s} \inf_{v \in {\mathcal P}_p(\widehat e)} \Vert u - v\Vert_{H^{1}(\widehat{e})}, \qquad 
\mbox{if $s \in [0,1]$}, \\
\label{eq:lemma:demkowicz-grad-1D-20-b}
\Vert u-\widehat{\Pi}_p^{\operatorname*{grad},1d}u\Vert_{\widetilde H^{1-s}(\widehat{e})}
&\leq C_{s} p^{-s} \inf_{v \in {\mathcal P}_p(\widehat e)} \Vert u - v\Vert_{H^{1}(\widehat{e})}, \qquad 
\mbox{if $s \ge 1$}, \\
\label{eq:lemma:demkowicz-grad-1D-20-c}
\Vert (u-\widehat{\Pi}_p^{\operatorname*{grad},1d}u)^\prime\Vert_{\widetilde H^{-s}(\widehat{e})}
&\leq C_{s} p^{-s} \inf_{v \in {\mathcal P}_p(\widehat e)} \Vert u - v\Vert_{H^{1}(\widehat{e})}, \qquad 
\mbox{if $s \ge 0$}.
\end{align}
\end{subequations} 
\end{lemma}

\begin{proof}
The case $s = 0$ in (\ref{eq:lemma:demkowicz-grad-1D-20-a}) reflects the $p$-uniform $H^1$-stability 
of $\widehat{\Pi}_p^{\operatorname*{grad},1d}$ and the projection property of 
$\widehat{\Pi}_p^{\operatorname*{grad},1d}$. More specifically, to see the stability, i.e., 
\begin{equation}
\label{eq:pigrad-H1-stable}
\|\widehat{\Pi}_p^{\operatorname{grad},1d} u\|_{H^1(\widehat e)} \leq C \|u\|_{H^1(\widehat e)} 
\qquad \forall u \in H^1(\widehat e), 
\end{equation}
let ${\mathcal L} u \in {\mathcal P}_1(\widehat e)$
interpolate $u$ in the endpoints $\pm 1$ and note 
$\|{\mathcal L} u\|_{H^1(\widehat e)} \lesssim \|u\|_{H^1(\widehat e)}$ by Sobolev's embedding theorem. 
Since $u - {\mathcal L}u \in H^1_0(\widehat e)$ 
(and thus $\widehat \Pi_p^{\operatorname{grad},1d}(u - {\mathcal L}u) \in H^1_0(\widehat e)$) the orthogonality 
\eqref{eq:lemma:demkowicz-grad-1D-10-a} implies 
$|\widehat{\Pi}_p^{\operatorname*{grad},1d} (u  - {\mathcal L} u)|_{H^1(\widehat e)}
\leq |u  - {\mathcal L} u|_{H^1(\widehat e)}$. A Poincar\'e inequality yields 
$\|\widehat{\Pi}_p^{\operatorname*{grad},1d} (u  - {\mathcal L} u)\|_{H^1(\widehat e)}
\lesssim \|u - {\mathcal L} u\|_{H^1(\widehat e)}$. 
Writing $\widehat{\Pi}_p^{\operatorname*{grad},1d} u = 
\widehat{\Pi}_p^{\operatorname*{grad},1d} (u  - {\mathcal L} u) + 
\widehat{\Pi}_p^{\operatorname*{grad},1d} {\mathcal L} u =  
\widehat{\Pi}_p^{\operatorname*{grad},1d} (u  - {\mathcal L} u) + 
{\mathcal L} u$ finishes the proof of 
\eqref{eq:pigrad-H1-stable}. The projection property of $\widehat \Pi_p^{\operatorname{grad},1d}$ 
yields for any $v \in {\mathcal P}_p(\widehat e)$ that 
$u - \widehat \Pi_p^{\operatorname{grad},1d} u = 
u -v - \widehat \Pi_p^{\operatorname{grad},1d} (u-v)$. From this and 
\eqref{eq:pigrad-H1-stable} follows the estimate 
(\ref{eq:lemma:demkowicz-grad-1D-20-a}) for $s = 0$. 
%\hrule 
%The case $s = 0$ in (\ref{eq:lemma:demkowicz-grad-1D-20-a}) reflects the well-known best approximation property 
%of $\widehat{\Pi}_p^{\operatorname*{grad},1d}$. For $s\ge 1$, one proceeds by a standard duality argument. 
To see 
\eqref{eq:lemma:demkowicz-grad-1D-20-b} for 
$s\ge 1$, one proceeds by a duality argument. 
We 
set $\widetilde{e}:=u-\widehat{\Pi}_p^{\operatorname*{grad},1d}u$ and $t=-(1-s)\ge 0$ and estimate 
\begin{align*}
\Vert\widetilde{e}\Vert_{\widetilde{H}^{-t}(\widehat{e})} = \operatorname*{sup}_{v\in H^t(\widehat{e})} \frac{(\widetilde{e},v)_{L^2(\widehat{e})}}{\Vert v\Vert_{H^t(\widehat{e})}}.
\end{align*}
For every $v\in H^t(\widehat{e})$, there exists a unique solution $z\in H^{t+2}(\widehat{e}) \cap H_0^1(\widehat{e})$ of 
\begin{align*}
-z^{\prime\prime} = v \text{ in } \widehat{e}, \qquad z=0 \text{ on }\partial\widehat{e},
\end{align*}
which satisfies $\Vert z\Vert_{H^{t+2}(\widehat{e})} \lesssim \Vert v\Vert_{H^t(\widehat{e})}$. 
Using integration by parts, the orthogonality condition \eqref{eq:lemma:demkowicz-grad-1D-10-a}, 
and the estimate for $s=0$, we obtain 
\begin{align*}
&|(\widetilde{e},v)_{L^2(\widehat{e})}| = |(\widetilde{e}^\prime,z^\prime)_{L^2(\widehat{e})}| 
\stackrel{\eqref{eq:lemma:demkowicz-grad-1D-10-a}}{ \leq }
\Vert\widetilde{e}^\prime\Vert_{L^2(\widehat{e})} \operatorname*{inf}_{\pi\in\mathcal{P}_p(\widehat{e}) \cap H_0^1(\widehat{e})} \Vert z^\prime-\pi^\prime\Vert_{L^2(\widehat{e})} \\
&\quad \lesssim \Vert \widetilde e^\prime \Vert_{L^2(\widehat{e})} p^{-(t+1)} \Vert z\Vert_{H^{t+2}(\widehat{e})} 
\stackrel{\text{(\ref{eq:lemma:demkowicz-grad-1D-20-a}) with $s = 0$}}{\lesssim} 
p^{-(t+1)} \inf_{v \in {\mathcal P}_p(\widehat e)} \Vert u - v \Vert_{H^1(\widehat{e})} \Vert v\Vert_{H^t(\widehat{e})},
\end{align*}
which implies {\eqref{eq:lemma:demkowicz-grad-1D-20-b}} for $s\ge 1$.
Noting that $\Vert\cdot\Vert_{\widetilde H^0(\widehat e)} = \Vert\cdot\Vert_{L^2(\widehat e)} = \Vert\cdot\Vert_{H^0(\widehat e)}$, the cases  $s \in(0,1)$ follow by interpolation.
Finally, \eqref{eq:lemma:demkowicz-grad-1D-20-c} is shown by similar duality arguments, just use the dual problem
\begin{align*}
-z^{\prime\prime} = v-\overline{v} \text{ in } \widehat{e}, \qquad z^\prime = 0 \text{ on }\partial\widehat{e}.
\end{align*}
\end{proof}

\section{Stability of the projection operators in two space dimensions}

%-----------------------------------
\subsection{Preliminaries}
%-----------------------------------
We recall the following unconstrained approximation results:

\begin{lemma}
\label{lemma:Pgrad1d} 
Let $K$ be a fixed $d$-dimensional simplex in ${\mathbb R}^d$, $d \in \{1,2,3\}$. 
Fix $r \ge 0$. 
Then there are approximation operators $J_p:\mathbf{H}^r(K) \rightarrow ({\mathcal P}_p)^d$
such that 
$$
\|{\mathbf u} - J_p {\mathbf u}\|_{\mathbf{H}^s(K)} \leq C (p+1)^{-(r-s)} \|{\mathbf u}\|_{\mathbf{H}^r(K)}, 
\qquad \forall p \in {\mathbb N}_0, \quad 0 \leq s \leq r. 
$$
\end{lemma}

\begin{proof}
The scalar case $d = 1$ is well-known, see, e.g., \cite[Thm.~{5.1}]{melenk_nshpinterpolation_article}. 
The case $d > 1$ follows from a componentwise application of the case $d = 1$. 
\end{proof}

We need a shift theorem for the Laplacian: 
\begin{lemma}
\label{lemma:shift-theorem}
Let $\widehat f \subset {\mathbb R}^2$ be a triangle and $\widehat{s}$ be given by 
\eqref{eq:maximal-angle}.
For every $s \in [0,\widehat{s}-1)$ there is $C_s > 0$ such that the following shift theorems are true:
\begin{enumerate}[(i)]
\item For every $v\in H^s(\widehat{f})$ the solution $z$ of the problem
\begin{align*}
-\Delta z=v \text{ in } \widehat{f}, \quad z=0 \text{ on } \partial\widehat{f}, 
\end{align*}
satisfies $z\in H^{s+2}(\widehat{f}) \cap H_0^1(\widehat{f})$ 
with the estimate $\Vert z\Vert_{H^{s+2}(\widehat{f})} \leq C_s \Vert v\Vert_{H^s(\widehat{f})}$.
\item For every $v\in H^s(\widehat{f})$ and 
$g \in L^2(\partial\widehat f)$ with 
$g|_{e} \in H^{s+1/2}(e)$ $\forall e \in {\mathcal E}(\widehat f)$ that satisfies additionally the compatibility
condition $\int_{\widehat f} v + \int_{\partial \widehat f} g = 0$, the solution $z$ of the problem
\begin{align*}
-\Delta z=v \text{ in } \widehat{f}, \quad \partial_n z=g \text{ on } \partial\widehat{f}, 
\qquad \int_{\widehat f} z = 0, 
\end{align*}
satisfies $z \in H^{s+2}(\widehat f)$ and 
$\|z\|_{H^{s+2}(\widehat f)} \leq C_s \left[ \|v\|_{H^s(\widehat f)} + \sum_{e \in {\mathcal E}(\widehat f)} 
\|g\|_{H^{s+1/2}(e)}\right]$. 
\end{enumerate}
\end{lemma}

\begin{proof}
\emph{1st~step:} It follows from \cite{Grisvard85,dauge88,rojikdiss} that both regularity assertions
are satisfied for the case of homogeneous Dirichlet and Neumann conditions (i.e., $g = 0$). 
The key observation for that is that the leading corner singularities for both the homogeneous
Dirichlet and Neumann problem are in $H^{\widehat{s}+1-\varepsilon}(\widehat f)$ for every $\varepsilon > 0$ 
as can be seen from their explicit representation, \cite[Sec.~4, 5]{Grisvard85}, 
\cite[p.~{82}]{grisvard92}.

\emph{2nd~step:} For the case of inhomogeneous Neumann conditions $g \ne 0$, 
one constructs a vector field ${\boldsymbol \sigma} \in {\mathbf H}^{s+1}(\widehat f)$ such that 
${\boldsymbol \sigma} \cdot {\mathbf n} = g$ on $\partial \widehat f$. It is easy to construct such a vector field
away from the vertices, and near the vertices, an affine coordinate change together with a Piola transformation
for ${\boldsymbol \sigma}$
reduces the construction to one in a 
quarter plane, where each component of ${\boldsymbol \sigma}$ can be constructed separately by lifting from 
one of the coordinate axes. Next, one solves the two problems 
\begin{align*}
-\Delta z_0 &= v + \operatorname{div} {\boldsymbol \sigma} \quad \mbox{ in $\widehat f$}, 
\qquad \partial_n z_0 = 0 \quad \mbox{ on $\partial \widehat f$}, \\
-\Delta \widetilde z_0 &= \operatorname{curl } {\boldsymbol \sigma} \quad \mbox{ in $\widehat f$}, 
\qquad \widetilde z_0 = 0 \quad \mbox{ on $\partial \widehat f$}. 
\end{align*} 
From the first step, one has that $z_0$, $\widetilde z_0 \in H^{s+2}(\widehat f)$. It remains to see that 
$\nabla z = {\boldsymbol \sigma} - \operatorname*{\mathbf{curl}} \widetilde z_0 + \nabla z_0$. The 
difference ${\boldsymbol \delta}:= \nabla z - 
({\boldsymbol \sigma} - \operatorname*{\mathbf{curl}} \widetilde z_0 + \nabla z_0)$ satisfies 
$\operatorname{div} {\boldsymbol \delta}  = 0 = \operatorname{curl} {\boldsymbol \delta}$ and 
${\boldsymbol \delta} \cdot {\mathbf n} = 0 $ on $\partial \widehat f$. 
From $\operatorname{curl}{\boldsymbol \delta} = 0$ we get 
${\boldsymbol \delta}= \nabla \varphi$ for some $\varphi$, and 
$\operatorname{div} {\boldsymbol \delta}= 0$ produces $-\Delta \varphi = 0$. 
Together with 
${\boldsymbol \delta} \cdot {\mathbf n} = 0 $ we conclude that 
$\nabla \varphi = 0$.
\end{proof}

\begin{lemma}[\!\!\protect{\cite[Thm.~{4.2} with $s = 1$]{demkowicz08}}]
\label{lemma:Pgrad2d}
Let $P^{\operatorname*{grad},2d}u\in W_{p+1}(\widehat{f})$ be defined by
\begin{subequations}
\label{eq:lemma:Pgrad2d}
\begin{align}
\label{eq:lemma:Pgrad2d-a}
(\nabla(u-P^{\operatorname*{grad},2d}u),\nabla v)_{L^2(\widehat{f})} &= 0 \qquad \forall v\in W_{p+1}(\widehat{f}),\\
(u-P^{\operatorname*{grad},2d}u,1)_{L^2(\widehat{f})} &= 0.
\label{eq:lemma:Pgrad2d-b}
\end{align}
\end{subequations}
Then, for $r>1$, there holds 
%\begin{align}
$\displaystyle 
\Vert u-P^{\operatorname*{grad},2d}u\Vert_{H^1(\widehat{f})} \leq C_r p^{-(r-1)} \Vert u\Vert_{H^r(\widehat{f})}.
$
%\end{align}
\end{lemma}

%\begin{proof}
%See \cite{demkowicz08}.
%\end{proof}

\begin{lemma}[\!\!\protect{\cite[Thm.~{4.2} with $s=1$] {demkowicz08}}]
\label{lemma:Pcurl2d} Let $P^{\operatorname*{curl},2d}{\mathbf{u}}\in\mathbf{Q}_p(\widehat{f})$ be defined by 
\begin{subequations}
\label{eq:lemma:Pcurl2d-a}%
\begin{align}
(\operatorname{curl}({\mathbf{u}}-P^{\operatorname*{curl},2d}{\mathbf{u}%
}),\operatorname{curl}{\mathbf{v}})_{L^{2}(\widehat{f})}  &  =0\qquad\forall
{\mathbf{v}}\in\mathbf{Q}_p(\widehat{f}),\\
({\mathbf{u}}-P^{\operatorname*{curl},2d}{\mathbf{u}},\nabla v)_{L^{2}(\widehat{f})}
&  =0\qquad\forall v\in W_{p+1}(\widehat{f}).
\end{align}
Then, for $r>0$, there holds 
\end{subequations}
%\begin{equation}
$\displaystyle \Vert{\mathbf{u}}-P^{\operatorname*{curl},2d}{\mathbf{u}}\Vert_{{\mathbf{H}%
}(\widehat{f},\operatorname{curl})}\leq C_r p^{-r}\Vert{\mathbf{u}}\Vert_{{\mathbf{H}%
}^{r}(\widehat{f},\operatorname{curl})}.
$
%\end{equation}
\end{lemma}

%\begin{proof}
%See \cite{demkowicz08}.
%\end{proof}

The next lemma provides right inverses for the differential operators $\nabla$ and $\operatorname{curl}$. 

\begin{lemma}
[\!\!{\cite{costabel-mcintosh10},{\cite[Sec.~{2.3}]{bespalov-heuer09}}}]\label{lemma:mcintosh-2d} Let $B\subset\widehat{f}$ be a ball. Let $\theta\in C_{0}^{\infty}(B)$ with $\int_{B}%
\theta=1$. Define the operators
\begin{align*}
{R}^{\operatorname*{grad}}{\mathbf{u}}(\mathbf{x})  &  :=\int_{{\mathbf{a}}\in B} \theta(\mathbf{a})%
\int_{t=0}^{1}{\mathbf{u}}(\mathbf{a}+t(\mathbf{x}-{\mathbf{a}}))\,dt\cdot(\mathbf{x}-{\mathbf{a}%
})\,d{\mathbf{a}},\\
{\mathbf{R}}^{\operatorname*{curl}}{u}(\mathbf{x})  &  :=\int_{{\mathbf{a}}\in
B} \theta(\mathbf{a}) \int_{t=0}^{1}t{u}(\mathbf{a}+t(\mathbf{x}-{\mathbf{a}}))\,dt \, \left(\begin{array}{c}-(\mathbf{x}_2-\mathbf{a}_2)\\\mathbf{x}_1-\mathbf{a}_1\end{array}\right)\,d{\mathbf{a}}.
\end{align*}
Then:

\begin{enumerate}
[(i)]

\item 
\label{item:lemma:mcintosh-2d-i}
For $u\in L^2(\widehat{f})$, there holds
$\operatorname*{curl}{\mathbf{R}}^{\operatorname*{curl}}{u}=u$.

\item 
\label{item:lemma:mcintosh-2d-ii}
For ${\mathbf{u}}$ with $\operatorname*{curl}{\mathbf{u}}=0$, there
holds $\nabla{R}^{\operatorname*{grad}}{\mathbf{u}}={\mathbf{u}}$.

\item 
\label{item:lemma:mcintosh-2d-iii}
If ${\mathbf{u}}\in\mathbf{Q}_p(\widehat{f})$,
then ${R}^{\operatorname*{grad}}{\mathbf{u}}\in W_{p+1}(\widehat{f})$.

\item 
\label{item:lemma:mcintosh-2d-iv}
If ${u}\in V_p(\widehat{f})$, then ${\mathbf{R}}^{\operatorname*{curl}%
}{u}\in\mathbf{Q}_p(\widehat{f})$.

\item 
\label{item:lemma:mcintosh-2d-v}
For every $k\geq0$, the operators ${R}^{\operatorname*{grad}}$ and
${\mathbf{R}}^{\operatorname*{curl}}$ are bounded linear operators
$\mathbf{H}^{k}(\widehat{f})\rightarrow H^{k+1}(\widehat{f})$ and $H^{k}(\widehat{f})\rightarrow \mathbf{H}^{k+1}(\widehat{f})$, respectively.
\end{enumerate}
\end{lemma}

Lemma~\ref{lemma:mcintosh-2d} can now be used to construct regular Helmholtz-like decompositions.

\begin{lemma}
\label{lemma:helmholtz-like-decomp-2d} Let $s \ge0$. Then each ${\mathbf{u}}
\in{\mathbf{H}}^{s}(\widehat{f},\operatorname*{curl})$ can be written as
${\mathbf{u}} = \nabla\varphi+ {\mathbf{z}}$ with $\varphi\in H^{s+1}%
(\widehat{f})$, ${\mathbf{z}} \in{\mathbf{H}}^{s+1}(\widehat{f})$.
\end{lemma}

\begin{proof}
With the aid of the operators ${\mathbf{R}}^{\operatorname*{curl}}$,
$R^{\operatorname*{grad}}$ of Lemma~\ref{lemma:mcintosh-2d}, we write
$\displaystyle 
{\mathbf{u}}=\nabla R^{\operatorname*{grad}}({\mathbf{u}}-{\mathbf{R}%
}^{\operatorname*{curl}}(\operatorname*{curl}{\mathbf{u}}))+{\mathbf{R}%
}^{\operatorname*{curl}}(\operatorname*{curl}{\mathbf{u}}).
$
The mapping properties of ${\mathbf{R}}^{\operatorname*{curl}}$ and
$R^{\operatorname*{grad}}$ of Lemma~\ref{lemma:mcintosh-2d} then imply the result.
\end{proof}
\begin{lemma}
\label{lemma:helmholtz-like-decomp-2d-v2} Let $\widehat f \subset {\mathbb R}^2$
be a triangle.  Let $s \in [1,\widehat{s})$ with 
$\widehat{s}$ given by \eqref{eq:maximal-angle}.
Then each ${\mathbf{u}}
\in{\mathbf{H}}^{s}(\widehat{f})$ can be written as
${\mathbf{u}} = \nabla\varphi+ \operatorname{\mathbf{curl}}z$ with 
$\varphi\in H^{s+1}(\widehat{f}) \cap H^1_0(\widehat{f})$, 
$z\in H^{s+1}(\widehat{f})$ with $(z,1)_{L^2(\widehat f)} = 0$ and 
satisfying $\|\varphi\|_{H^{s+1}(\widehat{f})} + \|z\|_{H^{s+1}(\widehat{f})} \lesssim \|\mathbf{u}\|_{\mathbf{H}^s(\widehat{f})}$.% and $\|z\|_{H^{s+1}(\widehat{f})} \lesssim \|\mathbf{u}\|_{\mathbf{H}^s(\widehat{f})}$.
\end{lemma}

\begin{proof}
We define $\varphi,z\in H^{s+1}(\widehat{f})$ as the solutions of the equations
\begin{subequations}
\begin{align}
-\Delta\varphi &  =-\operatorname{div}{\mathbf{u}},\quad\varphi
=0\quad\text{on }\partial \widehat{f},
\label{eq:lemma:picurl-negative-I-20a}\\
-\Delta z  &  =\operatorname{curl}{\mathbf{u}},\quad\partial_{n}%
%z=-{\mathbf{t}}\cdot\operatorname{\mathbf{curl}}z=-{\mathbf{t}}%
z=-{\mathbf{t}}%
\cdot({\mathbf{u}}-\nabla\varphi)\quad\text{on }\partial \widehat{f},\quad\int_{\widehat{f}}z=0.
\label{eq:lemma:picurl-negative-I-20b}%
\end{align}
\end{subequations}
Here, ${\mathbf{t}}$ denotes the unit tangent vector on $\partial \widehat{f}$ oriented
such that $\widehat{f}$ is \textquotedblleft on the left\textquotedblright. We note that
(\ref{eq:lemma:picurl-negative-I-20b}) is a Neumann problem; integration by
parts (cf.~\eqref{eq:2d-stokes}) 
shows that the solvability condition is satisfied. By Lemma~\ref{lemma:shift-theorem} we have
\begin{align}
&\Vert\varphi\Vert_{H^{s+1}(\widehat{f})}\lesssim\Vert\operatorname{div}{\mathbf{u}}\Vert_{H^{s-1}(\widehat{f})} \lesssim \|\mathbf{u}\|_{\mathbf{H}^s(\widehat{f})},\qquad\Vert z\Vert_{H^{s+1}(\widehat{f})}\lesssim\Vert{\mathbf{u}}%
\Vert_{{\mathbf{H}}^{s}(\widehat{f})}. 
%\qedhere
\end{align}
The representation ${\mathbf u} = \nabla \varphi + \operatorname{\mathbf{curl}} z$ follows as 
at the end of the proof of Lemma~\ref{lemma:shift-theorem}: 
${\boldsymbol \delta}:= {\mathbf u} - (\nabla \varphi + \operatorname{\mathbf{curl}} z)$ satisfies 
$\operatorname{div} {\boldsymbol \delta} = 0 = 
\operatorname{curl} {\boldsymbol \delta}$, and ${\mathbf t} \cdot {\boldsymbol \delta} = 0$
so that ${\boldsymbol \delta} = 0$.
\end{proof}
\begin{lemma}
[discrete Friedrichs inequality in 2D]\label{lemma:discrete-friedrichs} 
Let $\widehat f \subset {\mathbb R}^2$ be a triangle. 
There exists $C > 0$ independent of $p$
and ${\mathbf{u}}$ such that
\begin{equation}
\label{eq:lemma:discrete-friedrichs-2d}\|{\mathbf{u}}\|_{L^{2}(\widehat{f})} \leq C
\|\operatorname{curl} {\mathbf{u}}\|_{L^{2}(\widehat{f})}%
\end{equation}
in the following two cases:

\begin{enumerate}
[(i)]

\item \label{item:lemma:discrete-friedrichs-i} ${\mathbf{u}}\in
\mathbf{Q}_{p}(\widehat{f})$ satisfies $({\mathbf{u}%
},\nabla v)_{L^{2}(\widehat{f})}=0$ for all $v\in W_{p+1}(\widehat{f})$.

\item \label{item:lemma:discrete-friedrichs-ii} ${\mathbf{u}} \in
\mathring{\mathbf{Q}}_p(\widehat{f})$ satisfies $({\mathbf{u}}, \nabla v)_{L^{2}(\widehat{f})} = 0$
for all $v \in\mathring{W}_{p+1}(\widehat{f})$.
\end{enumerate}
\end{lemma}

\begin{proof}
Statement (\ref{item:lemma:discrete-friedrichs-i}) is proved in \cite[Lemma~6]%
{demkowicz-buffa05} or \cite[Lemma~{4.1}]{demkowicz08}. Statement
(\ref{item:lemma:discrete-friedrichs-ii}) is shown with similar techniques:
Using the operators $R^{\operatorname*{grad}}$ and ${\mathbf{R}}^{\operatorname*{curl}}$
of Lemma~\ref{lemma:mcintosh-2d}, we decompose ${\mathbf{u}}\in\mathring{\mathbf{Q}}_p(\widehat{f})$ as
\[
{\mathbf{u}}=\nabla\psi+{\mathbf{R}}^{\operatorname*{curl}}%
(\operatorname{curl}{\mathbf{u}}),\qquad\psi:=R^{\operatorname*{grad}%
}({\mathbf{u}}-{\mathbf{R}}^{\operatorname*{curl}}(\operatorname{curl}{\mathbf{u}})).
\]
Since ${\mathbf{u}}\in\mathbf{Q}_p(\widehat{f})$ we
have $\psi\in W_{p+1}(\widehat{f})$. The property ${\mathbf{u}}\in\mathring{\mathbf{Q}}_p(\widehat{f})$ 
implies with the tangential vector ${\mathbf{t}}$ on the boundary $\partial \widehat{f}$
\[
{\mathbf{t}}\cdot\nabla\psi=-{\mathbf{t}}\cdot{\mathbf{R}}%
^{\operatorname*{curl}}(\operatorname{curl}{\mathbf{u}}).
\]
Since $\psi$ is continuous at the vertices of $\widehat{f}$, we infer
\begin{align*}
|\psi|_{H^{1/2}(\partial \widehat{f})}  &  \lesssim|\psi|_{H^{1}(\partial \widehat{f})}%
= \|{\mathbf{t}}\cdot{\mathbf{R}}^{\operatorname*{curl}}(\operatorname{curl}{\mathbf{u}})\|_{L^{2}(\partial \widehat{f})}\leq\|{\mathbf{R}}^{\operatorname*{curl}%
}(\operatorname{curl}{\mathbf{u}})\|_{L^{2}(\partial \widehat{f})}\\
&  \lesssim\Vert{\mathbf{R}}^{\operatorname*{curl}}(\operatorname{curl}{\mathbf{u}})\Vert_{H^{1/2}(\partial \widehat{f})}\lesssim\Vert{\mathbf{R}%
}^{\operatorname*{curl}}(\operatorname{curl}{\mathbf{u}})\Vert_{H^{1}%
(\widehat{f})}\lesssim\Vert\operatorname{curl}{\mathbf{u}}\Vert_{L^{2}(\widehat{f})}.% 
\end{align*}
Next, we decompose $\psi=\psi_{0}+{\mathcal{L}^{\operatorname{grad},2d}}(\psi|_{\partial \widehat{f}})$, where
${\mathcal{L}^{\operatorname{grad},2d}}:H^{1/2}(\partial \widehat{f})\rightarrow H^{1}(\widehat{f})$ is the lifting
operator of \cite{babuska-craig-mandel-pitkaranta91}. The lifting ${\mathcal L}^{\operatorname{grad},2d}$ of 
\cite{babuska-craig-mandel-pitkaranta91} is such that 
$\psi \in W_{p+1}(\widehat{f})$ 
implies ${\mathcal L}^{\operatorname{grad},2d}(\psi|_{\partial\widehat{f}}) \in W_{p+1}(\widehat f)$ 
so that $\psi_0 \in \mathring{W}_{p+1}(\widehat f)$. Since furthermore ${\mathcal L}^{\operatorname{grad},2d} 1 = 1$, 
we deduce $\|\nabla {\mathcal L}^{\operatorname{grad},2d} \psi\|_{L^2(\widehat f)} \lesssim |\psi|_{H^{1/2}(\partial \widehat f)}$ 
and estimate
%Since ${\mathcal{L}}$
%produces a polynomial and $\psi\in W_{p+1}(\widehat{f})$, we get that $\psi_{0}\in
%\mathring{W}_{p+1}(\widehat{f})$ and estimate 
\begin{align*}
& \Vert{\mathbf{u}}\Vert_{L^{2}(\widehat{f})}^{2}    =({\mathbf{u}},\nabla\psi_{0}%
+\nabla{\mathcal{L}^{\operatorname{grad},2d}}(\psi|_{\partial \widehat{f}})+{\mathbf{R}}^{\operatorname*{curl}%
}(\operatorname{curl}{\mathbf{u}}))_{L^{2}(\widehat{f})} \\
& \quad =({\mathbf{u}}%
,\nabla{\mathcal{L}^{\operatorname{grad},2d}}(\psi|_{\partial \widehat{f}})+{\mathbf{R}}^{\operatorname*{curl}%
}(\operatorname{curl}{\mathbf{u}}))_{L^{2}(\widehat{f})}\\
&  \quad \leq\Vert{\mathbf{u}}\Vert_{L^{2}(\widehat{f})}\left\{  \Vert\nabla{\mathcal{L}^{\operatorname{grad},2d}}%
(\psi|_{\partial \widehat{f}})\Vert_{L^{2}(\widehat{f})}+\Vert{\mathbf{R}}^{\operatorname*{curl}%
}(\operatorname{curl}{\mathbf{u}})\Vert_{L^{2}(\widehat{f})}\right\} \\
&  \quad \leq\Vert{\mathbf{u}}\Vert_{L^{2}(\widehat{f})}\left\{  |\psi|_{H^{1/2}%
(\partial \widehat{f})}+\Vert{\mathbf{R}}^{\operatorname*{curl}}(\operatorname{curl}{\mathbf{u}})\Vert_{L^{2}(\widehat{f})}\right\}  \lesssim\Vert{\mathbf{u}}%
\Vert_{L^{2}(\widehat{f})}\Vert\operatorname{curl}{\mathbf{u}}\Vert_{L^{2}(\widehat{f})}.
\qedhere
\end{align*}
%This concludes the proof.%{\Huge komisch: das wirkt suboptimal...}
\end{proof}
%-------------------------------
\subsection{Stability of the operator $\protect\hatPigradcomtwod$}
%-------------------------------

\begin{theorem}
\label{lemma:demkowicz-grad-2D} 
Let $\widehat f\subset{\mathbb R}^2$ be a triangle and 
$\widehat{s}$ be given by \eqref{eq:maximal-angle}.
For every $s\in [0,\widehat{s})$
there is $C_{s}>0$ such that
for $u\in H^{3/2}(\widehat{f})$
\begin{subequations}
\label{eq:lemma:demkowicz-grad-2D-10}%
\begin{align}
\label{eq:lemma:demkowicz-grad-2D-10-a}%
\Vert u-\hatPigradcomtwod u\Vert_{H^{1-s}(\widehat{f})}& \leq
C_{s}p^{-(1/2+s)}\!\!\!\inf_{v \in {\mathcal P}_{p+1}(\widehat f)} \Vert u -v \Vert_{H^{3/2}(\widehat{f})} 
\,\,\mbox{ if $s \in [0,1]$}, \\
\label{eq:lemma:demkowicz-grad-2D-10-b}%
\Vert u-\hatPigradcomtwod u\Vert_{\widetilde H^{1-s}(\widehat{f})}& \leq
C_{s}p^{-(1/2+s)}\!\!\!\inf_{v \in {\mathcal P}_{p+1}(\widehat f)} \Vert u -v \Vert_{H^{3/2}(\widehat{f})} 
\,\,\mbox{ if $s \in [1,\widehat{s})$}, \\
\label{eq:lemma:demkowicz-grad-2D-10-c}%
\Vert \nabla(u-\hatPigradcomtwod u)\Vert_{\widetilde{\mathbf{H}}^{-s}(\widehat{f})}& \leq
C_{s}p^{-(1/2+s)}\!\!\!\inf_{v \in {\mathcal P}_{p+1}(\widehat f)} \Vert u -v \Vert_{H^{3/2}(\widehat{f})} 
\,\,\mbox{ if $s \in [0,\widehat{s})$}.
\end{align}
\end{subequations}
%here, $\widetilde H^{-t}(\widehat f) = (H^{-t}(\widehat f))^\prime$ for $t \ge 0$. 
\end{theorem}

\begin{proof} 
%\neu{Before proving Theorem~\ref{lemma:demkowicz-grad-2D} we highlight
%that  the estimate (\ref{eq:lemma:demkowicz-grad-2D-10-a}), i.e., the estimate
%in the restricted range $s \in [0,1]$ holds for any (fixed, convex) triangle
%$\widehat f$.}
%
By the projection property of $\hatPigradcomtwod$
it suffices to show the estimates 
(\ref{eq:lemma:demkowicz-grad-2D-10-a}), 
(\ref{eq:lemma:demkowicz-grad-2D-10-b}), (\ref{eq:lemma:demkowicz-grad-2D-10-c})
for the special case $v = 0$ in the infimum. 

\emph{1st step:}
We show (\ref{eq:lemma:demkowicz-grad-2D-10-a}) for the case $s = 0$. 
%In a second step, we show 
%(\ref{eq:lemma:demkowicz-grad-2D-10-b}) \clr{and (\ref{eq:lemma:demkowicz-grad-2D-10-c})} for the cases $s \in [1,\widehat{s})$. The remaining cases 
%$s \in (0,1)$ are obtained by interpolating between the case $s = 0$ \clr{(for which we have $\Vert \nabla(u-\hatPigradcomtwod u)\Vert_{\widetilde{\mathbf{H}}^{-s}(\widehat{f})} \lesssim \Vert u-\hatPigradcomtwod u\Vert_{H^{1-s}(\widehat{f})}$)} and the case $s=1$ (for which
%(\ref{eq:lemma:demkowicz-grad-2D-10-a}) and (\ref{eq:lemma:demkowicz-grad-2D-10-b}) coincide).  
%
The trace theorem gives $u \in H^1(e)$ for each edge $e \in {\mathcal E}(\widehat f)$ 
with $\|u\|_{H^1(e)} \lesssim \|u\|_{H^{3/2}(\widehat f)}$ 
(cf.\ Lemma~\ref{lemma:traces-2d}).
By Lemma~\ref{lemma:demkowicz-grad-1D}, 
we have for every edge $e\in{\mathcal{E}%
}(\widehat{f})$
\begin{equation}
\Vert u-\hatPigradcomtwod u\Vert_{H^{1-s}(e)}\leq Cp^{-s}%
\Vert u\Vert_{H^{3/2}(\widehat{f})},\qquad s\in\lbrack0,1].
\end{equation}
Since $\hatPigradcomtwod u$ is piecewise polynomial and continuous
on $\partial\widehat{f}$, we infer in particular for $s  = 0$ and $s=1$ the bounds
\begin{equation}
\Vert u-\hatPigradcomtwod u\Vert_{H^{1-s}(\partial\widehat{f})}\leq
Cp^{-s}\Vert u\Vert_{H^{3/2}(\widehat{f})},
\label{eq:lemma:demkowicz-grad-2D-15}%
\end{equation}
and then, by interpolation, also for the intermediate
$s\in(0,1)$. 

To show (\ref{eq:lemma:demkowicz-grad-2D-10-a}) for $s=0$, we first observe that 
$P^{\operatorname*{grad},2d} u - \hatPigradcomtwod u$ is discrete harmonic, i.e., 
\begin{align*} 
(\nabla \underbrace{(P^{\operatorname{grad},2d} u - \hatPigradcomtwod u)
                   }_{=:\delta_p \in W_{p+1}(\widehat f)} ,\nabla v)_{L^2(\widehat f)} = 0 
\qquad \forall v \in \mathring{W}_{p+1}(\widehat f).
\end{align*}
With the continuous, polynomial preserving lifting 
${\mathcal{L}}^{\operatorname{grad},2d}:H^{1/2}(\partial\widehat{f}) \rightarrow H^1(\widehat{f})$ 
of \cite{babuska-craig-mandel-pitkaranta91} 
we then get
\begin{align*}
\left(\nabla \delta_p ,\nabla (\delta_p-{\mathcal{L}}^{\operatorname{grad},2d} \delta_p)\right)_{L^2(\widehat f)} = 0
\end{align*}
and thus
\begin{equation}
\label{eq:lemma:demkowicz-grad-2D-17}
|\delta_p|_{H^1(\widehat{f})}^2 \!=\! (\nabla \delta_p,\nabla \delta_p)_{L^2(\widehat{f})} 
\!=\! (\nabla \delta_p ,\nabla({\mathcal{L}}^{\operatorname{grad},2d}\delta_p))_{L^2(\widehat{f})} 
\!\lesssim\! |\delta_p|_{H^1(\widehat{f})} |\delta_p|_{H^{1/2}(\partial\widehat{f})}.
\!\!
\!\!
\!\!
\end{equation}
We infer from Lemma~\ref{lemma:Pgrad2d}, 
\eqref{eq:lemma:demkowicz-grad-2D-15}, and \eqref{eq:lemma:demkowicz-grad-2D-17} for the seminorm 
$|\cdot|_{H^1(\widehat{f})}$
\begin{align}
\label{eq:lemma:demkowicz-2d-105}
& |u-\hatPigradcomtwod u|_{H^1(\widehat{f})} \leq |u-P^{\operatorname*{grad},2d}u|_{H^1(\widehat{f})} + |P^{\operatorname*{grad},2d}u-\hatPigradcomtwod u|_{H^1(\widehat{f})} \\
\nonumber 
&\qquad \stackrel{\text{Lem.~\ref{lemma:Pgrad2d},(\ref{eq:lemma:demkowicz-grad-2D-17}) }}{\lesssim} p^{-1/2} \Vert u\Vert_{H^{3/2}(\widehat{f})} + \Vert P^{\operatorname*{grad},2d}u-\hatPigradcomtwod u\Vert_{H^{1/2}(\partial\widehat{f})} \\
\nonumber 
&\qquad 
\stackrel{\eqref{eq:lemma:demkowicz-grad-2D-15}}{ \lesssim}
 p^{-1/2} \Vert u\Vert_{H^{3/2}(\widehat{f})} + \Vert u-P^{\operatorname*{grad},2d}u\Vert_{H^1(\widehat{f})} 
\lesssim p^{-1/2} \Vert u\Vert_{H^{3/2}(\widehat{f})},
\end{align}
which is (\ref{eq:lemma:demkowicz-grad-2D-10-a}) for $s=0$. 

\emph{2nd step:}
We next show 
the estimate (\ref{eq:lemma:demkowicz-grad-2D-10-b}) for $s \in [1,\widehat{s})$ by a duality argument. 
Let $\widetilde{e}=u-\hatPigradcomtwod u$, and set $t=-(1-s)$. To estimate 
\begin{align}
\label{eq:lemma:demkowicz-2d-15}
\Vert\widetilde{e}\Vert_{\widetilde{H}^{-t}(\widehat{f})} = \operatorname*{sup}_{v\in H^t(\widehat{f})} \frac{(\widetilde{e},v)_{L^2(\widehat{f})}}{\Vert v\Vert_{H^t(\widehat{f})}}
\end{align}
let $v\in H^t(\widehat{f})$ and $z\in H^{t+2}(\widehat{f})\cap H_{0}%
^{1}(\widehat{f})$ solve (cf.~Lemma~\ref{lemma:shift-theorem}) 
\begin{align*}
-\Delta z=v\quad\text{in }\widehat{f},\qquad z|_{\partial
\widehat{f}}=0.
\end{align*}
Note the \textsl{a priori} estimate $\Vert z\Vert_{H^{t+2}(\widehat{f})}\leq
C\Vert v\Vert_{H^t(\widehat{f})}$. Integration by parts
yields
\begin{equation}
(\widetilde{e},v)_{L^{2}(\widehat{f})}=\int_{\widehat{f}}%
\nabla\widetilde{e}\cdot\nabla z-\int_{\partial\widehat{f}}\partial
_{n}z\widetilde{e}. \label{eq:lemma:demkowicz-2d-100}%
\end{equation}
For the first term in (\ref{eq:lemma:demkowicz-2d-100}) we get by the
orthogonality properties satisfied by $\widetilde{e}$, 
Lemma~\ref{lemma:Pgrad1d}, and \eqref{eq:lemma:demkowicz-2d-105} 
\begin{align}
\nonumber 
 \left\vert \int_{\widehat{f}}\nabla z\cdot\nabla\widetilde{e}\right\vert  
& 
\leq\inf_{\pi\in{\mathcal{P}}_{p}\cap H_{0}^{1}(\widehat{f})}\Vert z-\pi
\Vert_{H^{1}(\widehat{f})}\Vert\nabla\widetilde{e}\Vert_{L^{2}(\widehat{f}%
)}
\lesssim p^{-(t+1)}\Vert z\Vert_{H^{t+2}(\widehat{f})}\Vert\nabla
\widetilde{e}\Vert_{L^{2}(\widehat{f})}  \\
\label{eq:lemma:demkowicz-2d-17}
&\lesssim p^{-(t+1)} \Vert\nabla
\widetilde{e}\Vert_{L^{2}(\widehat{f})} \Vert v\Vert_{H^t(\widehat{f})}  
\stackrel{\text{(\ref{eq:lemma:demkowicz-grad-2D-10-a}), $s=0$}}{\lesssim} p^{-(1/2+s)} \Vert u\Vert_{H^{3/2}(\widehat{f})} \Vert v\Vert_{H^t(\widehat{f})}.
\end{align}
%For the second term in (\ref{eq:lemma:demkowicz-2d-100}) we define, for each
%edge $e\in\mathcal{E}(\widehat{f})$, the function $Z\in H^{2}(e)\cap H_{0}^{1}(e)$ by
%\begin{equation}
%-\nabla_{e}^{2}Z=\partial_{n}z\quad\text{on }e,\qquad Z|_{\partial e}=0.
%\end{equation}
%We have the \textsl{a priori} estimate $\Vert Z\Vert_{H^{5/2}(e)}\leq
%\Vert\partial_{n}z\Vert_{H^{1/2}(e)}\leq C\Vert z\Vert_{H^{2}(\widehat{f}%
%)}\leq C\Vert\widetilde{e}\Vert_{L^{2}(\widehat{f})}$. We obtain (note that
%$\widetilde{e}(V)=0$ for the vertices $V\in{\mathcal{V}}(\widehat{f})$)
%\[
%\int_{e}\partial_{n}z\widetilde{e}=\int_{e}-\nabla_{e}^{2}Z\widetilde{e}%
%=\int_{e}\nabla_{e}Z\nabla_{e}\widetilde{e}=\inf_{\pi\in{\mathcal{P}}_{p}\cap
%H_{0}^{1}(e)}\int_{e}\nabla_{e}(Z-\pi)\nabla_{e}\widetilde{e}\leq
%Cp^{-3/2}\Vert\widetilde{e}\Vert_{L^{2}(\widehat{f})}\Vert\partial
%_{t}\widetilde{e}\Vert_{L^{2}(e)}.
%\]
For the second term in \eqref{eq:lemma:demkowicz-2d-100} we use 
Lemma~\ref{lemma:demkowicz-grad-1D} to obtain for each edge $e\in\mathcal{E}(\widehat{f})$
\begin{align}
\nonumber 
|(\partial_n z,\widetilde{e})_{L^2(e)}| 
&\lesssim \Vert\widetilde{e}\Vert_{\widetilde{H}^{-(t+1/2)}(e)} \Vert\partial_n z\Vert_{H^{t+1/2}(e)} \lesssim p^{-(3/2+t)} \Vert u\Vert_{H^{1}(e)} \Vert z\Vert_{H^{t+2}(\widehat{f})} \\
\label{eq:lemma:demkowicz-2d-19}
& \lesssim p^{-(1/2+s)} \Vert u\Vert_{H^{3/2}(\widehat{f})} \Vert v\Vert_{H^t(\widehat{f})}.
\end{align}
Inserting 
(\ref{eq:lemma:demkowicz-2d-17}) and (\ref{eq:lemma:demkowicz-2d-19}) 
in 
(\ref{eq:lemma:demkowicz-2d-15}) yields 
%\begin{align*}
%|(\widetilde{e},v)_{L^{2}(\widehat{f})}| \lesssim p^{-(k-s)} \Vert u\Vert_{H^k(\widehat{f})} \Vert v\Vert_{H^t(\widehat{f})},
%\end{align*}
(\ref{eq:lemma:demkowicz-grad-2D-10-b}) for $s\in [1,\widehat{s})$. 

\emph{3rd step:}
The estimate
(\ref{eq:lemma:demkowicz-grad-2D-10-a}) for $s \in (0,1)$ follows by interpolation between 
$s = 0$ (cf.~1st step) and $s = 1$, noting that 
\eqref{eq:lemma:demkowicz-grad-2D-10-a} and \eqref{eq:lemma:demkowicz-grad-2D-10-b} coincide for $s = 1$. 

\emph{4th step:}
We show the estimate \eqref{eq:lemma:demkowicz-grad-2D-10-c} for $s\in [1,\widehat{s})$ again by duality. We set $\widetilde{e}:=u-\hatPigradcomtwod u$ and need an estimate for
\begin{align}
\label{eq:lemma:demkowicz-2d-21}
\|\nabla\widetilde{e}\|_{\widetilde{\mathbf{H}}^{-s}(\widehat{f})} = \operatorname*{sup}_{\mathbf{v}\in \mathbf{H}^s(\widehat{f})} \frac{(\nabla\widetilde{e},\mathbf{v})_{L^2(\widehat{f})}}{\|\mathbf{v}\|_{\mathbf{H}^s(\widehat{f})}}.
\end{align}
According to Lemma~\ref{lemma:helmholtz-like-decomp-2d-v2}, any $\mathbf{v}\in \mathbf{H}^s(\widehat{f})$ can be decomposed as $\mathbf{v}=\nabla\varphi+\operatorname{\mathbf{curl}} z$ with 
$\varphi \in H^{s+1}(\widehat{f}) \cap H^1_0(\widehat{f})$,  
$z\in H^{s+1}(\widehat{f})$. Integration by parts then gives
\begin{align*}
(\nabla\widetilde{e},\mathbf{v})_{L^2(\widehat{f})} \!=\! (\nabla\widetilde{e},\nabla\varphi)_{L^2(\widehat{f})} \!+\! (\nabla\widetilde{e},\operatorname{\mathbf{curl}} z)_{L^2(\widehat{f})} \!=\! (\nabla\widetilde{e},\nabla\varphi)_{L^2(\widehat{f})} \!+\! (\mathbf{t}\cdot \nabla\widetilde{e},z)_{L^2(\partial\widehat{f})}.
\end{align*}
For the first term, we use Lemma~\ref{lemma:Pgrad1d} and \eqref{eq:lemma:demkowicz-grad-2D-10-a} 
(applied with $s=0$) to obtain
\begin{align*}
\left| (\nabla\widetilde{e},\nabla\varphi)_{L^2(\widehat{f})}\right|\! \lesssim \|\nabla\widetilde{e}\|_{L^2(\widehat{f})} \!\inf_{\pi\in \mathring{W}_{p+1}(\widehat{f})}\!\! \|\varphi-\pi\|_{H^1(\widehat{f})} \lesssim p^{-(s+1/2)} \|u\|_{H^{3/2}(\widehat{f})} \|\mathbf{v}\|_{\mathbf{H}^s(\widehat{f})},
\end{align*}
analogously to \eqref{eq:lemma:demkowicz-2d-17}. To treat the second term, 
we note that $z\in H^{s+1}(\widehat{f})$ implies $z\in H^{s+1/2}(e)$ 
for each edge $e\in\mathcal{E}(\widehat{f})$ and get
\begin{align*}
&\left| (\mathbf{t}\cdot \nabla\widetilde{e},z)_{L^2(e)}\right| = 
\left| (\nabla_e\widetilde{e},z)_{L^2(e)} \right| 
\lesssim \|\nabla_e \widetilde{e}\|_{\widetilde{H}^{-(s+1/2)}(e)} \|z\|_{H^{s+1/2}(e)} \\
&\quad \stackrel{\text{Lemma~\ref{lemma:demkowicz-grad-1D}}}{\lesssim} p^{-(s+1/2)} \|u\|_{H^1(e)} \|z\|_{H^{s+1/2}(e)} \lesssim p^{-(s+1/2)} \|u\|_{H^{3/2}(\widehat{f})} \|\mathbf{v}\|_{\mathbf{H}^s(\widehat{f})}.
\end{align*}
Inserting the last two estimates in \eqref{eq:lemma:demkowicz-2d-21} yields \eqref{eq:lemma:demkowicz-grad-2D-10-c} for $s\in [1,\widehat{s})$. 

\emph{5th step:}
The estimate \eqref{eq:lemma:demkowicz-grad-2D-10-c} for $s\in (0,1)$ follows by interpolation between 
$s = 0$ and $s = 1$. 
\end{proof}

%\begin{corollary}
%\label{cor:higher-regularity-grad-2d}
%For all $k\geq \frac{3}{2}$ and all $u\in
%H^{k}(\widehat{f})$ there holds
%\begin{align*}
%\Vert u-\hatPigradcomtwod u\Vert
%_{H^{s}(\widehat{f})}\leq C_{k,s}p^{-(k-s)}\inf_{v\in W_{p+1}(\widehat{f})}\Vert u-v\Vert_{H^{k}(\widehat{K})},\qquad
%s\in (-2,1].
%\end{align*}
%\end{corollary}
%
%\begin{proof}
%The desired result follows from Theorem~\ref{lemma:demkowicz-grad-2D} together with the projection property $\hatPigradcomtwod v = v$ for all $v\in W_{p+1}(\widehat{f})$ by
%\begin{align*}
%\Vert u-\hatPigradcomtwod u\Vert_{H^s(\widehat{f})}=\Vert(u-v)-\hatPigradcomtwod(u-v)\Vert_{H^s(\widehat{f})} \leq C_s p^{-(k-s)}\Vert u-v\Vert_{H^k(\widehat{f})}
%\end{align*}
%and infimizing over $v\in W_{p+1}(\widehat{f})$. 
%\end{proof}
%-------------------------------
\subsection{Stability of the operator $\protect\hatPicurlcomtwod$}
%-------------------------------

The following two lemmas present the duality arguments that are needed later on 
to estimate negative Sobolev norms.

\begin{lemma}
\label{lemma:picurl-negative-I} 
%\neu{Let $\widehat f$ be the equilateral reference triangle.}
Let
${\mathbf{E}}\in \mathbf{H}^{1/2}(\widehat{f},\operatorname*{curl})$ satisfy
the orthogonality conditions%
\begin{subequations}
\label{eq:lemma:picurl-negative-I-orth}
\begin{align}
(\operatorname{curl}{\mathbf{E}},\operatorname{curl}{\mathbf{v}%
})_{L^{2}(\widehat{f})}  &  =0\qquad\forall{\mathbf{v}}\in\mathring{\mathbf{Q}}_p(\widehat{f}),
\label{eq:lemma:picurl-negative-I-orth-a}\\
({\mathbf{E}},\nabla\varphi)_{L^{2}(\widehat{f})}  &  =0\qquad\forall\varphi\in
\mathring{W}_{p+1}(\widehat{f}),
\label{eq:lemma:picurl-negative-I-orth-b}\\
({\mathbf{E}}\cdot{\mathbf{t}}_{e},\nabla_{e}\varphi)_{L^{2}(e)}  &
=0\qquad\forall\varphi\in\mathring{W}_{p+1}(e) \quad\forall e\in{\mathcal{E}%
}(\widehat{f}),
\label{eq:lemma:picurl-negative-I-orth-c}\\
({\mathbf{E}}\cdot{\mathbf{t}_e},1)_{L^{2}(e)}  &  =0\qquad\forall
e\in{\mathcal{E}}(\widehat{f}). 
\label{eq:lemma:picurl-negative-I-orth-d}%
\end{align}
Then, for $s\in [0,\widehat{s})$, there holds 
$\displaystyle 
\Vert{\mathbf{E}}\Vert_{\widetilde{\mathbf{H}}^{-s}(\widehat{f})}
\leq C_s p^{-s}\Vert{\mathbf{E}%
}\Vert_{\mathbf{H}(\widehat{f},\operatorname{curl})}.$
%\end{equation}
\end{subequations}%
\end{lemma}

\begin{proof}
%\neu{Before proving Lemma~\ref{lemma:picurl-negative-I}, we hightlight that for any fixed, convex triangle
%$\widehat f$, the estimate is valid in the restricted range $s \in [0,1]$. }
%
Before proceeding with the proof, we point out that Lemma~\ref{lemma:traces-2d} shows that 
${\mathbf E} \cdot {\mathbf t}_e \in L^2(e)$ so that the 
conditions \eqref{eq:lemma:picurl-negative-I-orth-c}, \eqref{eq:lemma:picurl-negative-I-orth-d} 
 are meaningful. 

\emph{1st~step:} 
We may restrict to the case $s \ge 1$ as the case $s = 0$ is trivial and the 
remaining cases $s \in (0,1)$ follow then by an interpolation argument.

\emph{2nd~step:} By Lemma~\ref{lemma:helmholtz-like-decomp-2d-v2}, 
any ${\mathbf{v}}\in{\mathbf{H}}^{s}(\widehat{f})$ can be decomposed as
\begin{equation}
{\mathbf{v}}=\nabla\varphi+\operatorname{\mathbf{curl}}z,
\label{eq:lemma:picurl-negative-I-10}%
\end{equation}
with $\varphi \in H^{s+1}(\widehat{f}) \cap H^1_0(\widehat{f})$ and 
$z\in H^{s+1}(\widehat{f})$ with $(z,1)_{L^2(\widehat f)} = 0$ 
as well as $\|\varphi\|_{H^{s+1}(\widehat{f})} + 
\|z\|_{H^{s+1}(\widehat{f})} \lesssim \|{\mathbf v}\|_{{\mathbf H}^s(\widehat f)}$. 
%where $\varphi, z\in H^{s+1}(\widehat{f})$ are determined by the following equations:%
%\begin{subequations}
%\begin{align}
%-\Delta\varphi &  =-\operatorname{div}{\mathbf{v}},\quad\varphi
%=0\quad\text{on }\partial \widehat{f},
%\label{eq:lemma:picurl-negative-I-20a}\\
%-\Delta z  &  =\operatorname{curl}{\mathbf{v}},\quad\partial_{n}%
%z=-{\mathbf{t}}\cdot\operatorname{\mathbf{curl}}z=-{\mathbf{t}}%
%\cdot({\mathbf{v}}-\nabla\varphi)\quad\text{on }\partial \widehat{f},\quad\int_{\widehat{f}}z=0.
%\label{eq:lemma:picurl-negative-I-20b}%
%\end{align}
%Here, ${\mathbf{t}}$ denotes the unit tangent vector on $\partial \widehat{f}$ oriented
%such that $\widehat{f}$ is \textquotedblleft on the left\textquotedblright. We note that
%(\ref{eq:lemma:picurl-negative-I-20b}) is a Neumann problem; integration by
%parts (cf.~\eqref{eq:2d-stokes}) 
%shows that the solvability condition is satisfied. By Lemma~\ref{lemma:shift-theorem}
%we have 
%the {\sl a priori} estimates 
%\end{subequations}
%\begin{equation}
%\Vert\varphi\Vert_{H^{s+1}(\widehat{f})}\lesssim\Vert\operatorname{div}{\mathbf{v}}%
%\Vert_{H^{s-1}(\widehat{f})},\qquad\Vert z\Vert_{H^{s+1}(\widehat{f})}\lesssim\Vert{\mathbf{v}}%
%\Vert_{{\mathbf{H}}^{s}(\widehat{f})}. \label{eq:lemma:picurl-negative-I-30}%
%\end{equation}
Integration by parts (cf.~\eqref{eq:2d-stokes}) yields 
\begin{align}
\label{eq:lemma:picurl-negative-I-100}
({\mathbf{E}},{\mathbf{v}})_{L^{2}(\widehat{f})}&=
({\mathbf{E}},\nabla \varphi)_{L^{2}(\widehat{f})}+({\mathbf{E}},\operatorname{\mathbf{curl}}z)_{L^{2}(\widehat{f})} \\
\nonumber 
 &= 
({\mathbf{E}},\nabla \varphi)_{L^{2}(\widehat{f})}
+ (\operatorname{curl}{\mathbf{E}},z)_{L^{2}(\widehat{f})}-\int_{\partial \widehat{f}}z{\mathbf{E}%
}\cdot{\mathbf{t}}. 
\end{align}
We estimate each of the three terms separately.

\emph{3rd~step:} Using the orthogonalities satisfied by ${\mathbf{E}}$ and
$\varphi\in H_{0}^{1}(\widehat{f})\cap H^{s+1}(\widehat{f})$ we obtain
for the first term in (\ref{eq:lemma:picurl-negative-I-100})
\begin{align*}
\bigl| 
({\mathbf{E}},\nabla\varphi)_{L^{2}(\widehat{f})}\bigr|
&=
\bigl| \inf_{w\in\mathring{W}_{p+1}%
(\widehat{f})}({\mathbf{E}},\nabla(\varphi-w))_{L^{2}(\widehat{f})} \bigr|
\lesssim p^{-s}
%\Vert \operatorname{div}{\mathbf{v}}\Vert_{H^{s-1}(\widehat{f})}
\|\varphi\|_{H^{s+1}(\widehat f)} 
\Vert{\mathbf{E}}\Vert
_{L^{2}(\widehat{f})}\\
& \lesssim p^{-s}\Vert
{\mathbf{v}}\Vert_{\mathbf{H}^s(\widehat{f})}\Vert{\mathbf{E}}\Vert
_{\mathbf{H}(\widehat{f},\operatorname*{curl})}.
\end{align*}

\emph{4th~step:} The term $(z,{\mathbf{E}}\cdot{\mathbf{t}})_{L^{2}(\partial
\widehat{f})}$ in (\ref{eq:lemma:picurl-negative-I-100}) 
can be treated using the orthogonalities satisfied by ${\mathbf{E}}$:
Using that $z\in H^{s+1}(\widehat{f})$ so that $z\in C(\partial \widehat{f})$ and $z\in H^{s+1/2}(e)$
for each edge $e\in{\mathcal{E}}(\widehat{f})$ and the orthogonality properties
(\ref{eq:lemma:picurl-negative-I-orth-c}), 
(\ref{eq:lemma:picurl-negative-I-orth-d}), we get
\begin{align*}
\left\vert \int_{\partial \widehat{f}}{\mathbf{E}}\cdot{\mathbf{t}}z\right\vert  &
\!=\!\!\inf_{w\in W_{p}(\widehat{f})}\!\left\vert \int_{\partial\widehat{f}}{\mathbf{E}}%
\cdot{\mathbf{t}}(z-w)\right\vert \lesssim\Vert{\mathbf{E}}\cdot{\mathbf{t}%
}\Vert_{H^{-1/2}(\partial \widehat{f})}\!\inf_{w\in W_{p}(\widehat{f})}\!\Vert z-w\Vert
_{H^{1/2}(\partial \widehat{f})}\\
&  \lesssim p^{-s}\Vert{\mathbf{E}}\cdot{\mathbf{t}}\Vert_{H^{-1/2}(\partial
\widehat{f})}\Vert z\Vert_{H^{s+1}(\widehat{f})}\lesssim p^{-s}\Vert{\mathbf{E}}\Vert
_{\mathbf{H}(\widehat{f},\operatorname{curl})}\Vert \mathbf{v}\Vert_{\mathbf{H}^{s}(\widehat{f})},
\end{align*}
where, in the final step, we used the continuity of the tangential trace map, 
i.e., 
$\Vert{\mathbf{E}}\cdot{\mathbf{t}}\Vert_{H^{-1/2}(\partial \widehat{f})}\lesssim
\Vert{\mathbf{E}}\Vert_{\mathbf{H}(\widehat{f},\operatorname{curl})}$ 
(cf., e.g.,  \cite[{eq. (154)}]{demkowicz08}). 

\emph{5th~step:} For the
second term in (\ref{eq:lemma:picurl-negative-I-100}), we introduce an auxiliary function ${\mathbf{z}}$ with the
following key properties:
\[
\operatorname{curl}{\mathbf{z}}=z,\quad{\mathbf{z}}\cdot{\mathbf{t}}=0. 
\]
Such a function can be obtained as ${\mathbf{z}}%
=\operatorname{\mathbf{curl}}\widetilde{z}$, where $\widetilde{z}$
solves the following Neumann problem (note that the solvability condition is satisfied 
since $\int_{\widehat{f}}z=0$)
\[
-\Delta\widetilde{z}=z\quad\mbox{ in }\widehat{f},\qquad\partial_{n}\widetilde{z}%
=0\quad\mbox{ on }\partial \widehat{f}.
\]
By Lemma~\ref{lemma:shift-theorem}, $\|\widetilde z\|_{H^{s+1}(\widehat f)} 
\lesssim \|z\|_{H^{s-1}(\widehat f)}$. Thus,  
$\|{\mathbf z}\|_{{\mathbf H}^s(\operatorname*{curl},\widehat f)} 
%\lesssim \|z\|_{H^{s-1}(\widehat f)}
\lesssim \|z\|_{H^{s+1}(\widehat f)}$
and
\begin{align*}
(\operatorname{curl}{\mathbf{E}},z)_{L^{2}(\widehat{f})}  &  =(\operatorname{curl}%
{\mathbf{E}},\operatorname{curl}{\mathbf{z}})_{L^{2}(\widehat{f})}%
\overset{\text{(\ref{eq:lemma:picurl-negative-I-orth-a})}}{=}\inf
_{{\mathbf{w}}\in\mathring{\mathbf{Q}}_{p}(\widehat{f})}(\operatorname{curl}{\mathbf{E}%
},\operatorname{curl}({\mathbf{z}}-{\mathbf{w}}))_{L^{2}(\widehat{f})}.
\end{align*}
Now note that there exists a continuous, polynomial-preserving lifting
\begin{align*}
\boldsymbol{\mathcal{L}}^{\operatorname*{curl},2d}:H^{-1/2}(\partial \widehat{f}) \rightarrow {\mathbf{H}}(\widehat{f},\operatorname*{curl})
\end{align*}
that is in $p$ uniformly bounded, cf. \cite{ainsworth-demkowicz09} and \cite[eq.~(164)]{demkowicz08}. Hence, it follows
\begin{align*}
\boldsymbol{\mathcal{L}}^{\operatorname*{curl},2d}(P^{\operatorname*{curl},2d}{\mathbf{z}} \cdot \mathbf{t}) - P^{\operatorname*{curl},2d}{\mathbf{z}} \in \mathring{\mathbf{Q}}_p(\widehat{f}).
\end{align*}
Since
\begin{align*}
\|\boldsymbol{\mathcal{L}}^{\operatorname*{curl},2d}(P^{\operatorname*{curl},2d}{\mathbf{z}} \cdot \mathbf{t})\|_{\mathbf{H}(\widehat{f},\operatorname{curl})} &\lesssim \|P^{\operatorname*{curl},2d}{\mathbf{z}} \cdot \mathbf{t}\|_{H^{-1/2}(\partial\widehat{f})} = \|(\mathbf{z}-P^{\operatorname*{curl},2d}{\mathbf{z}}) \cdot \mathbf{t}\|_{H^{-1/2}(\partial\widehat{f})} \\
&\lesssim \|\mathbf{z} - P^{\operatorname*{curl},2d}{\mathbf{z}}\|_{\mathbf{H}(\widehat{f},\operatorname{curl})}
\end{align*}
by the properties of the lifting operator and a trace inequality, cf. \cite[eq. (154)]{demkowicz08}, we get with Lemma~\ref{lemma:Pcurl2d} and Lemma~\ref{lemma:helmholtz-like-decomp-2d-v2}
\begin{align}
\label{eq:lemma:picurl-negative-I-IMPORTANT}
&(\operatorname{curl}{\mathbf{E}},z)_{L^{2}(\widehat{f})} \leq \|\operatorname{curl}\mathbf{E}\|_{L^2(\widehat{f})} \|\mathbf{z}-(P^{\operatorname*{curl},2d}{\mathbf{z}} - \boldsymbol{\mathcal{L}}^{\operatorname*{curl},2d}(P^{\operatorname*{curl},2d}{\mathbf{z}} \cdot \mathbf{t}))\|_{\mathbf{H}(\widehat{f},\operatorname{curl})} \\
&\qquad \lesssim \|\mathbf{E}\|_{\mathbf{H}(\widehat{f},\operatorname{curl})} \left(\|\mathbf{z}-P^{\operatorname*{curl},2d}{\mathbf{z}}\|_{\mathbf{H}(\widehat{f},\operatorname{curl})} + \|\boldsymbol{\mathcal{L}}^{\operatorname*{curl},2d}(P^{\operatorname*{curl},2d}{\mathbf{z}} \cdot \mathbf{t})\|_{\mathbf{H}(\widehat{f},\operatorname{curl})}\right) \\
&\qquad \stackrel{\text{Lem.~\ref{lemma:Pcurl2d}}}{\lesssim} \|\mathbf{E}\|_{\mathbf{H}(\widehat{f},\operatorname{curl})} p^{-s} \Vert{\mathbf{z}}\Vert_{\mathbf{H}^{s}(\widehat{f},\operatorname{curl})} \\
&\qquad \stackrel{\text{Lem.~\ref{lemma:helmholtz-like-decomp-2d-v2}}}{\lesssim} p^{-s} \Vert\mathbf{E}\Vert_{\mathbf{H}(\widehat{f},\operatorname{curl})}\Vert{\mathbf{v}} \Vert_{{\mathbf{H}}^{s}(\widehat{f})}.
\end{align}
\end{proof}

\begin{lemma}
\label{lemma:picurl-negative-II} 
Let ${\mathbf{E}}\in\mathbf{H}^{1/2}(\widehat{f},\operatorname*{curl})$ satisfy (\ref{eq:lemma:picurl-negative-I-orth-a}%
), (\ref{eq:lemma:picurl-negative-I-orth-d}). Then, for $s\in [0,\widehat{s})$, there holds 
%\begin{equation}
$\displaystyle 
\Vert\operatorname{curl}{\mathbf{E}}\Vert_{\widetilde{H}^{-s}(\widehat{f})}%
%=\sup_{v\in H^{s}(\widehat{f})}\frac{(\operatorname{curl}{\mathbf{E}},v)_{L^{2}(\widehat{f})}%
%}{\Vert v\Vert_{H^{s}(\widehat{f})}}
\leq C_s p^{-s}\Vert\operatorname{curl}{\mathbf{E}%
}\Vert_{L^{2}(\widehat{f})}. %\label{eq:lemma:picurl-negative-II-10}%
$
%\end{equation}

\end{lemma}

\begin{proof} 
%\neu{As in the proof of Lemma~\ref{lemma:picurl-negative-I}, we mention
%that for any fixed, convex triangle $\widehat f$ the estimate is valid in the 
%restricted range $s \in [0,1]$.}

As in the proof of Lemma~\ref{lemma:picurl-negative-I}, we restrict to $s \ge 1$ since again
the case $s = 0$ is trivial and we argue by interpolation for $s \in (0,1)$.  
Let $v\in H^{s}(\widehat{f})$ and $\overline{v}:=(\int_{\widehat{f}}v)/|\widehat{f}|\in{\mathbb{R}}$ be its
average. Integration by parts yields
\[
(\operatorname{curl}{\mathbf{E}},v)_{L^{2}(\widehat{f})}=(\operatorname{curl}%
{\mathbf{E}},v-\overline{v})_{L^{2}(\widehat{f})}+\overline{v}({\mathbf{E}}%
\cdot{\mathbf{t}},1)_{L^{2}(\partial \widehat{f})}%
\overset{(\ref{eq:lemma:picurl-negative-I-orth-d})}{=}(\operatorname{curl}%
{\mathbf{E}},v-\overline{v})_{L^{2}(\widehat{f})}.
\]
Next, we define the auxiliary function $\varphi\in H^{1}(\widehat{f})$ as
the solution of 
\[
-\Delta\varphi=v-\overline{v}\quad\mbox{ in }\widehat{f},\qquad\partial_{n}%
\varphi=0\quad\mbox{ on }\partial \widehat{f},
\] 
and set ${\mathbf{v}}:=\operatorname{\mathbf{curl}}\varphi$. 
We note that Lemma~\ref{lemma:shift-theorem} implies 
$\|\varphi\|_{H^{s+1}(\widehat{f})} \lesssim \|v-\overline{v}\|_{H^{s-1}(\widehat f)}
\lesssim \|v\|_{H^s(\widehat f)}$.
We observe
%-----
$\operatorname{curl}{\mathbf{v}}=-\Delta
\varphi=v-\overline{v}$ in $\widehat{f}$ and ${\mathbf{t}}\cdot{\mathbf{v}}%
=-\partial_{n}\varphi=0$ on $\partial \widehat{f}$ so that integration by parts gives
\begin{align*}
(\operatorname{curl}{\mathbf{E}},v-\overline{v})_{L^{2}(\widehat{f})}  &
=(\operatorname{curl}{\mathbf{E}},\operatorname{curl}{\mathbf{v}})_{L^{2}%
(\widehat{f})}\overset{(\ref{eq:lemma:picurl-negative-I-orth-a})}{=}\inf_{{\mathbf{w}%
}\in\mathring{\mathbf{Q}}_{p}(\widehat{f})}(\operatorname{curl}{\mathbf{E}}%
,\operatorname{curl}({\mathbf{v}}-{\mathbf{w}}))_{L^{2}(\widehat{f})}\\
&  
\!\!\!
\!\!\!
\!\!\!
\stackrel{\text{Lem.~\ref{lemma:Pcurl2d}}}{\lesssim} 
\!\!\!\!\!
p^{-s}\Vert\operatorname{curl}{\mathbf{E}}\Vert_{L^{2}(\widehat{f})}%
\Vert{\mathbf{v}}\Vert_{\mathbf{H}^{s}(\widehat{f},\operatorname{curl})}\lesssim p^{-s}%
\Vert\operatorname{curl}{\mathbf{E}}\Vert_{L^{2}(\widehat{f})}\Vert v\Vert_{H^{s}(\widehat{f})}.
\qedhere
\end{align*}
%which concludes the proof.
\end{proof}

After the next lemma about approximation on edges $e\in\mathcal{E}(\widehat{f})$, we can prove the stability results in 2D stated in Theorem~\ref{thm:projection-based-interpolation-2d}.

\begin{lemma}
\label{lemma:Picurl-edge} For each edge $e\in{\mathcal{E}}(\widehat{f})$ we
have for ${\mathbf{u}}\in{\mathbf{H}}^{1/2}(\widehat{f},\operatorname*{curl})$  and $s \ge 0$ 
\begin{equation}
\Vert({\mathbf{u}}-\hatPicurlcomtwod{\mathbf{u}%
})\cdot{\mathbf{t}}_{e}\Vert_{\widetilde{H}^{-s}(e)}\leq C_{s}  p^{-s}%
\inf_{v \in {\mathcal P}_p(e)} \Vert{\mathbf{u}}\cdot{\mathbf{t}}_{e} - v\Vert_{L^{2}(e)}. 
\label{eq:lemma:Picurl-edge-20}%
\end{equation}
\end{lemma}

\begin{proof}
We point out that Lemma~\ref{lemma:traces-2d} shows that ${\mathbf u} \cdot {\mathbf t}_e \in L^2(e)$. 

We recall that on edges, the operator $\hatPicurlcomtwod$ is simply the $L^{2}$-projection. Thus, (\ref{eq:lemma:Picurl-edge-20})
holds for $s=0$. For $s>0$, (\ref{eq:lemma:Picurl-edge-20}) is shown by a
standard duality argument. 
Let $\widetilde{e}:=\left(  {\mathbf{u}%
}-\hatPicurlcomtwod{\mathbf{u}}\right)
\cdot{\mathbf{t}}_{e}$ be the error and $v\in H^{s}(e)$. 
Identifying $e$ with the interval $(0,L)$ of length 
$L = \operatorname{diam} e$, we observe that  
a function $w\in \mathcal{P}_p(e)$ 
can be decomposed into $w(x)=\overline{w}+\left(\int_0^x w(t)\,dt-x \overline{w}\right)^\prime$, 
where $\overline{w}$ denotes the average of $w$ on $e$. Note that 
$\int_0^x w(t)\,dt - x \overline{w} \in \mathring W_{p+1}(e)$. 
Hence, $(\widetilde{e},w)_{L^2(e)} = 0$ 
by \eqref{eq:Pi_curl-f} and \eqref{eq:Pi_curl-e}, and we obtain
\begin{align*}
(\widetilde{e},v)_{L^{2}(e)}&=
\inf_{w\in \mathcal{P}_p(e)} (\widetilde{e},v-w)_{L^{2}(e)} 
\leq \Vert\widetilde{e}\Vert_{L^2(e)} 
 \inf_{w\in{\mathcal{P}}_{p}(e)} \Vert v-w\Vert _{L^{2}(e)} \\
& \lesssim p^{-s} \Vert\widetilde{e}\Vert_{L^{2}(e)} \Vert v\Vert_{H^{s}(e)}.
\qedhere
\end{align*}
\end{proof}

\begin{lemma}
\label{lemma:Picurl-face}
For ${\mathbf{u}}\in{\mathbf{H}}^{1/2}(\widehat{f},\operatorname*{curl})$ there holds 
for $s \in [0,\widehat{s})$
\[
\Vert{\mathbf{u}}-\hatPicurlcomtwod{\mathbf{u}}\Vert_{\widetilde{\mathbf{H}}^{-s}(\widehat{f},\operatorname{curl})}
\leq C_s  p^{-(1/2+s)}%
\inf_{{\mathbf v} \in {\mathbf Q}_p(\widehat f)} 
\Vert{\mathbf{u} - \mathbf{v} }\Vert_{\mathbf{H}^{1/2}(\widehat{f},\operatorname*{curl})}. 
\]

\end{lemma}

\begin{proof}
By the projection property of $\hatPicurlcomtwod$, it suffices to show the bound 
with ${\mathbf v} = 0$ in the infimum.

\emph{1st~step:} 
As
discussed in \cite[Sec.~{4.2}]{demkowicz08} (which relies on
\cite{ainsworth-demkowicz09}) there is a polynomial-preserving lifting
$\boldsymbol{\mathcal{L}}^{\operatorname*{curl},2d}\!:\!H^{-1/2}(\partial \widehat{f})\!\rightarrow\!
{\mathbf{H}}(\widehat{f},\operatorname*{curl})$ that is uniformly (in $p$) bounded.

\emph{2nd~step:} Let $P^{\operatorname*{curl},2d}{\mathbf{u}}$ be the
polynomial best approximation of Lemma~\ref{lemma:Pcurl2d}. Following the
procedure suggested in \cite{demkowicz08}, we define
\[
{\mathbf{E}}:=P^{\operatorname*{curl},2d}{\mathbf{u}}-\hatPicurlcomtwod{\mathbf{u}}
\in \mathbf{Q}_p(\widehat{f}).
\]
Note ${\mathbf{E}}-\boldsymbol{\mathcal{L}}^{\operatorname*{curl},2d}({\mathbf{E}}\cdot \mathbf{t})\in\mathring{\mathbf{Q}}_p(\widehat{f})$. We get from the
orthogonalities (\ref{eq:Pi_curl-c}) and (\ref{eq:lemma:Pcurl2d-a})
\begin{equation}
(\operatorname{curl}({\mathbf{E}}-\boldsymbol{\mathcal{L}}^{\operatorname*{curl},2d}%
({\mathbf{E}}\cdot\mathbf{t})),\operatorname{curl}{\mathbf{E}})_{L^{2}(\widehat{f})}=0. 
\label{eq:lemma:Picurl-face-10}%
\end{equation}
Hence,
\begin{align*}
\Vert\operatorname{curl}{\mathbf{E}}\Vert_{L^{2}(\widehat{f})}^{2}  
%&  =\left(
%\operatorname{curl}({\mathbf{E}}-\boldsymbol{\mathcal{L}}^{\operatorname*{curl}%
%,2d}({\mathbf{E}}\cdot\mathbf{t})),\operatorname{curl}{\mathbf{E}}\right)
%_{L^{2}(\widehat{f})}\\
&\stackrel{\eqref{eq:lemma:Picurl-face-10}}{=}
\left(  \operatorname{curl}\boldsymbol{\mathcal{L}}^{\operatorname*{curl},2d%
}({\mathbf{E}}\cdot\mathbf{t}),\operatorname{curl}{\mathbf{E}}\right)
_{L^{2}(\widehat{f})}\\
& \leq  
\Vert
\operatorname{curl}\boldsymbol{\mathcal{L}}^{\operatorname*{curl},2d}({\mathbf{E}}\cdot\mathbf{t})\Vert_{L^{2}(\widehat{f})}\Vert\operatorname{curl}{\mathbf{E}}\Vert_{L^{2}(\widehat{f})},
\end{align*}
from which we obtain with the stability properties of the lifting operator $\boldsymbol{\mathcal{L}}^{\operatorname*{curl},2d}$
\begin{equation}
\Vert\operatorname{curl}{\mathbf{E}}\Vert_{L^{2}(\widehat{f})}\lesssim\Vert{\mathbf{E}}\cdot\mathbf{t}\Vert_{H^{-1/2}(\partial \widehat{f})}.
\label{eq:lemma:Picurl-face-20}%
\end{equation}

\emph{3rd~step:} With the discrete Friedrichs inequality of
Lemma~\ref{lemma:discrete-friedrichs},
(\ref{item:lemma:discrete-friedrichs-ii}) we arrive at
\begin{align}
\label{eq:lemma:Picurl-face-30}
\Vert{\mathbf{E}}\Vert_{L^{2}(\widehat{f})}  
&  \leq\Vert{\mathbf{E}}-\boldsymbol{\mathcal{L}%
}^{\operatorname*{curl},2d}({\mathbf{E}}\cdot\mathbf{t})\Vert_{L^{2}(\widehat{f})}%
+\Vert\boldsymbol{\mathcal{L}}^{\operatorname*{curl},2d}({\mathbf{E}}\cdot\mathbf{t})\Vert
_{L^{2}(\widehat{f})}\\
& \lesssim\Vert\operatorname{curl}({\mathbf{E}}-\boldsymbol{\mathcal{L}%
}^{\operatorname*{curl},2d}({\mathbf{E}}\cdot\mathbf{t}))\Vert_{L^{2}(\widehat{f})}%
+\Vert\boldsymbol{\mathcal{L}}^{\operatorname*{curl},2d}({\mathbf{E}}\cdot\mathbf{t})\Vert
_{L^{2}(\widehat{f})}\nonumber \\
&  \lesssim\Vert\operatorname{curl}{\mathbf{E}}\Vert_{L^{2}(\widehat{f})}%
+\Vert\boldsymbol{\mathcal{L}}^{\operatorname*{curl},2d}({\mathbf{E}}\cdot\mathbf{t})\Vert_{\mathbf{H}(\widehat{f},\operatorname{curl})}\overset{\text{(\ref{eq:lemma:Picurl-face-20}%
)}}{\lesssim}\Vert{\mathbf{E}}\cdot\mathbf{t}\Vert_{H^{-1/2}(\partial \widehat{f})}.
\nonumber 
\end{align}

\emph{4th~step:} 
Lemmas~\ref{lemma:Pcurl2d}, \ref{lemma:Picurl-edge} and $\|{\mathbf u} \cdot {\mathbf t}_e \|_{L^2(e)} 
\lesssim \|{\mathbf u}\|_{{\mathbf H}^{1/2}(\widehat f)}$ (cf.~Lemma~\ref{lemma:traces-2d}) give
\begin{align}
\nonumber 
\Vert{\mathbf{u}}-\hatPicurlcomtwod{\mathbf{u}}%
\Vert_{\mathbf{H}(\widehat{f},\operatorname{curl})}  &  \lesssim\Vert{\mathbf{u}}%
-P^{\operatorname*{curl},2d}{\mathbf{u}}\Vert_{\mathbf{H}(\widehat{f},\operatorname{curl})}%
+\Vert{\mathbf{E}}\Vert_{\mathbf{H}(\widehat{f},\operatorname{curl})} \\
\nonumber 
& \overset{\text{(\ref{eq:lemma:Picurl-face-20}),(\ref{eq:lemma:Picurl-face-30}%
)}}{\lesssim}\Vert{\mathbf{u}}-P^{\operatorname*{curl},2d}{\mathbf{u}}%
\Vert_{\mathbf{H}(\widehat{f},\operatorname{curl})}+\Vert{\mathbf{E}}\cdot\mathbf{t}\Vert
_{H^{-1/2}(\partial \widehat{f})}\\
\nonumber 
&  \lesssim\Vert{\mathbf{u}}-P^{\operatorname*{curl},2d}{\mathbf{u}}%
\Vert_{\mathbf{H}(\widehat{f},\operatorname{curl})}\!+\!\Vert(({\mathbf{u}}-\hatPicurlcomtwod{\mathbf{u}})\cdot\mathbf{t})\Vert_{H^{-1/2}(\partial
\widehat{f})} \\
\label{eq:lemma:Picurl-face-200}%
& \overset{\text{Lem.~\ref{lemma:Pcurl2d},~\ref{lemma:Picurl-edge}%
}}{\lesssim}p^{-1/2}\Vert{\mathbf{u}}\Vert_{{\mathbf{H}}^{1/2}(\widehat{f},\operatorname*{curl})}.
\end{align}
\emph{5th~step:} 
{}From Lemmas~\ref{lemma:picurl-negative-I} and \ref{lemma:picurl-negative-II} 
we get 
\begin{align*}
\Vert{\mathbf{u}}-\hatPicurlcomtwod{\mathbf{u}}\Vert_{\widetilde{\mathbf{H}}^{-s}(\widehat{f},\operatorname{curl})} \!\lesssim\! p^{-s} \Vert{\mathbf{u}}-\hatPicurlcomtwod{\mathbf{u}}%
\Vert_{\mathbf{H}(\widehat{f},\operatorname{curl})} 
\!\!\stackrel{(\ref{eq:lemma:Picurl-face-200})}{\lesssim}\!\!
p^{-(1/2+s)}%
\Vert{\mathbf{u}}\Vert_{\mathbf{H}^{1/2}(\widehat{f},\operatorname*{curl})},
\end{align*}
which concludes the proof.
\end{proof}

If $\operatorname{curl}{\mathbf u}$ is a polynomial, we have the following result.

\begin{lemma}
\label{lemma:better-regularity-2d}
%\neu{Let $\widehat f$ be the equilateral reference triangle.}
For all $k\geq1$ and
all ${\mathbf{u}}\in{\mathbf{H}}^{k}(\widehat{f})$ with $\operatorname*{curl}{\mathbf{u}}\in{\mathcal{P}}_{p}(\widehat{f})$ there holds
\begin{equation}
\Vert{\mathbf{u}}-\hatPicurlcomtwod{\mathbf{u}}%
\Vert_{\widetilde{\mathbf{H}}^{-s}(\widehat{f},\operatorname{curl})}\leq C_{s,k}p^{-(k+s)}\Vert{\mathbf{u}}\Vert
_{\mathbf{H}^{k}(\widehat{f})}, \qquad s\in [0,\widehat{s}).
\label{eq:proposition:better-regularity-2d}%
\end{equation}
If $p\geq k-1$, then the full norm $\Vert{\mathbf{u}}\Vert_{{\mathbf{H}}%
^{k}(\widehat{f})}$ can be replaced with the seminorm $|{\mathbf{u}%
}|_{{\mathbf{H}}^{k}(\widehat{f})}$.
\end{lemma}

\begin{proof}
We employ the regularized right inverses of the operators $\nabla$ and
$\operatorname*{\mathbf{curl}}$ and proceed as in \cite[Lemma~{5.8}]{hiptmair08}. We
write, using the decomposition of Lemma~\ref{lemma:helmholtz-like-decomp-2d},
\[
{\mathbf{u}}=\nabla R^{\operatorname*{grad}}({\mathbf{u}}-{\mathbf{R}%
}^{\operatorname*{curl}}\operatorname*{curl}{\mathbf{u}})+{\mathbf{R}}^{\operatorname*{curl}}\operatorname*{curl}{\mathbf{u}}=:\nabla
\varphi+{\mathbf{v}}%
\]
with $\varphi\in H^{k+1}(\widehat{f})$ and ${\mathbf{v}}\in{\mathbf{H}}%
^{k}(\widehat{f})$ together with
\begin{equation}
\Vert\varphi\Vert_{H^{k+1}(\widehat{f})}+\Vert{\mathbf{v}}\Vert_{{\mathbf{H}%
}^{k}(\widehat{f})}\leq C\left(  \Vert{\mathbf{u}}\Vert_{{\mathbf{H}}%
^{k}(\widehat{f})}+\Vert\operatorname*{curl}{\mathbf{u}}\Vert_{{H%
}^{k-1}(\widehat{f})}\right)  \leq C\Vert{\mathbf{u}}\Vert_{{\mathbf{H}}%
^{k}(\widehat{f})}.
\label{eq:lemma:projection-based-interpolation-approximation-100-2d}%
\end{equation}
The assumption $\operatorname*{curl}{\mathbf{u}}\in {\mathcal P}_p(\widehat f)$
and %\cite[(3.4)]{hiptmair08} 
Lemma~\ref{lemma:mcintosh-2d}, (\ref{item:lemma:mcintosh-2d-iv}) 
imply ${\mathbf{v}}={\mathbf{R}}%
^{\operatorname*{curl}}\operatorname*{curl}{\mathbf{u}}\in
\mathbf{Q}_p(\widehat{f})$; since
$\hatPicurlcomtwod$ is a projection, we conclude
${\mathbf{v}}-\hatPicurlcomtwod{\mathbf{v}}=0$. Thus,
with the commuting diagram property $\nabla\hatPigradcomtwod=\hatPicurlcomtwod\nabla$ 
and the bound (\ref{eq:lemma:demkowicz-grad-2D-10-c}) we get
\begin{align*}
\Vert(\operatorname{I}-\hatPicurlcomtwod){\mathbf{u}}\Vert
_{\widetilde{\mathbf{H}}^{-s}(\widehat{f},\operatorname{curl})} &= \Vert(\operatorname{I}-\hatPicurlcomtwod)\nabla\varphi+\underbrace{(\operatorname{I}-\hatPicurlcomtwod){\mathbf{v}}}_{=0}\Vert_{\widetilde{\mathbf{H}}^{-s}(\widehat{f},\operatorname{curl})} \\
&= \Vert\nabla(\operatorname{I}-\hatPigradcomtwod)\varphi\Vert_{\widetilde{\mathbf{H}}^{-s}(\widehat{f})}
\stackrel{(\ref{eq:lemma:demkowicz-grad-2D-10-c})}{\lesssim} p^{-(k+s)}\Vert\varphi\Vert_{H^{k+1}(\widehat{f})}.
\end{align*}
The proof of (\ref{eq:proposition:better-regularity-2d})
is complete in view of
(\ref{eq:lemma:projection-based-interpolation-approximation-100-2d}). Replacing
$\Vert{\mathbf{u}}\Vert_{{\mathbf{H}}^{k}(\widehat{f})}$ with $|{\mathbf{u}%
}|_{{\mathbf{H}}^{k}(\widehat{f})}$ is possible since the
projector $\hatPicurlcomtwod$ reproduces polynomials
of degree $p$.
\end{proof}

%-----------------------------------

\section{Stability of the projection operators in three space dimensions}

\subsection{Preliminaries}

For the approximation properties of $\hatPigradcom$, 
we need the following approximation results.

\begin{lemma}
[\!\!\protect{\cite[Thm.~{5.2}]{demkowicz08}}]
\label{lemma:Pgrad3d}
Let $P^{\operatorname*{grad},3d}u\in W_{p+1}(\widehat{K})$ be defined by the conditions
\begin{subequations}
\label{eq:lemma:Pgrad3d}
\begin{align}
\label{eq:lemma:Pgrad3d-a}
(\nabla(u-P^{\operatorname*{grad},3d}u),\nabla v)_{L^2(\widehat{K})} &= 0 \qquad \forall v\in W_{p+1}(\widehat{K}),\\
(u-P^{\operatorname*{grad},3d}u,1)_{L^2(\widehat{K})} &= 0.
\label{eq:lemma:Pgrad3d-b}
\end{align}
\end{subequations}
Then, for $r>1$, there holds 
%\begin{align}
$\displaystyle 
\Vert u-P^{\operatorname*{grad},3d}u\Vert_{H^1(\widehat{K})} \leq C_r p^{-(r-1)} \Vert u\Vert_{H^r(\widehat{K})}.
$
%\end{align}
\end{lemma}

\begin{lemma}
[\!\!\protect{\cite{demkowicz-buffa05},\cite[Thm.~{5.2}]{demkowicz08}}]
\label{lemma:Pcurl3d} Let $P^{\operatorname*{curl},3d}{\mathbf{u}}%
\in\mathbf{Q}_p(\widehat{K})$ be defined
by 
\begin{subequations}\label{eq:lemma:Pcurl3d}
\begin{align}
(\operatorname{\mathbf{curl}}({\mathbf{u}}-P^{\operatorname*{curl},3d}{\mathbf{u}%
}),\operatorname{\mathbf{curl}}{\mathbf{v}})_{L^{2}(\widehat{K})}  &  =0\qquad
\forall{\mathbf{v}}\in\mathbf{Q}_p(\widehat{K}),
\label{eq:lemma:Pcurl3d-a}\\
({\mathbf{u}}-P^{\operatorname*{curl},3d}{\mathbf{u}},\nabla v)_{L^{2}%
(\widehat{K})}  &  =0\qquad\forall v\in W_{p+1}(\widehat{K}). 
\label{eq:lemma:Pcurl3d-b}%
\end{align}
\end{subequations}
Then, for $r>0$, there holds 
%\begin{align}
$\displaystyle 
\Vert{\mathbf{u}}-P^{\operatorname*{curl},3d}{\mathbf{u}}\Vert_{{\mathbf{H}%
}(\widehat{K},\operatorname{\mathbf{curl}})}\leq C_{r}p^{-r}\Vert{\mathbf{u}}%
\Vert_{{\mathbf{H}}^{r}(\widehat{K},\operatorname{\mathbf{curl}})}. $
%\end{align}
\end{lemma}

\begin{lemma}
[\!\!\protect{\cite[Thm.~{5.2}]{demkowicz08}}]
\label{lemma:Pdiv3d} Let $P^{\operatorname*{div},3d}{\mathbf{u}}%
\in\mathbf{V}_{p}(\widehat{K})$ be defined
by the conditions
\begin{subequations}
\begin{align}
\label{eq:lemma:Pdiv3d}
(\operatorname{div}({\mathbf{u}}-P^{\operatorname*{div},3d}{\mathbf{u}%
}),\operatorname{div}{\mathbf{v}})_{L^{2}(\widehat{K})}  &  =0\qquad
\forall{\mathbf{v}}\in\mathbf{V}_p(\widehat{K}),\\
({\mathbf{u}}-P^{\operatorname*{div},3d}{\mathbf{u}},\operatorname*{div} \mathbf{v})_{L^{2}(\widehat{K})}  &  =0\qquad\forall \mathbf{v}\in \mathbf{Q}_p(\widehat{K}).
\label{eq:lemma:Pdiv3d-20}
\end{align}
\end{subequations}
Then, for $r > 0$, there holds 
%\begin{equation}
%\label{eq:lemma:Pdiv3d-30}
$\displaystyle 
\Vert{\mathbf{u}}-P^{\operatorname*{div},3d}{\mathbf{u}}\Vert_{{\mathbf{H}%
}(\widehat{K},\operatorname{div})}\leq C_{r}p^{-r}\Vert{\mathbf{u}}%
\Vert_{{\mathbf{H}}^{r}(\widehat{K},\operatorname{div})}.
$
%\end{equation}
\end{lemma}

In the next lemma, right inverses for the differential operators are recalled.

\begin{lemma}
[\!\!{{\cite{costabel-mcintosh10}, see also \cite[Sec.~{2}]{hiptmair08}}}]\label{lemma:mcintosh} Let $B\subset\widehat{K}$ be a ball. Let $\theta\in C_{0}^{\infty}(B)$ with $\int_{B}%
\theta=1$. Define the operators
\begin{align*}
{R}^{\operatorname*{grad}}{\mathbf{u}}(\mathbf{x})  &  :=\int_{{\mathbf{a}}\in B} \theta(\mathbf{a})%
\int_{t=0}^{1}{\mathbf{u}}(\mathbf{a}+t(\mathbf{x}-{\mathbf{a}}))\,dt\cdot(\mathbf{x}-{\mathbf{a}%
})\,d{\mathbf{a}},\\
{\mathbf{R}}^{\operatorname*{curl}}{\mathbf{u}}(\mathbf{x})  &  :=\int_{{\mathbf{a}}\in
B} \theta(\mathbf{a}) \int_{t=0}^{1}t{\mathbf{u}}(\mathbf{a}+t(\mathbf{x}-{\mathbf{a}}))\,dt\times(\mathbf{x}-{\mathbf{a}%
})\,d{\mathbf{a}},\\
{\mathbf{R}}^{\operatorname*{div}}{u}(\mathbf{x})  &  :=\int_{{\mathbf{a}}\in B} \theta(\mathbf{a})%
\int_{t=0}^{1}t^2 u(\mathbf{a}+t(\mathbf{x}-{\mathbf{a}}))\,dt (\mathbf{x}-{\mathbf{a}%
})\,d{\mathbf{a}}.
\end{align*}
Then:

\begin{enumerate}
[(i)]

\item 
\label{item:lemma:mcintosh-i}
For ${\mathbf{u}}$ with $\operatorname*{div}{\mathbf{u}}=0$, there holds
$\operatorname*{\mathbf{curl}}{\mathbf{R}}^{\operatorname*{curl}}{\mathbf{u}%
}={\mathbf{u}}$.

\item 
\label{item:lemma:mcintosh-ii}
For ${\mathbf{u}}$ with $\operatorname*{\mathbf{curl}}{\mathbf{u}}=0$, there
holds $\nabla{R}^{\operatorname*{grad}}{\mathbf{u}}={\mathbf{u}}$.

\item 
\label{item:lemma:mcintosh-iii}
For $u\in L^2(\widehat{K})$, there holds $\operatorname*{div}{\mathbf{R}}^{\operatorname*{div}}u=u$.

\item 
\label{item:lemma:mcintosh-iv}
If ${\mathbf{u}}\in\mathbf{Q}_p(\widehat{K})$,
then ${R}^{\operatorname*{grad}}{\mathbf{u}}\in W_{p+1}(\widehat{K})$.

\item 
\label{item:lemma:mcintosh-v}
If ${\mathbf{u}}\in\mathbf{V}_{p}(\widehat{K})$, then ${\mathbf{R}}^{\operatorname*{curl}%
}{\mathbf{u}}\in\mathbf{Q}_p(\widehat{K})$.

\item 
\label{item:lemma:mcintosh-vi}
If $u\in W_p(\widehat{K})$, then $\mathbf{R}^{\operatorname*{div}}u\in \mathbf{V}_p(\widehat{K})$.

\item 
\label{item:lemma:mcintosh-vii}
For every $k\geq0$, the operators ${R}^{\operatorname*{grad}}$, 
${\mathbf{R}}^{\operatorname*{curl}}$ and ${\mathbf{R}}^{\operatorname*{div}}$ are bounded linear operators
$\mathbf{H}^{k}(\widehat{K})\rightarrow H^{k+1}(\widehat{K})$, $\mathbf{H}^{k}(\widehat{K})\rightarrow \mathbf{H}^{k+1}(\widehat{K})$ and $H^{k}(\widehat{K})\rightarrow \mathbf{H}^{k+1}(\widehat{K})$, respectively.
\end{enumerate}
\end{lemma}

The right inverses of Lemma~\ref{lemma:mcintosh} can be used to construct regular 
Helmholtz-like decompositions of functions in ${\mathbf{H}}^{s}(\widehat{K},\operatorname*{\mathbf{curl}})$ and ${\mathbf{H}}^{s}(\widehat{K},\operatorname*{div})$.

\begin{lemma}
\label{lemma:helmholtz-like-decomp} Let $s \ge0$. Then each ${\mathbf{u}}
\in{\mathbf{H}}^{s}(\widehat{K},\operatorname*{\mathbf{curl}})$ can be written as
${\mathbf{u}} = \nabla\varphi+ {\mathbf{z}}$ with $\varphi\in H^{s+1}%
(\widehat{K})$, ${\mathbf{z}} \in{\mathbf{H}}^{s+1}(\widehat{K})$ satisfying $\Vert\varphi\Vert_{H^{s+1}(\widehat{K})} \lesssim \Vert\mathbf{u}\Vert_{\mathbf{H}^s(\widehat{K},\operatorname{\mathbf{curl}})}$ and $\Vert\mathbf{z}\Vert_{\mathbf{H}^{s+1}(\widehat{K})} \lesssim \Vert\operatorname{\mathbf{curl}}\mathbf{u}\Vert_{\mathbf{H}^s(\widehat{K})}$.
\end{lemma}

\begin{proof}
With the aid of the operators ${\mathbf{R}}^{\operatorname*{curl}}$,
$R^{\operatorname*{grad}}$ of Lemma~\ref{lemma:mcintosh}, we write
$\displaystyle 
{\mathbf{u}}=\nabla R^{\operatorname*{grad}}({\mathbf{u}}-{\mathbf{R}%
}^{\operatorname*{curl}}(\operatorname*{\mathbf{curl}}{\mathbf{u}}))+{\mathbf{R}%
}^{\operatorname*{curl}}(\operatorname*{\mathbf{curl}}{\mathbf{u}})
=:\nabla \varphi + {\mathbf z}.
$
%The mapping properties of ${\mathbf{R}}^{\operatorname*{curl}}$ and $R^{\operatorname*{grad}}$ 
%of Lemma~\ref{lemma:mcintosh} then imply the result. 
The stability properties of the operators $\mathbf{R}^{\operatorname{curl}}$ and $R^{\operatorname{grad}}$ 
give 
\begin{align*}
\Vert\varphi\Vert_{H^{s+1}(\widehat{K})}^2 &\lesssim \Vert\mathbf{u}-\mathbf{R}^{\operatorname{curl}}(\operatorname{\mathbf{curl}}\mathbf{u})\Vert_{\mathbf{H}^s(\widehat{K})}^2 \lesssim \Vert\mathbf{u}\Vert_{\mathbf{H}^s(\widehat{K})}^2 + \Vert\mathbf{R}^{\operatorname{curl}}(\operatorname{\mathbf{curl}}\mathbf{u})\Vert_{\mathbf{H}^{s+1}(\widehat{K})}^2 \\
&\lesssim \Vert\mathbf{u}\Vert_{\mathbf{H}^s(\widehat{K})}^2 + \Vert\operatorname{\mathbf{curl}}\mathbf{u}\Vert_{\mathbf{H}^s(\widehat{K})}^2 = \Vert\mathbf{u}\Vert_{\mathbf{H}^s(\widehat{K},\operatorname{\mathbf{curl}})}^2,\\
\Vert\mathbf{z}\Vert_{\mathbf{H}^{s+1}(\widehat{K})} &= \Vert\mathbf{R}^{\operatorname{curl}}(\operatorname{\mathbf{curl}}\mathbf{u})\Vert_{\mathbf{H}^{s+1}(\widehat{K})} \lesssim \Vert\operatorname{\mathbf{curl}}\mathbf{u}\Vert_{\mathbf{H}^s(\widehat{K})}.
\qedhere
\end{align*}
\end{proof}
\begin{lemma}
\label{lemma:helmholtz-like-decomp-v2} Let $s \in [0,1]$. Then each ${\mathbf{u}}
\in{\mathbf{H}}^{s}(\widehat{K})$ can be written as
${\mathbf{u}} = \nabla\varphi+ \operatorname{\mathbf{curl}}\mathbf{z}$ with $\varphi\in H^{s+1}%
(\widehat{K}) \cap H^1_0(\widehat{K})$, $\mathbf{z} \in \mathbf{H}^{s+1}(\widehat{K})$ and 
$\|\varphi\|_{H^{s+1}(\widehat{K})} + \|\mathbf{z}\|_{{\mathbf H}^{s+1}(\widehat{K})} \lesssim \|\mathbf{u}\|_{\mathbf{H}^s(\widehat{K})}$.  
%satisfying $\|\varphi\|_{H^{s+1}(\widehat{K})} \lesssim \|\mathbf{u}\|_{\mathbf{H}^s(\widehat{K})}$ and $\|\mathbf{z}\|_{\mathbf{H}^{s+1}(\widehat{K})} \lesssim \|\mathbf{u}\|_{\mathbf{H}^s(\widehat{K})}$.
\end{lemma}

\begin{proof} 
We define $\varphi\in H^{1}_0(\widehat{K})$ as the solution of the problem
\begin{align*}
-\Delta\varphi=-\operatorname{div}\mathbf{u}, \quad \varphi=0 \text{ on }\partial\widehat{K}.
\end{align*}
Since $\operatorname{div}:H^1(\widehat{K}) \rightarrow L^2(\widehat{K})$ and 
$\operatorname{div}:L^2(\widehat{K}) \rightarrow (H^1_0(\widehat{K}))^\prime =: H^{-1}(\widehat{K})$
the convexity of $\widehat{K}$ gives 
$\varphi \in H^1_0(\widehat{K})$ if $s = 0$ and 
$\varphi \in H^2(\widehat{K}) \cap H^1_0(\widehat{K})$ if $s = 1$. By interpolation 
$\varphi \in H^{s+1}(\widehat{K}) \cap H^1_0(\widehat{K})$ with 
$\|\varphi \|_{H^{s+1}(\widehat{K})} \lesssim \|{\mathbf u}\|_{{\mathbf H}^s(\widehat K)}$. 
With the operator $\mathbf{R}^{\operatorname{curl}}$ of Lemma~\ref{lemma:mcintosh}, we define $\mathbf{z}:=\mathbf{R}^{\operatorname{curl}}(\mathbf{u}-\nabla\varphi)\in\mathbf{H}^{s+1}(\widehat{K})$. 
Noting $\operatorname{div}(\mathbf{u}-\nabla\varphi)=0$, 
we have ${\mathbf u} = \nabla \varphi+ \operatorname{\mathbf{curl}} {\mathbf z}$ 
by Lemma~\ref{lemma:mcintosh}, (\ref{item:lemma:mcintosh-i}). The stability property of $\mathbf{R}^{\operatorname{curl}}$ gives
\begin{align*}
&\|\mathbf{z}\|_{\mathbf{H}^{s+1}(\widehat{K})} = \|\mathbf{R}^{\operatorname{curl}}(\mathbf{u}-\nabla\varphi)\|_{\mathbf{H}^{s+1}(\widehat{K})} \lesssim \|\mathbf{u}-\nabla\varphi\|_{\mathbf{H}^s(\widehat{K})} \lesssim \|\mathbf{u}\|_{\mathbf{H}^s(\widehat{K})}.
\qedhere
\end{align*}
\end{proof}
\begin{lemma}
\label{lemma:helmholtz-decomposition-div}
Let $s \geq 0$. Then each ${\mathbf{u}}
\in{\mathbf{H}}^{s}(\widehat{K},\operatorname*{div})$ can be written as
${\mathbf{u}} = \operatorname*{\mathbf{curl}} \boldsymbol{\varphi}+ {\mathbf{z}}$ with $\boldsymbol{\varphi}\in \mathbf{H}^{s+1}%
(\widehat{K})$, ${\mathbf{z}} \in{\mathbf{H}}^{s+1}(\widehat{K})$ satisfying $\Vert\boldsymbol{\varphi}\Vert_{\mathbf{H}^{s+1}(\widehat{K})} \lesssim \Vert\mathbf{u}\Vert_{\mathbf{H}^s(\widehat{K},\operatorname*{div})}$ and $\Vert\mathbf{z}\Vert_{\mathbf{H}^{s+1}(\widehat{K})} \lesssim \Vert\operatorname*{div}\mathbf{u}\Vert_{H^s(\widehat{K})}$.
\end{lemma}

\begin{proof}
With the operators $\mathbf{R}^{\operatorname*{curl}}$ and $\mathbf{R}^{\operatorname*{div}}$ of Lemma~\ref{lemma:mcintosh}, we write
\begin{align*}
{\mathbf{u}}=\operatorname*{\mathbf{curl}} \mathbf{R}^{\operatorname*{curl}}({\mathbf{u}}-{\mathbf{R}%
}^{\operatorname*{div}}(\operatorname*{div}{\mathbf{u}}))+{\mathbf{R}%
}^{\operatorname*{div}}(\operatorname*{div}{\mathbf{u}})
=:\operatorname*{\mathbf{curl}} \boldsymbol{\varphi} + {\mathbf z}.
\end{align*}
The stability properties of $\mathbf{R}^{\operatorname*{curl}}$ and $\mathbf{R}^{\operatorname*{div}}$ imply
\begin{align*}
\Vert\boldsymbol{\varphi}\Vert_{\mathbf{H}^{s+1}(\widehat{K})} &\lesssim \Vert \mathbf{u}-\mathbf{R}^{\operatorname*{div}}(\operatorname*{div}\mathbf{u})\Vert_{\mathbf{H}^s(\widehat{K})} \lesssim \Vert\mathbf{u}\Vert_{\mathbf{H}^s(\widehat{K})} + \Vert\mathbf{R}^{\operatorname{div}}(\operatorname*{div}\mathbf{u})\Vert_{H^{s+1}(\widehat{K})} \\
&\lesssim \Vert\mathbf{u}\Vert_{\mathbf{H}^s(\widehat{K})} + \Vert\operatorname*{div}\mathbf{u}\Vert_{H^s(\widehat{K})} \lesssim \Vert\mathbf{u}\Vert_{\mathbf{H}^s(\widehat{K},\operatorname*{div})}, \\
\Vert\mathbf{z}\Vert_{\mathbf{H}^{s+1}(\widehat{K})} &= \Vert\mathbf{R}^{\operatorname*{div}}(\operatorname*{div}\mathbf{u})\Vert_{\mathbf{H}^{s+1}(\widehat{K})} \lesssim \Vert\operatorname*{div}\mathbf{u}\Vert_{H^s(\widehat{K})}.
\qedhere
\end{align*}
\end{proof}

We now state the Friedrichs inequalities for the operators $\operatorname{\mathbf{curl}}$ and $\operatorname{div}$.

\begin{lemma}
[discrete Friedrichs inequality for $\mathbf{H}(\operatorname{\mathbf{curl}})$ in 3D, \protect{\cite[Lemma~{5.1}]{demkowicz08}}]
\label{lemma:discrete-friedrichs-3d}
There exists $C > 0$ independent of $p$ and ${\mathbf{u}}$ such that
\begin{equation}
\label{eq:lemma:discrete-friedrichs-3d}\|{\mathbf{u}}\|_{L^{2}(\widehat{K})}
\leq C \|\operatorname{\mathbf{curl}} {\mathbf{u}}\|_{L^{2}(\widehat{K})}%
\end{equation}
in the following two cases:

\begin{enumerate}
[(i)]

\item \label{item:lemma:discrete-friedrichs-3d-i} ${\mathbf{u}}\in
\mathbf{Q}_{p,\perp}(\widehat{K}) := \{ {\mathbf v} \in \mathbf{Q}_p(\widehat K)\colon({\mathbf{v}},\nabla \psi)_{L^{2}(\widehat{K})}=0 \quad \forall \psi\in W_{p+1}(\widehat{K})\}$,

\item \label{item:lemma:discrete-friedrichs-3d-ii} ${\mathbf{u}} \in
\mathring{\mathbf{Q}}_{p,\perp}(\widehat{K}):=\{ {\mathbf v} \in \mathring{\mathbf{Q}}_p(\widehat K)\colon
({\mathbf v},\nabla \psi)_{L^2(\widehat K)} = 0 \quad \forall \psi \in \mathring{W}_{p+1}(\widehat K)\}$.
\end{enumerate}
\end{lemma}

\begin{lemma}
[discrete Friedrichs inequality for $\mathbf{H}(\operatorname*{div})$]

\label{lemma:discrete-friedrichs-div} There exists $C > 0$ independent of $p$
and ${\mathbf{u}}$ such that
\begin{equation}
\label{eq:lemma:discrete-friedrichs-div-3d}\|{\mathbf{u}}\|_{L^{2}(\widehat{K})} \leq C
\|\operatorname{div} {\mathbf{u}}\|_{L^{2}(\widehat{K})}%
\end{equation}
in the following two cases:

\begin{enumerate}
[(i)]

\item \label{item:lemma:discrete-friedrichs-div-i} ${\mathbf{u}}\in
\mathbf{V}_{p}(\widehat{K})$ satisfies $({\mathbf{u}%
},\operatorname*{\mathbf{curl}} \mathbf{v})_{L^{2}(\widehat{K})}=0$ for all $\mathbf{v}\in \mathbf{Q}_{p}(\widehat{K})$,

\item \label{item:lemma:discrete-friedrichs-div-ii} ${\mathbf{u}} \in
\mathring{\mathbf{V}}_p(\widehat{K})$ satisfies $({\mathbf{u}}, \operatorname*{\mathbf{curl}} \mathbf{v})_{L^{2}(\widehat{K})} = 0$
for all $\mathbf{v} \in\mathring{\mathbf{Q}}_{p,\perp}(\widehat{K})$.
\end{enumerate}
\end{lemma}

\begin{proof}
The statement (\ref{item:lemma:discrete-friedrichs-div-i}) is taken from \cite[Lemma~{5.2}]{demkowicz08}. 
It is also shown in \cite[Lemma~{5.2}]{demkowicz08} that the Friedrichs inequality 
(\ref{eq:lemma:discrete-friedrichs-div-3d}) holds for all ${\mathbf u}$ satisfying 
\begin{align}
\label{eq:item:lemma-discrete-friedrichs-div-ii-alternative}
{\mathbf{u}} \in
\mathring{\mathbf{V}}_p(\widehat{K}) 
\text{ satisfies } ({\mathbf{u}}, \operatorname*{\mathbf{curl}} \mathbf{v})_{L^{2}(\widehat{K})} = 0 
\text{ for all } \mathbf{v} \in \mathring{{\mathbf Q}}_{p}(\widehat K).
\end{align}
To see that the condition 
(\ref{item:lemma:discrete-friedrichs-div-ii}) in Lemma~\ref{lemma:discrete-friedrichs-div} 
suffices, assume that ${\mathbf u}$ satisfies the condition  
(\ref{item:lemma:discrete-friedrichs-div-ii}) in Lemma~\ref{lemma:discrete-friedrichs-div} and 
write ${\mathbf v} \in \mathring{\mathbf Q}_p(\widehat{K})$ as 
 $\mathbf{v}=\Pi_{\nabla \mathring{W}_{p+1}}\mathbf{v} + (\mathbf{v}-\Pi_{\nabla \mathring{W}_{p+1}}\mathbf{v})$, where $\Pi_{\nabla \mathring{W}_{p+1}}$ denotes the $L^2$-projection onto $\nabla \mathring{W}_{p+1}(\widehat{K}) \subset \mathring{\mathbf{Q}}_p(\widehat{K})$. Then observe that ${\mathbf v} - \Pi_{\nabla \mathring{W}_{p+1}}\mathbf{v} \in 
\mathring{\mathbf{Q}}_{p,\perp}(\widehat K)$ so that 
\begin{align*}
(\mathbf{u},\operatorname*{\mathbf{curl}}\mathbf{v})_{L^2(\widehat{K})} = 
(\mathbf{u},\underbrace{\operatorname*{\mathbf{curl}}(\Pi_{\nabla \mathring{W}_{p+1}}\mathbf{v}}_
                       {=0}                          )
)_{L^2(\widehat{K})} 
+ \underbrace{ (\mathbf{u},\operatorname*{\mathbf{curl}}(\mathbf{v}-\Pi_{\nabla \mathring{W}_{p+1}}\mathbf{v})
               )_{L^2(\widehat{K})}
             }_{=0
   \text{ since ${\mathbf v} - \Pi_{\nabla \mathring{W}_{p+1}} \mathbf {v} \in \mathring{\mathbf Q}_{p,\perp}(\widehat{K})$}}  = 0; 
\end{align*}
hence, ${\mathbf u}$ satisfies in fact 
(\ref{eq:item:lemma-discrete-friedrichs-div-ii-alternative}). Thus,
it satisfies the Friedrichs inequality  (\ref{eq:lemma:discrete-friedrichs-div-3d}).  
\end{proof}

\subsection{Stability of the operator $\protect\hatPigradcom$}

The three-dimensional analog of Theorem~\ref{lemma:demkowicz-grad-2D} is: 

\begin{theorem}
\label{lemma:demkowicz-grad-3D}
Assume that all interior angles of the 4 faces of $\widehat K$ are smaller than $2\pi/3$. Then, 
for every $s\in [0,1]$ 
there is $C_s > 0$ such that 
for all $u\in H^2(\widehat{K})$
\begin{subequations}
\begin{align}
\label{eq:lemma:demkowicz-grad-3D-10}
\Vert u-\hatPigradcom u\Vert_{H^{1-s}(\widehat{K})}&\leq
C_{s}p^{-(1+s)} \inf_{v \in W_{p+1}(\widehat K)} \Vert u -v\Vert_{H^{2}(\widehat{K})}, \\
\label{eq:lemma:demkowicz-grad-3D-20}
\Vert \nabla(u-\hatPigradcom u)\Vert_{\widetilde{\mathbf{H}}^{-s}(\widehat{K})}&\leq
C_{s}p^{-(1+s)} \inf_{v \in W_{p+1}(\widehat K)} \Vert u -v\Vert_{H^{2}(\widehat{K})}.
\end{align}
\end{subequations}
Additionally, \eqref{eq:lemma:demkowicz-grad-3D-10} holds for  $s = 0$ without the conditions on the 
angles of the faces of $\widehat K$. 
\end{theorem}

\begin{proof}
The proof proceeds along the same lines as in the 2D case. First, we observe from the 
projection property of $\hatPigradcom$ that it suffices to show 
(\ref{eq:lemma:demkowicz-grad-3D-10}), \eqref{eq:lemma:demkowicz-grad-3D-20} with $v = 0$ in the infimum. 
Next, the trace theorem implies $u|_{f} \in H^{3/2}(f)$ for every face $f \in {\mathcal F}(\widehat K)$. 
{}From Theorem~\ref{lemma:demkowicz-grad-2D} we get, for every face $f\in{\mathcal{F}%
}(\widehat{K})$ and $s\in\lbrack0,1]$,
\begin{equation}
\Vert u-\hatPigradcom u\Vert_{H^{1-s}(f)}\leq Cp^{-(1/2+s)}%
\Vert u\Vert_{H^{2}(\widehat{K})}.
\end{equation}
Since  $u-\hatPigradcom u\in
C(\partial\widehat{K})$, we conclude
\begin{equation}
\Vert u-\hatPigradcom u\Vert_{H^{1-s}(\partial \widehat{K})}\leq
Cp^{-(1/2+s)}\Vert u\Vert_{H^{2}(\widehat{K})}
\label{eq:lemma:demkowicz-grad-3D-120}%
\end{equation}
for $s\in\{0,1\}$ and then, by interpolation for all $s\in\lbrack0,1]$. 
Next, we show \eqref{eq:lemma:demkowicz-grad-3D-10} for $s=0$ (from which \eqref{eq:lemma:demkowicz-grad-3D-20} for $s=0$ follows). As in the 2D case, we use 
Lemma~\ref{lemma:Pgrad3d}, the estimate \eqref{eq:lemma:demkowicz-grad-3D-120}, 
the fact that $P^{\operatorname{grad},3d} u - \hatPigradcom u$ is discrete harmonic, and 
the polynomial preserving lifting of \cite{munoz-sola97}, to arrive at 
\begin{align}
%\begin{split}
\nonumber |u-\hatPigradcom u|_{H^{1}(\widehat{K})} &\leq |u-P^{\operatorname*{grad},3d}u|_{H^1(\widehat{K})} + |P^{\operatorname*{grad},3d}u-\hatPigradcom u|_{H^1(\widehat{K})} \\
\label{eq:lemma:demkowicz-grad-3D-145} &\lesssim p^{-1} \Vert u\Vert_{H^2(\widehat{K})} + \Vert P^{\operatorname*{grad},3d}u-\hatPigradcom u\Vert_{H^{1/2}(\partial\widehat{K})} \\
\nonumber &\lesssim p^{-1} \Vert u\Vert_{H^2(\widehat{K})} + \Vert u-P^{\operatorname*{grad},3d}u\Vert_{H^1(\widehat{K})} 
\lesssim p^{-1} \Vert u\Vert_{H^2(\widehat{K})}.
%\end{split}
\end{align}
The $L^{2}$-estimate, i.e., the case $s = 1$ in (\ref{eq:lemma:demkowicz-grad-3D-10}), 
is obtained by a duality argument: Let $z\in
H^{2}(\widehat{K})\cap H_{0}^{1}(\widehat{K})$ be given by
\[
-\Delta z=\widetilde{e}:=u-\hatPigradcom u\quad
\mbox {on $\widehat K$},\qquad z|_{\partial\widehat{K}}=0.
\]
Integration by parts leads to
\begin{equation}
\Vert\widetilde{e}\Vert_{L^{2}(\widehat{K})}^{2}=\int_{\widehat{K}}\nabla
z\cdot\nabla\widetilde{e}-\int_{\partial\widehat{K}}\partial_{n}%
z\widetilde{e}. \label{eq:lemma:demkowicz-grad-3D-200}%
\end{equation}
For the first term in (\ref{eq:lemma:demkowicz-grad-3D-200}) we use the
orthogonality properties satisfied by $\widetilde{e}$ and \eqref{eq:lemma:demkowicz-grad-3D-145} to get
\begin{equation}
|(\nabla z,\nabla\widetilde{e})_{L^{2}(\widehat{K})}| \leq \operatorname*{inf}_{\pi\in\mathring{W}_{p+1}(\widehat{K})} \Vert z-\pi\Vert_{H^1(\widehat{K})} \Vert\nabla\widetilde{e}\Vert_{L^2(\widehat{K})} \lesssim p^{-1}%
\Vert\widetilde{e}\Vert_{L^{2}(\widehat{K})}\Vert\nabla\widetilde{e}%
\Vert_{L^{2}(\widehat{K})}. \label{eq:lemma:demkowicz-grad-3D-500}%
\end{equation}
For the second term in (\ref{eq:lemma:demkowicz-grad-3D-200}), 
we use Theorem~\ref{lemma:demkowicz-grad-2D} for each face $f \in {\mathcal F}(\widehat K)$. 
The assumptions on the angles of the faces of $\widehat K$ imply that Theorem~\ref{lemma:demkowicz-grad-2D}
is applicable with $s = 3/2$ (since the pertinent $\widehat s > 3/2 = \pi/(2 \pi/3)$) to give
\begin{align}
\label{eq:lemma:demkowicz-grad-3D-510}%
|(\partial_n z,\widetilde{e})_{L^2(f)}| 
&\leq \Vert\partial_n z\Vert_{H^{1/2}(f)} \Vert\widetilde{e}\Vert_{\widetilde{H}^{-1/2}(f)} 
\lesssim p^{-2} \Vert\partial_n z\Vert_{H^{1/2}(f)} \Vert u\Vert_{H^{3/2}(f)}\\
\nonumber 
&\lesssim p^{-2} \Vert\widetilde{e}\Vert_{L^2(\widehat{K})} \Vert u\Vert_{H^2(\widehat{K})}.
\end{align}
Inserting (\ref{eq:lemma:demkowicz-grad-3D-500}), (\ref{eq:lemma:demkowicz-grad-3D-510}) 
in \eqref{eq:lemma:demkowicz-grad-3D-200} gives the desired estimate \eqref{eq:lemma:demkowicz-grad-3D-10} for $s = 1$. 
An interpolation argument completes the proof for the intermediate values $s\in(0,1)$.

We show the estimate \eqref{eq:lemma:demkowicz-grad-3D-20} for $s=1$ by duality. 
Again, we set $\widetilde{e}:=u-\hatPigradcom u$ and need an estimate for
\begin{align}
\label{eq:lemma:demkowicz-grad-3D-700}
\|\nabla\widetilde{e}\|_{\widetilde{\mathbf{H}}^{-1}(\widehat{K})} = \operatorname*{sup}_{\mathbf{v}\in \mathbf{H}^1(\widehat{K})} \frac{(\nabla\widetilde{e},\mathbf{v})_{L^2(\widehat{K})}}{\|\mathbf{v}\|_{\mathbf{H}^1(\widehat{K})}}.
\end{align}
According to Lemma~\ref{lemma:helmholtz-like-decomp-v2}, any $\mathbf{v}\in \mathbf{H}^1(\widehat{K})$ can be decomposed as $\mathbf{v}=\nabla\varphi+\operatorname{\mathbf{curl}} \mathbf{z}$ 
with $\varphi\in H^2(\widehat{K}) \cap H^1_0(\widehat{K})$ and $\mathbf{z}\in \mathbf{H}^2(\widehat{K})$. Integration by parts then gives
\begin{align*}
(\nabla\widetilde{e},\mathbf{v})_{L^2(\widehat{K})} = (\nabla\widetilde{e},\nabla\varphi)_{L^2(\widehat{K})} + (\Pi_\tau \nabla\widetilde{e},\gamma_\tau \mathbf{z})_{L^2(\partial\widehat{K})}.
\end{align*}
For the first term, we use Lemma~\ref{lemma:Pgrad1d} and \eqref{eq:lemma:demkowicz-grad-3D-10} 
(applied with $s=0$) to obtain
\begin{align*}
\bigl| (\nabla\widetilde{e},\nabla\varphi)_{L^2(\widehat{K})}\bigr| \lesssim \|\nabla\widetilde{e}\|_{L^2(\widehat{K})} \operatorname*{inf}_{\pi\in \mathring{W}_{p+1}(\widehat{K})} \|\varphi-\pi\|_{H^1(\widehat{K})} \lesssim p^{-2} \|u\|_{H^2(\widehat{K})} \|\mathbf{v}\|_{\mathbf{H}^1(\widehat{K})},
\end{align*}
imitating \eqref{eq:lemma:demkowicz-grad-3D-500}. To treat the second term, we note that $\mathbf{z}\in \mathbf{H}^2(\widehat{K})$ implies $\mathbf{z}\in \mathbf{H}^{3/2}(f)$ for each face $f\in\mathcal{F}(\widehat{K})$. 
Thus, Theorem~\ref{lemma:demkowicz-grad-2D} is again applicable with $s=3/2$, and we get
\begin{align*}
&\bigl| (\Pi_\tau \nabla\widetilde{e},\gamma_\tau \mathbf{z})_{L^2(f)} \bigr| = 
\bigl| (\nabla_f\widetilde{e},\gamma_\tau \mathbf{z})_{L^2(f)} \bigr| \lesssim \|\nabla_f \widetilde{e}\|_{\widetilde{\mathbf{H}}^{-3/2}(f)} \|\gamma_\tau\mathbf{z}\|_{\mathbf{H}^{3/2}(f)} \\
&\quad \stackrel{\text{Thm.~\ref{lemma:demkowicz-grad-2D}}}{\lesssim} 
p^{-2} \|u\|_{H^{3/2}(f)} \|{\mathbf z}\|_{H^2(\widehat{K})} \lesssim p^{-2} \|u\|_{H^2(\widehat{K})} \|\mathbf{v}\|_{\mathbf{H}^1(\widehat{K})}.
\end{align*}
Inserting the last two estimates in \eqref{eq:lemma:demkowicz-grad-3D-700} yields \eqref{eq:lemma:demkowicz-grad-3D-20} for $s=1$. The estimate \eqref{eq:lemma:demkowicz-grad-3D-20} for $s\in (0,1)$ now follows by interpolation.
\end{proof}

%---------------------------------------
\subsection{Stability of the operator $\protect\hatPicurlcom$}

%---------------------------------
As in the proof of Lemma~\ref{lemma:Picurl-face}, a key
ingredient is the existence of a polynomial preserving lifting operator from
the boundary to the element with the appropriate
mapping properties and an additional orthogonality property. 
For ${\mathbf{H}}(\widehat{K},\operatorname{\mathbf{curl}})$, 
a lifting operator has been constructed in
\cite{demkowicz-gopalakrishnan-schoeberl-II}. 
We formulate a simplified version
of their results and also explicitly modify that lifting to ensure a convenient orthogonality property.

\begin{lemma}
\label{lemma:Hcurl-lifting} 
Introduce on the trace space $\Pi_{\tau}{\mathbf{H}}(\widehat{K}%
,\operatorname{\mathbf{curl}})$ the norm
\begin{equation}
\Vert{\mathbf{z}}\Vert_{{\mathbf{X}}^{-1/2}}:=\inf\{\Vert{\mathbf{v}}%
\Vert_{{\mathbf{H}}(\widehat{K},\operatorname{\mathbf{curl}})}\,|\,\Pi_{\tau
}{\mathbf{v}}={\mathbf{z}}\}.
\end{equation}
There exists $C >0$ (independent of $p \in {\mathbb N}$) 
and, for each $p \in {\mathbb N}$, 
a lifting operator ${\boldsymbol{\mathcal{L}}}^{\operatorname*{curl}%
,3d}_p:\Pi_\tau {\mathbf Q}_p(\widehat{K}) \rightarrow {\mathbf Q}_p(\widehat K)$
with the following properties:

\begin{enumerate}
[(i)]

\item \label{item:lemma:Hcurl-lifting-i} 
$\Pi_\tau {\boldsymbol{\mathcal{L}}}^{\operatorname*{curl},3d}_p(\Pi_\tau {\mathbf z}) = 
\Pi_\tau {\mathbf z}$ for all ${\mathbf z} \in {\mathbf Q}_p(\widehat K)$. 

\item \label{item:lemma:Hcurl-lifting-ii} There holds $\displaystyle\Vert
{\boldsymbol{\mathcal{L}}}^{\operatorname*{curl},3d}_p{\mathbf{z}}\Vert_{{\mathbf{H}%
}(\widehat{K},\operatorname{\mathbf{curl}})}\leq C\Vert{\mathbf{z}}\Vert_{{\mathbf{X}%
}^{-1/2}}.$

\item \label{item:lemma:Hcurl-lifting-iia} There holds the orthogonality $(\boldsymbol{\mathcal{L}}^{\operatorname*{curl},3d}_p\mathbf{z},\nabla v)_{L^2(\widehat{K})} = 0$ for all $v\in \mathring{W}_{p+1}(\widehat{K})$.
\item \label{item:lemma:Hcurl-lifting-iii} 
Let ${\mathbf T}:= \Pi_\tau {\mathbf H}^{2}(\widehat{K})$. 
A function ${\mathbf{z}}\in{\mathbf{T}}$ is in ${L}^2(\partial\widehat{K})$, 
facewise in ${\mathbf H}^{3/2}_{T}$, and 
$$\displaystyle\Vert{\mathbf{z}}\Vert_{{\mathbf{X}}^{-1/2}}\leq
C\sum_{f\in{\mathcal{F}}(\widehat{K})}\left[  \Vert{\mathbf{z}}\Vert
_{\widetilde{\mathbf{H}}_T^{-1/2}(f)}+\Vert\operatorname{curl}_{f}{\mathbf{z}%
}\Vert_{\widetilde{H}^{-1/2}(f)}\right]  .
$$

Here, we recall from \eqref{eq:def-negative-norm} that $\|\cdot\|_{\widetilde{\mathbf{H}}_T^{-1/2}(f)}$ 
is defined to be dual to $\|\cdot\|_{{\mathbf{H}}_T^{1/2}(f)}$.
\end{enumerate}
\end{lemma}

\begin{proof}
The lifting operator $\boldsymbol{\mathcal{E}}^{\operatorname*{curl}}:\Pi_\tau({\mathbf H}(\widehat K,\operatorname{\mathbf{curl}})) \rightarrow {\mathbf H}(\widehat{K},\operatorname*{\mathbf{curl}})$ constructed in \cite{demkowicz-gopalakrishnan-schoeberl-II} has the desired polynomial preserving property (\ref{item:lemma:Hcurl-lifting-i}) and continuity property
(\ref{item:lemma:Hcurl-lifting-ii}), \cite[Thm.~{7.2}]%
{demkowicz-gopalakrishnan-schoeberl-II}. To ensure 
(\ref{item:lemma:Hcurl-lifting-iia}) we define the desired lifting operator 
by $\boldsymbol{\mathcal{L}}^{\operatorname*{curl},3d}_p\mathbf{z} := \boldsymbol{\mathcal{E}}^{\operatorname*{curl}}\mathbf{z}-\mathbf{w}_0$, 
where $\mathbf{w}_0$ is defined by the following saddle point problem: 
Find $\mathbf{w}_0\in \mathring{\mathbf{Q}}_p(\widehat{K})$ 
and $\varphi\in \mathring{W}_{p+1}(\widehat{K})$ such that
for all 
$\mathbf{q}\in \mathring{\mathbf{Q}}_p(\widehat{K})$ and all 
$\mu\in \mathring{W}_{p+1}(\widehat{K})$
\begin{subequations}
\label{eq:lemma:saddle-point-curl}
\begin{align}
\label{eq:lemma:saddle-point-curl-a}
(\operatorname*{\mathbf{curl}}\mathbf{w}_0,\operatorname*{\mathbf{curl}}\mathbf{q})_{L^2(\widehat{K})}  +  (\mathbf{q},\nabla\varphi)_{L^2(\widehat{K})} & =  (\operatorname*{\mathbf{curl}}(\boldsymbol{\mathcal{E}}^{\operatorname*{curl}}\mathbf{z}),\operatorname*{\mathbf{curl}}\mathbf{q})_{L^2(\widehat{K})} 
\\
%\qquad \forall\mathbf{q}\in \mathring{\mathbf{Q}}_p(\widehat{K}) \\
\label{eq:lemma:saddle-point-curl-b}
(\mathbf{w}_0,\nabla\mu)_{L^2(\widehat{K})} & =  (\boldsymbol{\mathcal{E}}^{\operatorname*{curl}}\mathbf{z},\nabla\mu)_{L^2(\widehat{K})}. 
%  \qquad \forall\mu\in \mathring{W}_{p+1}(\widehat{K}).
\end{align}
\end{subequations}
Problem~\eqref{eq:lemma:saddle-point-curl} is uniquely solvable: 
We define the bilinear forms 
$a(\mathbf{w},\mathbf{q}):=(\operatorname*{\mathbf{curl}}\mathbf{w},\operatorname*{\mathbf{curl}}\mathbf{q})_{L^2(\widehat{K})}$ 
and $b(\mathbf{w},\varphi) := (\mathbf{w},\nabla\varphi)_{L^2(\widehat{K})}$ 
for $\mathbf{w}, \mathbf{q}\in \mathring{\mathbf{Q}}_p(\widehat{K})$ and $\varphi\in \mathring{W}_{p+1}(\widehat{K})$. 
Coercivity of $a$ on the kernel of 
$b$ with
\begin{align*}
\operatorname*{ker}b=
\{\mathbf{q}\in \mathring{\mathbf{Q}}_p(\widehat{K})\colon (\mathbf{q},\nabla\mu)_{L^2(\widehat{K})} = 0 \, \forall\mu\in \mathring{W}_{p+1}\} = \mathring{\mathbf{Q}}_{p,\perp}(\widehat K),
\end{align*}
follows from the Friedrichs inequality (Lemma~\ref{lemma:discrete-friedrichs-3d}) by
\begin{align*}
a(\mathbf{v},\mathbf{v})&=\Vert\operatorname*{\mathbf{curl}}\mathbf{v}\Vert_{L^2(\widehat{K})}^2 \geq \frac{1}{2C^2} \Vert\mathbf{v}\Vert_{L^2(\widehat{K})}^2 + \frac{1}{2}\Vert\operatorname*{\mathbf{curl}}\mathbf{v}\Vert_{L^2(\widehat{K})}^2 \\
&\geq \operatorname*{min}\{\frac{1}{2C^2},\frac{1}{2}\}\Vert\mathbf{v}\Vert_{\mathbf{H}(\widehat{K},\operatorname*{\mathbf{curl}})}^2
\end{align*}
for all $\mathbf{v}\in \operatorname*{ker}b$. Next, we show the inf-sup condition
\begin{align*}
\operatornamewithlimits{inf}_{\varphi\in\mathring{W}_{p+1}(\widehat{K})} \operatornamewithlimits{sup}_{\mathbf{w}\in\mathring{\mathbf{Q}}_p(\widehat{K})} \frac{b(\mathbf{w},\varphi)}{\Vert\mathbf{w}\Vert_{\mathbf{H}(\widehat{K},\operatorname*{\mathbf{curl}})} \Vert\varphi\Vert_{H^1(\widehat{K})}} \geq C.
\end{align*}
Given $\varphi\in\mathring{W}_{p+1}(\widehat{K})$, choose $\mathbf{w}=\nabla\varphi\in\mathring{\mathbf{Q}}_p(\widehat{K})$. Hence,
\begin{align*}
\frac{b(\mathbf{w},\varphi)}{\Vert\mathbf{w}\Vert_{\mathbf{H}(\widehat{K},\operatorname*{\mathbf{curl}})} \Vert\varphi\Vert_{H^1(\widehat{K})}} = \frac{\Vert\nabla\varphi\Vert_{L^2(\widehat{K})}^2}{\Vert\nabla\varphi\Vert_{L^2(\widehat{K})} \Vert\varphi\Vert_{H^1(\widehat{K})}} \geq C
\end{align*}
by Poincar\'e's inequality. Thus, the saddle point problem \eqref{eq:lemma:saddle-point-curl} has a 
unique solution 
$(\mathbf{w}_0,\varphi) \in \mathring{\mathbf{Q}}_p(\widehat{K}) \times \mathring{W}_{p+1}(\widehat K)$. 
In fact, taking ${\mathbf q} = \nabla \varphi$ in (\ref{eq:lemma:saddle-point-curl-a}) reveals 
$\varphi = 0$. 
The lifting operator $\boldsymbol{\mathcal{L}}^{\operatorname*{curl},3d}_p$ now obviously satisfies 
(\ref{item:lemma:Hcurl-lifting-i}) and (\ref{item:lemma:Hcurl-lifting-iia}) by construction. 
For (\ref{item:lemma:Hcurl-lifting-ii}) note that the solution $\mathbf{w}_0$ satisfies the estimate 
$\Vert\mathbf{w}_0\Vert_{\mathbf{H}(\widehat{K},\operatorname*{\mathbf{curl}})} \lesssim \Vert f\Vert + \Vert g\Vert$, 
where $f(\mathbf{v})=(\operatorname*{\mathbf{curl}}(\boldsymbol{\mathcal{E}}^{\operatorname*{curl}}\mathbf{z}),\operatorname*{\mathbf{curl}}\mathbf{v})_{L^2(\widehat{K})}$, $g(v)=(\boldsymbol{\mathcal{E}}^{\operatorname*{curl}}\mathbf{z},\nabla v)_{L^2(\widehat{K})}$, and $\Vert \cdot \Vert$ denotes the operator norm. Thus,
\begin{align*}
\Vert f\Vert =\!\! \operatornamewithlimits{sup}_{\Vert\mathbf{v}\Vert_{\mathbf{H}(\widehat{K},\operatorname*{\mathbf{curl}})} \leq 1} \! |(\operatorname*{\mathbf{curl}}(\boldsymbol{\mathcal{E}}^{\operatorname*{curl}}\mathbf{z}),\operatorname*{\mathbf{curl}}\mathbf{v})_{L^2(\widehat{K})}| \leq \Vert\operatorname*{\mathbf{curl}}(\boldsymbol{\mathcal{E}}^{\operatorname*{curl}}\mathbf{z})\Vert_{L^2(\widehat{K})} \lesssim \Vert\mathbf{z}\Vert_{\mathbf{X}^{-1/2}}.
\end{align*}
The estimate
%\begin{align*}
$\displaystyle \Vert g\Vert \lesssim \Vert\mathbf{z}\Vert_{\mathbf{X}^{-1/2}}
$
%\end{align*}
is shown in a similar way. Hence, (\ref{item:lemma:Hcurl-lifting-ii}) follows from
\begin{align*}
\Vert\boldsymbol{\mathcal{L}}^{\operatorname*{curl},3d}_p\mathbf{z}\Vert_{\mathbf{H}(\widehat{K},\operatorname*{\mathbf{curl}})} \leq \Vert\boldsymbol{\mathcal{E}}^{\operatorname*{curl}}\mathbf{z}\Vert_{\mathbf{H}(\widehat{K},\operatorname*{\mathbf{curl}})} + \Vert\mathbf{w}_0\Vert_{\mathbf{H}(\widehat{K},\operatorname*{\mathbf{curl}})} \lesssim \Vert\mathbf{z}\Vert_{\mathbf{X}^{-1/2}}.
\end{align*}
We now show (\ref{item:lemma:Hcurl-lifting-iii}), proceeding in several steps. 
\newline 
\emph{1st~step:} 
Clearly, ${\mathbf z}$ is in $L^2(\partial \widehat K)$ and facewise in ${\mathbf H}^{3/2}_T$. 
The surface curl of ${\mathbf z} \in {\mathbf T}$, 
denoted $\operatorname{curl}_{\partial\widehat K} {\mathbf z}$, is defined by 
${\mathbf n} \cdot \operatorname{\mathbf{curl}} \widetilde{\mathbf z} \in H^{-1/2}(\partial\widehat K)$ 
for any lifting $\widetilde{\mathbf  z} \in {\mathbf H}(\widehat K,\operatorname{\mathbf{curl}})$
of ${\mathbf z}$. 
This definition is indeed independent of the lifting since the difference 
${\boldsymbol \delta}$ of two liftings is in ${\mathbf H}_0(\widehat K,\operatorname{\mathbf{curl}})$
and by the deRham diagram (see, e.g., \cite[eqn. (3.60)]{Monkbook}) we then have  
$\operatorname{\mathbf{curl}} {\boldsymbol \delta} \in {\mathbf H}_0(\widehat K,\operatorname{div})$. 
Furthermore, since an ${\mathbf H}^2$-lifting of ${\mathbf z}$ exists, 
$\operatorname{curl}_{\partial\widehat K} {\mathbf z} \in H^{-1/2}(\partial\widehat K)$ is facewise in ${\mathbf H}_T^{1/2}$ and coincides 
facewise with $\operatorname{curl}_f {\mathbf  z}$. 

\emph{2nd~step:} 
We construct a particular lifting ${\mathbf Z} \in {\mathbf H}(\widehat K,\operatorname{\mathbf{curl}})$ 
of ${\mathbf z} \in {\mathbf X}^{-1/2}$ and will use 
$\|{\mathbf z}\|_{{\mathbf X}^{-1/2}} \leq \|{\mathbf Z}\|_{{\mathbf H}(\widehat K,\operatorname{\mathbf{curl}})}$. 
This lifting ${\mathbf Z}$ is taken to be the solution of the following (constrained)
minimization problem:
\begin{align*}
& \mbox{ Minimize } \|\operatorname{\mathbf{curl}} {\mathbf Y}\|_{L^2(\widehat{K})} 
\mbox{ under the constraints} \\
& \mbox{$\Pi_\tau {\mathbf Y} = {\mathbf z}$ \qquad and \qquad $({\mathbf Y} ,\nabla \varphi)_{L^2(\widehat{K})} = 0$
for all $\varphi \in H^1_0(\widehat{K})$.}
\end{align*}
This minimization problem can be solved with the method of Lagrange multipliers as was done 
in (\ref{eq:lemma:saddle-point-curl}). Without repeating the arguments, 
we obtain, in strong form, the problem:
Find $({\mathbf Z},\varphi) \in {\mathbf H}(\widehat{K},\operatorname{\mathbf{curl}}) \times H^1_0(\widehat{K})$ such that
\begin{align*}
\operatorname{\mathbf{curl}}
\operatorname{\mathbf{curl}} {\mathbf Z} + \nabla \varphi = 0 \quad \mbox{ in $\widehat{K}$}, 
\qquad 
\operatorname{div} {\mathbf Z} = 0 \quad \mbox{ in $\widehat{K}$}, 
\qquad \Pi_\tau {\mathbf Z} = {\mathbf z}. 
\end{align*}
As was observed above, the Lagrange multiplier $\varphi$ in fact vanishes so that we conclude that 
the minimizer ${\mathbf Z}$ solves  
\begin{align*}
\operatorname{\mathbf{curl}} \operatorname{\mathbf{curl}} {\mathbf Z}  = 0, \qquad 
\operatorname{div} {\mathbf Z}  = 0, 
\qquad \Pi_\tau {\mathbf Z} = {\mathbf z}. 
\end{align*}
\emph{3rd~step:} We bound ${\mathbf w}:= \operatorname{\mathbf{curl}} {\mathbf Z}$. We have
\begin{align}
\label{eq:lemma:X-1/2-vs-H-1/2-curl-40-vorn}
\operatorname{\mathbf{curl}} {\mathbf w}  = 0 , \qquad 
\operatorname{div} {\mathbf w}  = 0, \qquad  
{\mathbf n} \cdot {\mathbf w}  = \operatorname{curl}_{\partial\widehat{K}} {\mathbf z}. 
\end{align}
{}From $\operatorname{\mathbf{curl}} {\mathbf w} = 0$, we get that ${\mathbf w}$ is a gradient:
${\mathbf w} = \nabla \psi$. The second and third conditions in (\ref{eq:lemma:X-1/2-vs-H-1/2-curl-40-vorn}) show
\begin{align*}
-\Delta \psi = 0 \qquad \partial_n \psi = {\mathbf n} \cdot {\mathbf w} = 
\operatorname{curl}_{\partial\widehat{K}} {\mathbf z}.
\end{align*}
The integrability condition is satisfied since
$( {\mathbf n} \cdot {\mathbf w},1)_{L^2(\partial\widehat{K})} 
= (\operatorname{div} {\mathbf w},1)_{L^2(\widehat{K})} = 0$. 
Thus we conclude by standard {\sl a priori} estimates for the Laplace problem
\begin{equation}
\label{eq:lemma:X-1/2-vs-H-1/2-curl-45-vorn}
\|\operatorname{\mathbf{curl}} {\mathbf Z}\|_{L^2(\widehat{K})} = 
\|{\mathbf w}\|_{L^2(\widehat{K})} = \|\nabla \psi\|_{L^2(\widehat{K})} 
\lesssim \|\operatorname{curl}_{\partial\widehat{K}} {\mathbf z}\|_{H^{-1/2}(\partial\widehat{K})}. 
\end{equation}
\emph{4th~step:} To bound ${\mathbf Z}$, we write it with the operators 
${\mathbf R}^{\operatorname{curl}}$ and $R^{\operatorname{grad}}$ of Lemma~\ref{lemma:mcintosh} as 
\begin{align}
\label{eq:lemma:X-1/2-vs-H-1/2-curl-50-vorn}
& {\mathbf Z} = \nabla \phi + \widetilde {\mathbf z}, 
\qquad \widetilde {\mathbf z}:= {\mathbf R}^{\operatorname{curl}}(\operatorname{\mathbf{curl}}{\mathbf Z}), 
\qquad \phi := R^{\operatorname{grad}} ({\mathbf Z} - \mathbf{R}^{\operatorname{curl}}(\operatorname{\mathbf {curl}} \widetilde {\mathbf z})), \\
\label{eq:lemma:X-1/2-vs-H-1/2-curl-100-vorn}
& \mbox{ with }  
\|\widetilde {\mathbf z}\|_{H^1(\widehat{K})} \lesssim \|\operatorname{\mathbf{curl}} {\mathbf Z}\|_{L^2(\widehat{K})} 
\lesssim \|\operatorname{curl}_{\partial\widehat{K}} {\mathbf z}\|_{H^{-1/2}(\partial\widehat{K})}.
\end{align}
For the control of $\phi$, we proceed by an integration by parts argument.
Noting that $\operatorname{div} {\mathbf Z} = 0$, we have
$$
\nabla \phi + \widetilde {\mathbf z} = {\mathbf Z} 
= \operatorname{\mathbf{curl}} {\mathbf R}^{\operatorname{curl}} ({\mathbf  Z}) 
= \operatorname{\mathbf{curl}} {\mathbf R}^{\operatorname{curl}}  (\nabla \phi) 
  + \operatorname{\mathbf{curl}} {\mathbf R}^{\operatorname{curl}} (\widetilde {\mathbf z}).  
$$
With the integration by parts formula (\ref{eq:integration-by-parts}) (which is actually
valid for functions in ${\mathbf H}(\widehat K,\operatorname{\mathbf{curl}})$ as shown in 
\cite[Thm.~{3.29}]{Monkbook}) we get 
\[
(\operatorname{\mathbf{curl}} {\mathbf Z}, {\mathbf v})_{L^2(\widehat{K})} 
\stackrel{(\ref{eq:integration-by-parts})}{= }
({\mathbf Z}, \operatorname{\mathbf{curl}} {\mathbf v})_{L^2(\widehat{K})} 
- ({\mathbf z}, \gamma_\tau {\mathbf v})_{L^2(\partial \widehat{K})}. 
\]
Selecting ${\mathbf v} = {\mathbf R}^{\operatorname{curl}}(\nabla \phi) \in {\mathbf H}^1(\widehat{K})$, we get 
\begin{align*}
(\operatorname{\mathbf{curl}} {\mathbf Z},{\mathbf R}^{\operatorname{curl}}(\nabla \phi) )_{L^2(\widehat{K})} &= 
(\nabla \phi + \widetilde {\mathbf z}, \nabla \phi + \widetilde {\mathbf z} - \operatorname{\mathbf{curl}} {\mathbf R}^{\operatorname{curl}}(\widetilde {\mathbf z}))_{L^2(\widehat{K})} \\
&\quad - ({\mathbf z}, \gamma_\tau {\mathbf R}^{\operatorname{curl}}(\nabla \phi))_{L^2(\partial \widehat{K})}.
\end{align*}
In view of the mapping property
${\mathbf R}^{\operatorname{curl}}:L^2(\widehat{K}) \rightarrow {\mathbf H}^1(\widehat{K})$ we obtain 
\begin{align}
\label{eq:lemma:X-1/2-vs-H-1/2-curl-200-vorn}
\|\nabla \phi\|^2_{L^2(\widehat{K})} &\lesssim 
\|\operatorname{\mathbf{curl}} {\mathbf Z}\|_{L^2(\widehat{K})} \|\nabla \phi\|_{L^2(\widehat{K})} + 
\|\widetilde{\mathbf z}\|_{L^2(\widehat K)} 
\|\widetilde{\mathbf z} - \operatorname{\mathbf{curl}} {\mathbf R}^{\operatorname{curl}} (\widetilde{\mathbf z})\|_{L^2(\widehat K)} \\
\nonumber 
& \quad \mbox{} + 
\|\widetilde {\mathbf z} - \operatorname*{\mathbf{curl}} {\mathbf R}^{\operatorname{curl}}(\widetilde {\mathbf z})\|_{L^2(\widehat{K})} \|\nabla \phi\|_{L^2(\widehat{K})} \\
\nonumber 
& \quad \mbox{}+ 
\|\widetilde {\mathbf z} \|_{L^2(\widehat{K})} \|\nabla \phi\|_{L^2(\widehat{K})} + 
\left| ({\mathbf z},\gamma_\tau {\mathbf R}^{\operatorname{curl}}(\nabla \phi))_{L^2(\partial\widehat K)} \right|.
\end{align}
Combining
(\ref{eq:lemma:X-1/2-vs-H-1/2-curl-50-vorn}),
(\ref{eq:lemma:X-1/2-vs-H-1/2-curl-100-vorn}),
(\ref{eq:lemma:X-1/2-vs-H-1/2-curl-200-vorn}) shows 
\begin{align}
\label{eq:lemma:X-1/2-vs-H-1/2-curl-500-vorn}
\|{\mathbf Z}\|_{\mathbf{H}(\widehat{K},\operatorname{\mathbf{curl}})} 
&\lesssim \|\widetilde {\mathbf z}\|_{L^2(\widehat K)} + 
\|\nabla \phi\|_{L^2(\widehat K)} + \|\operatorname{\mathbf{curl}}{\mathbf Z}\|_{L^2(\widehat K)} \\
\nonumber 
& \lesssim \sup_{{\mathbf v} \in {\mathbf H}^1(\widehat K)} 
  \frac{ ({\mathbf z},\gamma_\tau {\mathbf v})_{L^2(\partial \widehat K)}}
       {\|{\mathbf v}\|_{{\mathbf H}^1(\widehat K)}} 
 + \|\operatorname{curl}_{\partial \widehat K} {\mathbf z}\|_{H^{-1/2}(\partial \widehat K)}.
\end{align}
\emph{5th~step:}
Since ${\mathbf z}$ and 
$\operatorname{curl}_{\partial\widehat{K}} {\mathbf z}$ are actually $L^2$-functions, the norm 
$\|\cdot\|_{{\mathbf  X}^{-1/2}}$ can be estimated in a localized fashion: 
The continuity of the  inclusions 
$H^{1/2}(\partial\widehat{K}) \subset \prod_{f \in {\mathcal F}(\widehat{K})} H^{1/2}(f)$ 
and $\gamma_\tau {\mathbf H}^{1}(\widehat{K}) \subset 
\prod_{f \in {\mathcal F}(\widehat{K})} {\mathbf H}_T^{1/2}(f)$ implies 
\begin{subequations}
\label{eq:lemma:X-1/2-vs-H-1/2-curl-550-vorn}
\begin{align}
\|\operatorname{curl}_{\partial \widehat{K}} {\mathbf z}\|_{H^{-1/2}(\partial\widehat{K})} & \lesssim 
\sum_{f \in {\mathcal F}(\widehat{K})} \|\operatorname{curl}_f {\mathbf z}\|_{\widetilde{H}^{-1/2}(f)},  \\
\sup_{{\mathbf v} \in {\mathbf H}^1(\widehat K)} 
  \frac{({\mathbf z},\gamma_\tau {\mathbf v})_{L^2(\partial\widehat K)}}
       {\|{\mathbf v}\|_{{\mathbf H}^1(\widehat K)}} & \lesssim 
\sum_{f \in {\mathcal F}(\widehat{K})} \|{\mathbf z}\|_{\widetilde{\mathbf{H}}^{-1/2}_T(f)}. 
\end{align}
\end{subequations}
We finally obtain the desired estimate
$$
\|{\mathbf  z} \|_{{\mathbf X}^{-1/2}} 
\lesssim \|{\mathbf Z}\|_{{\mathbf H}(\widehat K,\operatorname{\mathbf{curl}}) }
\stackrel{\text{(\ref{eq:lemma:X-1/2-vs-H-1/2-curl-500-vorn}), 
(\ref{eq:lemma:X-1/2-vs-H-1/2-curl-550-vorn})}}{\lesssim}\sum_{f \in {\mathcal F}(\widehat K)} 
\|{\mathbf z}\|_{\widetilde{\mathbf{H}}^{-1/2}_T(f)} + 
\|\operatorname{curl}_f {\mathbf z}\|_{\widetilde{H}^{-1/2}(f)}. 
$$
This concludes the proof. We mention that an alternative proof of 
the assertion (\ref{item:lemma:Hcurl-lifting-iii}) could be based on the intrinsic characterization
of the trace spaces of ${\mathbf H}(\widehat K,\operatorname{\mathbf{curl}})$ given in 
\cite{BuffaCiarlet2001,BuffaCiarlet2001b}.
\end{proof}

\begin{theorem}
\label{thm:H1curl-approximation} 
Let $\widehat K$ be a fixed tetrahedron. Then
there exists $C>0$ independent of $p$ such
that for all ${\mathbf{u}}\in{\mathbf{H}}^{1}(\widehat{K},\operatorname{\mathbf{curl}})$
\begin{equation}
\Vert{\mathbf{u}}-\hatPicurlcom{\mathbf{u}}%
\Vert_{{\mathbf{H}}(\widehat{K},\operatorname{\mathbf{curl}})}\leq Cp^{-1}%
\inf_{{\mathbf v} \in {\mathbf Q}_p(\widehat K)}
\Vert{\mathbf{u}} - \mathbf{v}\Vert_{{\mathbf{H}}^{1}(\widehat{K},\operatorname{\mathbf{curl}})}.
\label{eq:thmH1curl-approximation-10}%
\end{equation}
\end{theorem}

\begin{proof}
\emph{1st~step:} Since $\hatPicurlcom$ is a projection operator, it suffices to show the bound
with ${\mathbf v} = 0$ in the infimum.

\emph{2nd~step:} Write, with the operators $R^{\operatorname*{grad}}$,
${\mathbf{R}}^{\operatorname*{curl}}$ of Lemma~\ref{lemma:mcintosh}, the
function ${\mathbf{u}}\in{\mathbf{H}}^{1}(\widehat{K},\operatorname{\mathbf{curl}})$ as
${\mathbf{u}}=\nabla\varphi+{\mathbf{v}}$ with $\varphi\in H^{2}(\widehat{K})$
and ${\mathbf{v}}\in{\mathbf{H}}^{2}(\widehat{K})$. We have $\Vert\varphi
\Vert_{H^{2}(\widehat{K})}\lesssim\Vert{\mathbf{u}}\Vert_{{\mathbf{H}}%
^{1}(\widehat{K},\operatorname{\mathbf{curl}})}$ and $\Vert{\mathbf{v}}\Vert
_{{\mathbf{H}}^{2}(\widehat{K})}\lesssim\Vert\operatorname{\mathbf{curl}}{\mathbf{u}%
}\Vert_{{\mathbf{H}}^{1}(\widehat{K})}$. From the commuting diagram property,
we readily get
\begin{align*}
\Vert\nabla\varphi-\hatPicurlcom\nabla
\varphi\Vert_{{\mathbf{H}}(\widehat{K},\operatorname{\mathbf{curl}})}\!&=\Vert
\nabla(\varphi-\hatPigradcom\varphi)\Vert
_{{\mathbf{H}}(\widehat{K},\operatorname{\mathbf{curl}})}\!=|\varphi-\hatPigradcom\varphi|_{H^{1}(\widehat{K})} \\
& \stackrel{\text{Thm.~\ref{lemma:demkowicz-grad-3D}}}{\lesssim} p^{-1}%
\Vert\varphi\Vert_{H^{2}(\widehat{K})}.
\end{align*}
\emph{3rd~step:} We claim
\begin{equation}
\Vert\Pi_{\tau}({\mathbf{v}}-\hatPicurlcom{\mathbf{v}})\Vert_{{\mathbf{X}}^{-1/2}}\leq Cp^{-1}\Vert{\mathbf{v}}%
\Vert_{{\mathbf{H}}^{2}(\widehat{K})}. \label{eq:thm:H1curl-approximation-30}%
\end{equation}
To see this, we note ${\mathbf{v}}\in{\mathbf{H}}^{2}(\widehat{K})$ and
estimate with Lemma~\ref{lemma:Hcurl-lifting}
\begin{align*}
&\Vert\Pi_{\tau}({\mathbf{v}}-\hatPicurlcom{\mathbf{v}})\Vert_{{\mathbf{X}}^{-1/2}}\lesssim\\
& \sum_{f\in{\mathcal{F}%
}(\widehat{K})}\Vert\Pi_{\tau}({\mathbf{v}}-\hatPicurlcom{\mathbf{v}})\Vert_{\widetilde{\mathbf{H}}_T^{-1/2}%
(f)}+\Vert\operatorname{curl}_{f}(\Pi_{\tau}({\mathbf{v}}-\hatPicurlcom{\mathbf{v}}))\Vert_{\widetilde{H}^{-1/2}(f)}.
\end{align*}
We consider each face $f\in{\mathcal{F}}(\widehat{K})$ separately.
Lemmas~\ref{lemma:picurl-negative-I}, \ref{lemma:picurl-negative-II}, \ref{lemma:Picurl-face} imply with 
the aid of the continuity of the trace 
$\Pi_\tau: {\mathbf H}^2(\widehat K) \rightarrow {\mathbf H}^{3/2}_T(f) \subset {\mathbf H}^{1/2}(f,\operatorname{curl})$ 
\begin{align*}
\Vert\Pi_{\tau}({\mathbf{v}}-\hatPicurlcom{\mathbf{v})%
}\Vert_{\widetilde{\mathbf{H}}_T^{-1/2}(f)}  &\!\!\!\! 
 \stackrel{\text{Lem.~\ref{lemma:picurl-negative-I}}}{\lesssim} \!\!p^{-1/2}\Vert\Pi_{\tau
}({\mathbf{v}}-\hatPicurlcom{\mathbf{v}})\Vert
_{{\mathbf{H}}(f,\operatorname{curl})}\\
& \!\!\!\! \stackrel{\text{Lem.~\ref{lemma:Picurl-face}}}{\lesssim} 
\!\!
p^{-1/2-1/2}\Vert\Pi_{\tau}{\mathbf{v}}\Vert_{{\mathbf{H}}%
^{1/2}(f,\operatorname{curl})}%+\Vert\Pi_\tau\mathbf{v}\Vert_{\mathbf{H}^1(f)}) 
\lesssim p^{-1}\Vert{\mathbf{v}}\Vert
_{{\mathbf{H}}^{2}(\widehat{K})},\\
\Vert\operatorname{curl}(\Pi_{\tau}({\mathbf{v}}-\hatPicurlcom%
{\mathbf{v}}))\Vert_{\widetilde{H}^{-1/2}(f)}  &  
\!\!\!\!
\stackrel{\text{Lem.~\ref{lemma:picurl-negative-II}}}{\lesssim}
\!\!
p^{-1/2}\Vert\operatorname{curl}(\Pi_{\tau}({\mathbf{v}}%
-\hatPicurlcom{\mathbf{v}}))\Vert_{L^{2}(f)}\\
&  \lesssim p^{-1/2-1/2} \Vert\Pi_{\tau}{\mathbf{v}}\Vert_{{\mathbf{H}}%
^{1/2}(f,\operatorname{curl})} %+ \Vert\Pi_\tau\mathbf{v}\Vert_{\mathbf{H}^1(f)})\\
\lesssim p^{-1}\Vert{\mathbf{v}}\Vert
_{{\mathbf{H}}^{2}(\widehat{K})}.
\end{align*}

\emph{4th~step:} Since ${\mathbf{v}}\in{\mathbf{H}}^{2}(\widehat{K})$, the
approximation $P^{\operatorname*{curl},3d}{\mathbf{v}}\in
\mathbf{Q}_p(\widehat{K})$ given by
Lemma~\ref{lemma:Pcurl3d} satisfies
\begin{equation}
\Vert{\mathbf{v}}-P^{\operatorname*{curl},3d}{\mathbf{v}}\Vert_{{\mathbf{H}%
}(\widehat{K},\operatorname{\mathbf{curl}})}\leq Cp^{-1}\Vert{\mathbf{v}}%
\Vert_{{\mathbf{H}}^{2}(\widehat{K})}. \label{eq:thm:H1curl-approximation-10c}%
\end{equation}
We note
\begin{align*}
\Vert{\mathbf{v}}-\hatPicurlcom{\mathbf{v}}%
\Vert_{{\mathbf{H}}(\widehat{K},\operatorname{\mathbf{curl}})}  &  \leq\Vert
{\mathbf{v}}-P^{\operatorname*{curl},3d}{\mathbf{v}}\Vert_{{\mathbf{H}%
}(\widehat{K},\operatorname{\mathbf{curl}})}
%\\ &\quad 
+\Vert\hatPicurlcom{\mathbf{v}}-P^{\operatorname*{curl},3d}{\mathbf{v}%
}\Vert_{{\mathbf{H}}(\widehat{K},\operatorname{\mathbf{curl}})}\\
&  \leq p^{-1}\Vert{\mathbf{v}}\Vert_{{\mathbf{H}}^{2}(\widehat{K})}%
+\Vert\hatPicurlcom{\mathbf{v}}%
-P^{\operatorname*{curl},3d}{\mathbf{v}}\Vert_{{\mathbf{H}}(\widehat{K}%
,\operatorname{\mathbf{curl}})}.
\end{align*}
For the term $\Vert\hatPicurlcom{\mathbf{v}%
}-P^{\operatorname*{curl},3d}{\mathbf{v}}\Vert_{{\mathbf{H}}(\widehat{K}%
,\operatorname{\mathbf{curl}})}$, we introduce the abbreviation ${\mathbf{E}%
}:=\hatPicurlcom{\mathbf{v}}%
-P^{\operatorname*{curl},3d}{\mathbf{v}}\in\mathbf{Q}_p(\widehat{K})$ and observe that the orthogonality
conditions (\ref{eq:Pi_curl-b}), (\ref{eq:Pi_curl-a}) satisfied by
$\hatPicurlcom{\mathbf{v}}$ and the conditions
(\ref{eq:lemma:Pcurl3d-a}), (\ref{eq:lemma:Pcurl3d-b}) satisfied by
$P^{\operatorname*{curl},3d}{\mathbf{v}}$, lead to two orthogonalities:
\begin{subequations}
\label{eq:orth-3d}%
\begin{align}
\label{eq:orth-3d-a}%
(\operatorname{\mathbf{curl}}{\mathbf{E}},\operatorname{\mathbf{curl}}{\mathbf{w}}%
)_{L^{2}(\widehat{K})}& =0\quad\forall{\mathbf{w}}\in\mathring{\mathbf{Q}}_{p}(\widehat{K}),\\
\label{eq:orth-3d-b}%
({\mathbf{E}},\nabla w)_{L^{2}(\widehat{K})}%
&=0\quad\forall w\in\mathring{W}_{p+1}(\widehat{K}). 
\end{align}
\end{subequations}
By Lemma~\ref{lemma:Hcurl-lifting}, the orthogonality condition
\begin{align*}
(\boldsymbol{\mathcal{L}}^{\operatorname*{curl},3d}_p\Pi_\tau \mathbf{E},\nabla w)_{L^2(\widehat{K})} = 0 \quad \forall w\in\mathring{W}_{p+1}(\widehat{K})
\end{align*}
holds. Hence, the discrete Friedrichs inequality of
Lemma~\ref{lemma:discrete-friedrichs-3d} is applicable to 
$\mathbf{E}-\boldsymbol{\mathcal{L}}^{\operatorname*{curl},3d}_p\Pi_\tau \mathbf{E}$, and we get 
\begin{align}
\Vert{\mathbf{E}}\Vert_{L^{2}(\widehat{K})}  &  \leq\Vert{\boldsymbol{\mathcal{L}}%
}^{\operatorname*{curl},3d}_p\Pi_{\tau}{\mathbf{E}}\Vert_{L^{2}(\widehat{K}%
)}+\Vert{\mathbf{E}}-{\boldsymbol{\mathcal{L}}}^{\operatorname*{curl},3d}_p\Pi_{\tau
}{\mathbf{E}}\Vert_{L^{2}(\widehat{K})}%
\label{eq:thm:H1curl-approximation-100}\\
\nonumber 
&  \lesssim\Vert{\boldsymbol{\mathcal{L}}}^{\operatorname*{curl},3d}_p\Pi_{\tau}{\mathbf{E}%
}\Vert_{L^{2}(\widehat{K})}+\Vert\operatorname{\mathbf{curl}}({\mathbf{E}}%
-{\boldsymbol{\mathcal{L}}}^{\operatorname*{curl},3d}_p\Pi_{\tau}{\mathbf{E}})\Vert
_{L^{2}(\widehat{K})}\\
\nonumber 
&\lesssim\Vert{\boldsymbol{\mathcal{L}}}^{\operatorname*{curl},3d}%
\Pi_{\tau}{\mathbf{E}}\Vert_{\mathbf{H}(\widehat{K},\operatorname{\mathbf{curl}})}%
+\Vert\operatorname{\mathbf{curl}}{\mathbf{E}}\Vert_{L^{2}(\widehat{K})}
\\
\nonumber 
&  \lesssim\Vert\Pi_{\tau}{\mathbf{E}}\Vert_{{\mathbf{X}}^{-1/2}}%
+\Vert\operatorname{\mathbf{curl}}{\mathbf{E}}\Vert_{L^{2}(\widehat{K})}.
\end{align}
Using again the lifting $\boldsymbol{\mathcal{L}}^{\operatorname*{curl},3d}_p$ of
Lemma~\ref{lemma:Hcurl-lifting} and \eqref{eq:orth-3d-a}, we get 
\begin{equation}
\Vert\operatorname{\mathbf{curl}}{\mathbf{E}}\Vert_{L^{2}(\widehat{K})}\leq
\Vert\operatorname{\mathbf{curl}}{\boldsymbol{\mathcal{L}}}^{\operatorname*{curl},3d}_p\Pi_{\tau
}{\mathbf{E}}\Vert_{L^{2}(\widehat{K})}\lesssim\Vert\Pi_{\tau}{\mathbf{E}%
}\Vert_{{\mathbf{X}}^{-1/2}}. \label{eq:thm:H1curl-approximation-200}%
\end{equation}
We conclude the proof by observing
\begin{align*}
& \Vert{\mathbf{v}}-\hatPicurlcom{\mathbf{v}}%
\Vert_{\mathbf{H}(\widehat{K},\operatorname{\mathbf{curl}})}    \leq\Vert{\mathbf{v}%
}-P^{\operatorname*{curl},3d}{\mathbf{v}}\Vert_{\mathbf{H}(\widehat{K}%
,\operatorname{\mathbf{curl}})}+\Vert{\mathbf{E}}\Vert_{\mathbf{H}(\widehat{K}%
,\operatorname{\mathbf{curl}})}
%\overset{(\ref{eq:thm:H1curl-approximation-10c}%
%)}{\lesssim}p^{-1}\Vert{\mathbf{v}}\Vert_{{\mathbf{H}}^{2}(\widehat{K})}%
%+\Vert{\mathbf{E}}\Vert_{\mathbf{H}(\widehat{K},\operatorname{\mathbf{curl}})}\\
\\
&\quad    \overset{(\ref{eq:thm:H1curl-approximation-100}),(\ref{eq:thm:H1curl-approximation-200})}
            {\lesssim}
\Vert {\mathbf v}-P^{\operatorname*{curl},3d}{\mathbf{v}}\Vert_{\mathbf{H}(\widehat{K}%
,\operatorname{\mathbf{curl}})}
+\Vert\Pi_{\tau}{\mathbf{E}}%
\Vert_{{\mathbf{X}}^{-1/2}}\\
& \quad \lesssim 
\Vert{\mathbf{v}}-P^{\operatorname*{curl},3d}{\mathbf{v}}\Vert_{\mathbf{H}(\widehat{K}%
,\operatorname{\mathbf{curl}})}+\Vert\Pi_{\tau}({\mathbf{v}}-\hatPicurlcom{\mathbf{v}})\Vert_{{\mathbf{X}}^{-1/2}}
  \!\!\!\! 
\overset{(\ref{eq:thm:H1curl-approximation-30}%
),(\ref{eq:thm:H1curl-approximation-10c})}{\lesssim}
\! \!\!\!
p^{-1}\Vert{\mathbf{v}%
}\Vert_{{\mathbf{H}}^{2}(\widehat{K})}.
\qedhere
\end{align*}
\end{proof}

For negative norm estimates $\|{\mathbf u} - \hatPicurlcom {\mathbf u}\|_{\widetilde {\mathbf H}^{-s}(\widehat K,
\operatorname{\mathbf{curl}})}$ with $s \ge 0$ we need the following Helmholtz decompositions: 
\begin{lemma}[Helmholtz decomposition] 
\label{lemma:helmholtz-3d}
Any ${\mathbf v} \in {\mathbf H}^1(\widehat K)$ can be written as 
\begin{align}
\label{eq:lemma:helmholtz-3d-10}
{\mathbf v} & = \nabla \varphi_0 + \operatorname{\mathbf{ curl}} \operatorname{\mathbf{curl}} {\mathbf z}_0, \\ 
\label{eq:lemma:helmholtz-3d-20}
{\mathbf v} & = \nabla \varphi_1 + \operatorname{\mathbf{curl}} {\mathbf z}_1,
\end{align}
where $\varphi_0 \in H^2(\widehat K) \cap H^1_0(\widehat K)$ and 
${\mathbf z}_0 \in {\mathbf H}^1(\widehat K,\operatorname{\mathbf{curl}}) \cap {\mathbf H}_0(\widehat K,\operatorname{\mathbf{curl}})$ and 
where $\varphi_1 \in H^2(\widehat K)$ and 
${\mathbf z}_1 \in {\mathbf H}^1(\widehat K,\operatorname{\mathbf{curl}}) \cap {\mathbf H}_0(\widehat K,\operatorname{\mathbf{curl}})$ together with the estimates 
\begin{align*}
\|\varphi_0\|_{H^2(\widehat K)} + 
\|{\mathbf z}_0\|_{{\mathbf H}^1(\widehat K,\operatorname{\mathbf{curl}})} 
& \leq C \|{\mathbf v}\|_{{\mathbf H}^1(\widehat K)}, \\
\|\varphi_1\|_{H^2(\widehat K)} + 
\|{\mathbf z}_1\|_{{\mathbf H}^1(\widehat K,\operatorname{\mathbf{curl}})} 
& \leq C \|{\mathbf v}\|_{{\mathbf H}^1(\widehat K)}. 
\end{align*}
\end{lemma}
\begin{proof}

Before proving these decompositions, we recall the continuous embeddings 
\begin{equation}
\label{eq:saranen}
{\mathbf H}_0(\widehat K,\operatorname{\mathbf{curl}}) \cap {\mathbf H}(\widehat K,\operatorname{div}) \subset {\mathbf H}^1(\widehat K)
\quad \mbox{ and } \quad 
{\mathbf H}(\widehat K,\operatorname{\mathbf{curl}}) \cap {\mathbf H}_0(\widehat K,\operatorname{div}) 
\subset {\mathbf H}^1(\widehat K), 
\end{equation}
which hinge on the convexity of $\widehat K$ (see \cite{birman-solomyak87,saranen82} and the discussion
in \cite[Rem.~{3.48}]{Monkbook}).

We construct the decomposition (\ref{eq:lemma:helmholtz-3d-20}): 
We define $\varphi_1 \in H^1(\widehat K)$ as the solution of 
$$
-\Delta \varphi_1 = -\operatorname{div} {\mathbf v} \quad \mbox{ in $\widehat K$}, 
\qquad \partial_n \varphi_1 = {\mathbf n} \cdot {\mathbf v} \quad \mbox{ on $\partial \widehat K$.}
$$
The contribution ${\mathbf z}_1$ is defined by the following saddle point problem: 
Find $({\mathbf z}_1, \psi) \in {\mathbf H}_0(\widehat K,\operatorname{\mathbf{curl}}) \times H^1_0(\widehat K)$ 
such that 
\begin{align*}
(\operatorname{\mathbf{curl}} {\mathbf z}_1,
\operatorname{\mathbf{curl}} {\mathbf w})_{L^2(\widehat K)} - (\nabla \psi,{\mathbf w})_{L^2(\widehat K)} & = 
(\operatorname{\mathbf{curl}} {\mathbf v} ,{\mathbf w})_{L^2(\widehat K)} \qquad \forall {\mathbf w} \in {\mathbf H}_0(\widehat K,\operatorname{\mathbf{curl}}),
\\
( {\mathbf z}_1,\nabla q)_{L^2(\widehat K)} &=0 \qquad \forall q \in H^1_0(\widehat K). 
\end{align*}
This problem is uniquely solvable, 
we have $\psi = 0$ (since $\operatorname{div} \operatorname{\mathbf{curl}} {\mathbf v} =0$)
and the {\sl a priori} estimate
\begin{align*}
\|{\mathbf z}_1\|_{{\mathbf H}(\operatorname{\mathbf{curl}},\widehat K)} 
\lesssim \|\operatorname{\mathbf{curl}} {\mathbf v}\|_{L^2(\widehat K)} 
\lesssim \|{\mathbf v}\|_{{\mathbf H}^1(\widehat K)}.
\end{align*}
(In the proof of Lemma~\ref{lemma:Hcurl-lifting}, we considered a similar problem in a discrete setting; here, 
the appeal to the discrete Friedrichs inequality of Lemma~\ref{lemma:discrete-friedrichs-3d} needs to replaced with that to the continuous
one, \cite[Cor.~{3.51}]{Monkbook}.)
{}From $\operatorname{div} {\mathbf z}_1 = 0$ and (\ref{eq:saranen}), we furthermore infer 
$\|{\mathbf z}_1\|_{{\mathbf H}^1(\widehat K)} \lesssim \|{\mathbf v}\|_{{\mathbf H}^1(\widehat K)}$. 
The representation (\ref{eq:lemma:helmholtz-3d-20}) is obtained from the observation that the difference 
${\boldsymbol\delta}:= {\mathbf v} - \nabla \varphi_1  - \operatorname{\mathbf{curl}} {\mathbf z}_1$ satisfies, 
by construction, $\operatorname{div} {\boldsymbol \delta} = 0$, $\operatorname{\mathbf{curl}}{\boldsymbol\delta} = 0$, 
${\mathbf n} \cdot {\boldsymbol\delta} = 
({\mathbf n} \cdot {\mathbf v} - \partial_n\varphi_1) - {\mathbf n} \cdot \operatorname{\mathbf{curl}}{\mathbf z}_1 = 
0 - \operatorname{curl}_{\partial\widehat K} \Pi_\tau {\mathbf z}_1  = 0 - 0 = 0$ so that again
(\ref{eq:saranen}) (specifically, in the form \cite[Cor.~{3.51}]{Monkbook}) implies ${\boldsymbol \delta} = 0$. 
Finally, from ${\mathbf v} \in {\mathbf H}^1(\widehat K)$, $\varphi_1 \in H^2(\widehat K)$ and the representation
(\ref{eq:lemma:helmholtz-3d-20}), we infer $\operatorname{\mathbf{curl}} {\mathbf z}_1 \in {\mathbf H}^1(\widehat K)$.

We construct the decomposition (\ref{eq:lemma:helmholtz-3d-10}): 
We define $\varphi_0 \in H^1_0(\widehat K)$ as the solution of 
$$
-\Delta \varphi_0 = -\operatorname{div} {\mathbf v} \quad \mbox{ in $\widehat K$}, 
\qquad \varphi_0 = 0 \quad \mbox{ on $\partial \widehat K$.}
$$
Next, we define $({\mathbf z}_0,\psi) \in {\mathbf H}_0(\widehat K,\operatorname{\mathbf{curl}}) \times H^1_0(\widehat K)$ as the solution 
of the saddle point problem 
\begin{align*}
(\operatorname{\mathbf{curl}} {\mathbf z}_0,
\operatorname{\mathbf{curl}} {\mathbf w})_{L^2(\widehat K)} - (\nabla \psi,{\mathbf w})_{L^2(\widehat K)} & = 
({\mathbf v} - \nabla \varphi_0,{\mathbf w})_{L^2(\widehat K)} \quad \forall {\mathbf w} \in {\mathbf H}_0(\widehat K,\operatorname{\mathbf{curl}}),
\\
( {\mathbf z}_0,\nabla q)_{L^2(\widehat K)} &=0 \qquad \forall q \in H^1_0(\widehat K). 
\end{align*}
Again, this problem is uniquely solvable and, in fact $\psi = 0$ 
(since $\operatorname{div} ({\mathbf v} - \nabla \varphi_0) = 0$).
We have 
$\|{\mathbf z}_0\|_{{\mathbf H}(\operatorname{\mathbf{curl}},\widehat K)} 
\lesssim \|{\mathbf v} - \nabla \varphi_0\|_{L^2(\widehat K)} \lesssim \|{\mathbf v}\|_{L^2(\widehat K)}$. 
Since $\operatorname{div} {\mathbf z}_0 = 0$, we get from (\ref{eq:saranen}) that 
$\|{\mathbf z}_0\|_{{\mathbf H}^1(\widehat K)} \lesssim \|{\mathbf v}\|_{L^2(\widehat K)}$. Finally, 
an integration by parts reveals 
$$
\operatorname{\mathbf{curl}} 
\operatorname{\mathbf{curl}} {\mathbf z_0} = {\mathbf v} - \nabla \varphi_0, 
$$
which is the representation (\ref{eq:lemma:helmholtz-3d-10}). 
\end{proof}

We control the approximation error in negative Sobolev norms.

\begin{theorem}
\label{thm:duality-again}
Assume that all interior angles of the 4 faces of $\widehat K$ are smaller than $2\pi/3$. Then 
for $s \in [0,1]$ and all $\mathbf{u}\in \mathbf{H}^1(\widehat{K},\operatorname{\mathbf{curl}})$ 
there holds the estimate
\begin{align*}
\Vert\mathbf{u}-\hatPicurlcom\mathbf{u}\Vert_{\widetilde{\mathbf{H}}^{-s}(\widehat{K},\operatorname{\mathbf{curl}})} 
\leq C_s p^{-(1+s)} \inf_{{\mathbf v} \in {\mathbf Q}_p(\widehat{K})} \Vert\mathbf{u} - \mathbf{v}\Vert_{\mathbf{H}^1(\widehat{K},\operatorname{\mathbf{curl}})}.
\end{align*}
\end{theorem}

\begin{proof}
By the familiar argument that $\hatPicurlcom$ is a projection, we may restrict the proof to the case 
${\mathbf v} = 0$ in the infimum. The case $s = 0$ is covered by 
Theorem~\ref{thm:H1curl-approximation}. In the remainder of the proof, we will show the case $s = 1$ as the 
case $s \in (0,1)$ then follows by an interpolation argument.

We write $\mathbf{E}:=\mathbf{u}-\hatPicurlcom\mathbf{u}$ for simplicity. 
By definition we have
\begin{align}
\label{eq:lemma:duality-again-100} 
\Vert\mathbf{E}\Vert_{\widetilde{\mathbf{H}}^{-1}(\widehat{K},\operatorname{\mathbf{curl}})} 
& \sim  \Vert\mathbf{E}\Vert_{\widetilde{\mathbf{H}}^{-1}(\widehat{K})} + \Vert\operatorname{\mathbf{curl}}\mathbf{E}\Vert_{\widetilde{\mathbf{H}}^{-1}(\widehat{K})} \\
\nonumber 
&= 
\operatorname*{sup}_{\mathbf{v}\in\mathbf{H}^1(\widehat{K})} \frac{(\mathbf{E},\mathbf{v})_{L^2(\widehat{K})}}{\Vert\mathbf{v}\Vert_{\mathbf{H}^1(\widehat{K})}} + 
\operatorname*{sup}_{\mathbf{v}\in\mathbf{H}^1(\widehat{K})} \frac{(\operatorname{\mathbf{curl}} \mathbf{E},\mathbf{v})_{L^2(\widehat{K})}}{\Vert\mathbf{v}\Vert_{\mathbf{H}^1(\widehat{K})}}. 
\end{align}
We start with estimating the first supremum  in (\ref{eq:lemma:duality-again-100}). 
According to Lemma~\ref{lemma:helmholtz-3d}, any $\mathbf{v}\in\mathbf{H}^1(\widehat{K})$ can be decomposed as
\begin{align*}
\mathbf{v}=\nabla\varphi + \operatorname{\mathbf{curl}}\operatorname{\mathbf{curl}}\mathbf{z}
\end{align*}
with $\varphi\in H^2(\widehat{K}) \cap H_0^1(\widehat{K})$ and $\mathbf{z}\in\mathbf{H}^1(\widehat{K},\operatorname{\mathbf{curl}}) \cap \mathbf{H}_0(\widehat{K},\operatorname{\mathbf{curl}})$. We also observe 
$\operatorname{\mathbf{curl}}\mathbf{z}\in \mathbf{H}^1(\widehat{K},\operatorname{\mathbf{curl}})$. Thus by Lemma~\ref{lemma:helmholtz-like-decomp} we can further decompose $\operatorname{\mathbf{curl}}\mathbf{z}$ as
\begin{align}
\label{eq:lemma:duality-again-25}
\operatorname{\mathbf{curl}}\mathbf{z} = \nabla\varphi_2 + \mathbf{z}_2
\end{align}
with $\varphi_2\in H^2(\widehat{K})$ and $\mathbf{z}_2 \in \mathbf{H}^2(\widehat{K})$. We estimate each
term in the decomposition 
$(\mathbf{E},\mathbf{v})_{L^2(\widehat{K})} = 
(\mathbf{E},\nabla\varphi)_{L^2(\widehat{K})} + 
(\mathbf{E},\operatorname{\mathbf{curl}}\operatorname{\mathbf{curl}}\mathbf{z})_{L^2(\widehat{K})}$ 
separately. Using the orthogonality condition \eqref{eq:Pi_curl-b} and Theorem~\ref{thm:H1curl-approximation}, we get
\begin{align}
\nonumber 
\bigl|(\mathbf{E},\nabla\varphi)_{L^2(\widehat{K})}\bigr| &= 
\bigl| \operatorname*{inf}_{w\in\mathring{W}_{p+1}(\widehat{K})} (\mathbf{E},\nabla(\varphi-w))_{L^2(\widehat{K})} \bigr|\lesssim p^{-1} \Vert\varphi\Vert_{H^2(\widehat{K})} \Vert\mathbf{E}\Vert_{L^2(\widehat{K})} \\
\label{eq:lemma:duality-again-30}
&\lesssim p^{-1} \Vert\mathbf{v}\Vert_{H^1(\widehat{K})} \Vert\mathbf{E}\Vert_{\mathbf{H}(\widehat{K},\operatorname{\mathbf{curl}})} \lesssim p^{-2} \Vert\mathbf{v}\Vert_{H^1(\widehat{K})} \Vert\mathbf{u}\Vert_{\mathbf{H}^1(\widehat{K},\operatorname{\mathbf{curl}})}.
\end{align}
Integration by parts and (\ref{eq:lemma:duality-again-25}) give 
\begin{align}
\nonumber 
&(\mathbf{E},\operatorname{\mathbf{curl}}\operatorname{\mathbf{curl}}\mathbf{z})_{L^2(\widehat{K})} 
 = 
(\mathbf{E},\operatorname{\mathbf{curl}}\mathbf{z}_2)_{L^2(\widehat{K})}  
\\ \nonumber & \quad 
= 
(\operatorname{\mathbf{curl}} \mathbf{E},\mathbf{z}_2)_{L^2(\widehat{K})}  + 
(\Pi_\tau \mathbf{E},\gamma_\tau \mathbf{z}_2)_{L^2(\partial \widehat{K})}   \\
\nonumber
& \quad = 
(\operatorname{\mathbf{curl}} \mathbf{E},\operatorname{\mathbf{curl}} \mathbf{z})_{L^2(\widehat{K})}  - 
(\operatorname{\mathbf{curl}} \mathbf{E},\nabla \varphi_2)_{L^2(\widehat{K})}  + 
(\Pi_\tau \mathbf{E},\gamma_\tau \mathbf{z}_2)_{L^2(\partial \widehat{K})}   \\
\label{eq:lemma:duality-again-50}
& \quad = 
(\operatorname{\mathbf{curl}} \mathbf{E},\operatorname{\mathbf{curl}} \mathbf{z})_{L^2(\widehat{K})}  - 
({\mathbf n} \cdot \operatorname*{\mathbf{curl}}  \mathbf{E},\varphi_2)_{L^2(\partial \widehat{K})}  + 
(\Pi_\tau \mathbf{E},\gamma_\tau \mathbf{z}_2)_{L^2(\partial \widehat{K})}. 
\end{align}
We estimate these three terms separately. For the first term in (\ref{eq:lemma:duality-again-50}), we 
use the orthogonality \eqref{eq:Pi_curl-a} and Theorem~\ref{thm:H1curl-approximation} to get 
\begin{align}
\label{eq:lemma:duality-again-52}
& \bigl| (\operatorname{\mathbf{curl}}\mathbf{E},\operatorname{\mathbf{curl}}\mathbf{z})_{L^2(\widehat{K})} 
\bigr| = \bigl| \operatorname*{inf}_{\mathbf{w}\in\mathring{\mathbf{Q}}_p(\widehat{K})} (\operatorname{\mathbf{curl}}\mathbf{E},\operatorname{\mathbf{curl}}(\mathbf{z}-\mathbf{w}))_{L^2(\widehat{K})} 
\bigr|\\
\nonumber 
& \qquad \qquad \lesssim p^{-1} \Vert\operatorname{\mathbf{curl}}\mathbf{E}\Vert_{L^2(\widehat{K})} \Vert\mathbf{z}\Vert_{\mathbf{H}^1(\widehat{K},\operatorname{\mathbf{curl}})} \\ 
\nonumber 
&
\qquad 
\qquad 
\lesssim p^{-1} \Vert\mathbf{v}\Vert_{H^1(\widehat{K})} \Vert\mathbf{E}\Vert_{\mathbf{H}(\widehat{K},\operatorname{\mathbf{curl}})}
\lesssim p^{-2} \Vert\mathbf{v}\Vert_{H^1(\widehat{K})} \Vert\mathbf{u}\Vert_{\mathbf{H}^1(\widehat{K},\operatorname{\mathbf{curl}})},
\end{align}
cf. also the proof of Lemma~\ref{lemma:picurl-negative-I} for the approximation arguments (use the lifting of Lemma~\ref{lemma:Hcurl-lifting}). For the second term in (\ref{eq:lemma:duality-again-50}), we note that 
$\operatorname*{\mathbf{curl}} {\mathbf E} \in {\mathbf H}^1(\widehat K)$ so that 
the integral over $\partial\widehat K$ can be split into a sum of face 
contributions and $({\mathbf n} \cdot 
\operatorname*{\mathbf{curl}} {\mathbf E})|_f  = \operatorname*{curl}_f \Pi_\tau {\mathbf E}$.
We also observe that our assumption on the angles of the faces of $\widehat K$ 
allows us to select $s = 3/2$ in Lemmas~\ref{lemma:picurl-negative-II} and \ref{lemma:Picurl-face}
since the pertinent $\widehat s$ satisfies $\widehat s > 3/2 = \pi/(2 \pi/3)$. 
We get for each face contribution
\begin{align}
\label{eq:lemma:duality-again-61}
 \bigl| (\operatorname{curl}_f\Pi_\tau\mathbf{E},\varphi_2)_{L^2(f)}\bigr|
& \stackrel{\text{Lem.~\ref{lemma:picurl-negative-II}}}{\lesssim} p^{-3/2} \Vert\operatorname{curl}_f\Pi_\tau\mathbf{E}\Vert_{L^2(f)} \Vert\varphi_2\Vert_{H^{3/2}(f)} \\
&\stackrel{\text{Lem.~\ref{lemma:Picurl-face}}}{\lesssim} p^{-2} \Vert\Pi_\tau\mathbf{u}\Vert_{\mathbf{H}^{1/2}(\operatorname{curl},f)} \Vert\varphi_2\Vert_{H^2(\widehat{K})} 
\\ \nonumber & 
%\nonumber 
\lesssim p^{-2} \Vert\mathbf{u}\Vert_{\mathbf{H}^1(\operatorname{\mathbf{curl}},\widehat{K})} \Vert\mathbf{v}\Vert_{\mathbf{H}^1(\widehat{K})}.
\end{align}
Finally, for the third term in (\ref{eq:lemma:duality-again-50}) we infer with 
Lemmas~\ref{lemma:picurl-negative-I}, \ref{lemma:Picurl-face} 
\begin{align}
\label{eq:lemma:duality-again-54} 
\left| (\Pi_\tau\mathbf{E},\gamma_\tau\mathbf{z}_2)_{L^2(f)}\right|
&\stackrel{\text{Lem.~\ref{lemma:picurl-negative-I}}}{\lesssim} p^{-3/2} \Vert\Pi_\tau\mathbf{E}\Vert_{\mathbf{H}(f,\operatorname{curl})} \Vert\gamma_\tau\mathbf{z}_2\Vert_{\mathbf{H}^{3/2}(f)}\\
\nonumber 
&\stackrel{\text{Lem.~\ref{lemma:Picurl-face}}}{\lesssim} p^{-2}\Vert\Pi_\tau\mathbf{u}\Vert_{\mathbf{H}^{1/2}(f,\operatorname{curl})} \Vert\mathbf{z_2}\Vert_{\mathbf{H}^2(\widehat{K})} \\
\nonumber 
& \lesssim p^{-2} \Vert\mathbf{u}\Vert_{\mathbf{H}^1(\widehat{K},\operatorname{\mathbf{curl}})} \Vert\mathbf{v}\Vert_{\mathbf{H}^1(\widehat{K})}.
\end{align}
Adding \eqref{eq:lemma:duality-again-61} and \eqref{eq:lemma:duality-again-54} over all faces 
and taking note of (\ref{eq:lemma:duality-again-52}) shows that we estimate the first supremum 
(\ref{eq:lemma:duality-again-100}) in the desired fashion. 

To estimate the second supremum in (\ref{eq:lemma:duality-again-100}),
we decompose $\mathbf{v}\in\mathbf{H}^1(\widehat{K})$ as
\begin{align*}
\mathbf{v}=\nabla\varphi+\operatorname{\mathbf{curl}}\mathbf{z}
\end{align*}
with $\varphi\in H^2(\widehat{K})$ and $\mathbf{z} \in \mathbf{H}^1(\widehat{K},\operatorname{\mathbf{curl}}) \cap \mathbf{H}_0(\widehat{K},\operatorname{\mathbf{curl}})$ according to Lemma~\ref{lemma:helmholtz-3d}. Thus we have 
to control the expression $(\operatorname{\mathbf{curl}}\mathbf{E},\mathbf{v})_{L^2(\widehat{K})} = (\operatorname{\mathbf{curl}}\mathbf{E},\operatorname{\mathbf{curl}}\mathbf{z})_{L^2(\widehat{K})} + (\operatorname{\mathbf{curl}}\mathbf{E},\nabla\varphi)_{L^2(\widehat{K})}$. Using the orthogonality condition \eqref{eq:Pi_curl-a} and Theorem~\ref{thm:H1curl-approximation}, the first term is estimated by
\begin{align*}
\bigr| (\operatorname{\mathbf{curl}}\mathbf{E}&,\operatorname{\mathbf{curl}}\mathbf{z})_{L^2(\widehat{K})} 
\bigr| =\bigl| \operatorname*{inf}_{\mathbf{w}\in\mathring{\mathbf{Q}}_p(\widehat{K})} (\operatorname{\mathbf{curl}}\mathbf{E},\operatorname{\mathbf{curl}}(\mathbf{z}-\mathbf{w}))_{L^2(\widehat{K})} \bigr| \\
& \lesssim p^{-1} \Vert\mathbf{E}\Vert_{\mathbf{H}(\widehat{K},\operatorname{\mathbf{curl}})} \Vert\mathbf{z}\Vert_{\mathbf{H}^1(\widehat{K},\operatorname{\mathbf{curl}})} 
\lesssim p^{-2} \Vert\mathbf{u}\Vert_{\mathbf{H}^1(\widehat{K},\operatorname{\mathbf{curl}})} \Vert\mathbf{v}\Vert_{\mathbf{H}^1(\widehat{K},\operatorname{\mathbf{curl}})}.
\end{align*}
For the second term, an integration by parts yields in view of 
$\operatorname{curl}_f \Pi_\tau {\mathbf E} = {\mathbf n} \cdot \operatorname{\mathbf{curl}} {\mathbf E}$ 
\begin{align*}
(\operatorname{\mathbf{curl}}\mathbf{E},\nabla\varphi)_{L^2(\widehat{K})} = \sum_{f\in\mathcal{F}(\widehat{K})}(\operatorname{curl}_f\Pi_\tau\mathbf{E},\varphi)_{L^2(f)},
%(\operatorname{\mathbf{curl}}\mathbf{E},\nabla\varphi)_{L^2(\widehat{K})} = (\mathbf{E},\operatorname{\mathbf{curl}}\nabla\varphi)_{L^2(\widehat{K})} - (\Pi_\tau\mathbf{E},\gamma_\tau\nabla\varphi)_{L^2(\partial\widehat{K})} = \sum_{f\in\mathcal{F}(\widehat{K})}(\operatorname{curl}_f\Pi_\tau\mathbf{E},\varphi)_{L^2(f)},
\end{align*}
where the decomposition into face contributions is again permitted by the regularity of ${\mathbf E}$ and $\varphi$. 
We obtain
\begin{align*}
\bigl| (\operatorname{curl}_f\Pi_\tau\mathbf{E},\varphi)_{L^2(f)}\bigr|\! \lesssim p^{-3/2} \Vert\Pi_\tau\mathbf{E}\Vert_{\mathbf{H}(f,\operatorname{curl})} \Vert\varphi\Vert_{H^{3/2}(f)}\! \lesssim p^{-2} \Vert\mathbf{u}\Vert_{\mathbf{H}^1(\widehat{K},\operatorname{\mathbf{curl}})} \Vert\mathbf{v}\Vert_{\mathbf{H}^1(\widehat{K})}
\end{align*}
by Lemmas~\ref{lemma:picurl-negative-II} and \ref{lemma:Picurl-face}, which finishes the proof.
\end{proof}

For functions ${\mathbf u}$ with discrete $\operatorname{\mathbf{curl}}$, we have the following result.

\begin{lemma}
\label{lemma:better-regularity}
Assume that all interior angles of the $4$ faces of $\widehat K$ are smaller than $2 \pi/3$.
Then for all $k\geq1$ and
all ${\mathbf{u}}\in{\mathbf{H}}^{k}(\widehat{K})$ with $\operatorname*{\mathbf{curl}}%
{\mathbf{u}}\in {\mathbf V}_p(\widehat K)  \supset ({\mathcal{P}}_{p}(\widehat{K}))^{3}$ 
\begin{equation}
\Vert{\mathbf{u}}-\hatPicurlcom{\mathbf{u}}%
\Vert_{\widetilde{\mathbf{H}}^{-s}(\widehat{K},\operatorname{\mathbf{curl}})}\leq C_{s,k}p^{-(k+s)}\Vert{\mathbf{u}}\Vert
_{\mathbf{H}^{k}(\widehat{K})}, \qquad s\in [0,1].
\label{eq:proposition:better-regularity}%
\end{equation}
If $p\geq k-1$, then $\Vert{\mathbf{u}}\Vert_{{\mathbf{H}}%
^{k}(\widehat{K})}$ can be replaced with the seminorm $|{\mathbf{u}%
}|_{{\mathbf{H}}^{k}(\widehat{K})}$. 
Moreover, 
\eqref{eq:proposition:better-regularity} holds for $s =0$ without the conditions on the angles of the
faces of $\widehat K$.
\end{lemma}

\begin{proof}
We employ the regularized right inverses of the operators $\nabla$ and
$\operatorname*{\mathbf{curl}}$ and proceed as in Lemma~\ref{lemma:better-regularity-2d}. We
write, using the decomposition of Lemma~\ref{lemma:helmholtz-like-decomp},
$\displaystyle 
{\mathbf{u}}=\nabla R^{\operatorname*{grad}}({\mathbf{u}}-{\mathbf{R}%
}^{\operatorname*{curl}}\operatorname*{\mathbf{curl}}{\mathbf{u}})+{\mathbf{R}}^{\operatorname*{curl}}\operatorname*{\mathbf{curl}}{\mathbf{u}}=:\nabla
\varphi+{\mathbf{v}}%
$
with $\varphi\in H^{k+1}(\widehat{K})$ and ${\mathbf{v}}\in{\mathbf{H}}%
^{k}(\widehat{K})$ together with
\begin{equation}
\Vert\varphi\Vert_{H^{k+1}(\widehat{K})}+\Vert{\mathbf{v}}\Vert_{{\mathbf{H}%
}^{k}(\widehat{K})}\lesssim  \Vert{\mathbf{u}}\Vert_{{\mathbf{H}}%
^{k}(\widehat{K})}+\Vert\operatorname*{\mathbf{curl}}{\mathbf{u}}\Vert_{{\mathbf{H}%
}^{k-1}(\widehat{K})}  \lesssim \Vert{\mathbf{u}}\Vert_{{\mathbf{H}}%
^{k}(\widehat{K})}.
\label{eq:lemma:projection-based-interpolation-approximation-100}%
\end{equation}
The assumption $\operatorname*{\mathbf{curl}}{\mathbf{u}}\in {\mathbf V}_p(\widehat K)$ 
and %\cite[(3.4)]{hiptmair08} 
Lemma~\ref{lemma:mcintosh}, (\ref{item:lemma:mcintosh-v}) 
imply ${\mathbf{v}}={\mathbf{R}}%
^{\operatorname*{curl}}\operatorname*{\mathbf{curl}}{\mathbf{u}}\in
\mathbf{Q}_p(\widehat{K})$; furthermore, since
$\hatPicurlcom$ is a projection, we have 
${\mathbf{v}}-\hatPicurlcom{\mathbf{v}}=0$. 
With the commuting diagram property $\nabla\hatPigradcom=\hatPicurlcom\nabla$ and \eqref{eq:lemma:demkowicz-grad-3D-20} we get
\begin{align*}
\Vert(\operatorname{I}-\hatPicurlcom){\mathbf{u}}\Vert
_{\widetilde{\mathbf{H}}^{-s}(\widehat{K},\operatorname{\mathbf{curl}})} &= \Vert(\operatorname{I}-\hatPicurlcom)\nabla\varphi+\underbrace{(\operatorname{I}-\hatPicurlcom){\mathbf{v}}}_{=0}\Vert_{\widetilde{\mathbf{H}}^{-s}(\widehat{K},\operatorname{\mathbf{curl}})} \\
&= \Vert\nabla(\operatorname{I}-\hatPigradcom)\varphi\Vert_{\widetilde{\mathbf{H}}^{-s}(\widehat{K})}\lesssim p^{-(k+s)}\Vert\varphi\Vert_{H^{k+1}(\widehat{K})}.
\end{align*}
The proof of (\ref{eq:proposition:better-regularity})
is complete in view of
(\ref{eq:lemma:projection-based-interpolation-approximation-100}). Replacing
$\Vert{\mathbf{u}}\Vert_{{\mathbf{H}}^{k}(\widehat{K})}$ with $|{\mathbf{u}%
}|_{{\mathbf{H}}^{k}(\widehat{K})}$ is possible since the
projector $\hatPicurlcom$ reproduces polynomials
of degree $p$.
\end{proof}

%---------------------------------------

\subsection{Stability of the operator $\protect\hatPidivcom$}

Similar to Lemma~\ref{lemma:Picurl-edge}, we have:

\begin{lemma}
\label{lemma:Pidiv-face} For 
${\mathbf{u}}\in{\mathbf{H}}^{1/2}(\widehat{K},\operatorname*{div})$ and $s \ge 0$ 
we have for each face $f\in{\mathcal{F}}(\widehat{K})$ 
\begin{equation}
\Vert({\mathbf{u}}-\hatPidivcom{\mathbf{u}%
})\cdot{\mathbf{n}}_{f}\Vert_{\widetilde{H}^{-s}(f)}\leq C_s p^{-s}%
\inf_{v \in V_p(f)}
\Vert{\mathbf{u}}\cdot{\mathbf{n}}_{f} - v\Vert_{L^{2}(f)}. 
\label{eq:lemma:Pidiv-face-20}%
\end{equation}
\end{lemma}

\begin{proof}
We first show that for ${\mathbf u} \in {\mathbf H}^{1/2}(\widehat K,\operatorname{div})$
the normal trace ${\mathbf n}_f \cdot {\mathbf u} \in L^2(f)$ for each face $f$. To that end,
we write with the aid of Lemma~\ref{lemma:helmholtz-decomposition-div}
${\mathbf u} = \operatorname{\mathbf{curl}} {\boldsymbol \varphi} + {\mathbf z}$ with
${\boldsymbol \varphi}$, ${\mathbf z} \in {\mathbf H}^{3/2}(\widehat K)$. We have
${\mathbf n}_f \cdot {\mathbf z} \in {\mathbf H}^1(f)$. Noting
${\boldsymbol \varphi}|_f  \in {\mathbf H}^1(f)$ and
$({\mathbf n}_f \cdot \operatorname{\mathbf{curl}} {\boldsymbol \varphi})|_f  
= \operatorname{curl}_f (\Pi_\tau {\boldsymbol \varphi})|_f$, we conclude that
$({\mathbf n}_f \cdot \operatorname{\mathbf{curl}} {\boldsymbol \varphi})|_f \in L^2(f)$.

Note that \eqref{eq:Pi_div-d} and \eqref{eq:Pi_div-c} imply that on faces the operator $\hatPidivcom$ 
is the $L^2$-projection onto $V_p(f)$. Thus, \eqref{eq:lemma:Pidiv-face-20} holds for $s=0$. 
The case $s>0$ follows by a duality argument. 
To that end define $\tilde{e}:=\left(\mathbf{u}-\hatPidivcom\mathbf{u}\right)\cdot \mathbf{n}_f$. 
% and let $v\in H^s(f)$. 
We observe that each $w\in \mathcal{P}_p(f)$ can be written as $w=\overline{w}+(w-\overline{w})$ 
with $\overline{w}$ being the average of $w$ on $f$. Since $w-\overline{w} \in \mathring{V_p}(f)$, \eqref{eq:Pi_div-d} and \eqref{eq:Pi_div-c} imply $(\tilde{e},w)_{L^2(f)} =0$ for any 
$w \in \mathcal{P}_p(f)$. Thus we have for arbitrary $v \in H^s(f)$ 
\begin{align*}
\bigl|(\tilde{e},v)_{L^2(f)}\bigr| &= \bigl|\inf_{w\in \mathcal{P}_p(f)} (\tilde{e},v-w)_{L^2(f)} \bigr| 
\leq \Vert\tilde{e}\Vert_{L^2(f)} \inf_{w\in\mathcal{P}_p(f)} \Vert v-w\Vert_{L^2(f)} \\
& \lesssim p^{-s} \Vert\tilde{e}\Vert_{L^2(f)} \Vert v\Vert_{H^s(f)}.
\qedhere
\end{align*}
\end{proof}

As in the analysis of the operators in the previous sections, the existence of 
a polynomial preserving lifting operator from the boundary $\partial\widehat{K}$ to $\widehat{K}$ 
with appropriate properties plays an important role. Such a lifting operator has been constructed in 
\cite{demkowicz-gopalakrishnan-schoeberl-III}. Paralleling Lemma~\ref{lemma:Hcurl-lifting} 
we modify that lifting slightly to explicitly ensure an additional orthogonality property.

\begin{lemma}
\label{lemma:lifting-operator-div}
Denote the (normal) trace space of ${\mathbf V}_p(\widehat K)$ by 
\begin{equation*} 
V_p(\partial \widehat K):= \{v \in L^2(\partial \widehat K)\,|\, 
\exists {\mathbf v} \in {\mathbf V}_p(\widehat K) \text{ such that } {\mathbf n}_f \cdot {\mathbf v}|_f  = v|_f
\quad 
\forall f \in {\mathcal F}(\widehat K)\}. 
\end{equation*}
There exist $C > 0$ (independent of $p$) and, 
for each $p \in {\mathbb N}_0$ a lifting operator  
$\boldsymbol{\mathcal{L}}^{\operatorname*{div},3d}_p: 
V_p(\partial\widehat K) \rightarrow {\mathbf V}_p(\widehat K)$ 
with the following properties:
\begin{enumerate}[(i)]
\item \label{item:lemma:Hdiv-lifting-i} 
${\mathbf n}_f \cdot \boldsymbol{\mathcal{L}}^{\operatorname*{div},3d}_pz  = z|_f$ for each 
$f \in {\mathcal F}(\widehat K)$ and $z  \in V_p(\partial \widehat K)$. 
\item \label{item:lemma:Hdiv-lifting-iii} There holds 
$\Vert\boldsymbol{\mathcal{L}}^{\operatorname*{div},3d}_p z
\Vert_{\mathbf{H}(\widehat{K},\operatorname*{div})} 
\leq C\Vert z\Vert_{\widetilde{H}^{-1/2}(\partial\widehat{K})}$.
\item \label{item:lemma:Hdiv-lifting-iv} There holds the 
orthogonality $(\boldsymbol{\mathcal{L}}^{\operatorname*{div},3d}_pz,\operatorname*{\mathbf{curl}}\mathbf{v})_{L^2(\widehat{K})} = 0$ for all $\mathbf{v}\in \mathring{\mathbf{Q}}_p(\widehat{K})$.
\end{enumerate}
\end{lemma}

\begin{proof}
Recall the space 
$\mathring{\mathbf{Q}}_{p,\perp}(\widehat K) =\{\mathbf{q}\in\mathring{\mathbf{Q}}_p(\widehat{K})
\colon\!(\mathbf{q},\nabla\psi)_{L^2(\widehat{K})} \!= 0 \, \forall \psi\in \mathring{W}_{p+1}(\widehat{K})\}$ 
defined in Lemma~\ref{lemma:discrete-friedrichs-3d}. 
Let $z\in \widetilde{H}^{-1/2}(\partial\widehat{K})$ be a function with the property $z|_f \in V_p(f)$ for all faces $f\in\mathcal{F}(\widehat{K})$. We define the lifting operator $\boldsymbol{\mathcal{L}}^{\operatorname*{div},3d}_pz:=\boldsymbol{\mathcal{E}}^{\operatorname*{div}}z-\mathbf{w}_0$, 
where $\boldsymbol{\mathcal{E}}^{\operatorname*{div}}: 
H^{-1/2}(\partial\widehat K) \rightarrow {\mathbf H}(\widehat K,\operatorname*{div})$ 
denotes the lifting operator 
from \cite{demkowicz-gopalakrishnan-schoeberl-III} and $\mathbf{w}_0$ is determined by the 
following saddle point problem: Find $\mathbf{w}_0\in \mathring{\mathbf{V}}_p(\widehat{K})$ 
and $\boldsymbol{\varphi}\in \mathring{\mathbf{Q}}_{p,\perp}(\widehat{K})$ 
such that
\begin{subequations}
\label{eq:lemma:saddle-point-div}
\begin{align}
\label{eq:lemma:saddle-point-div-a}
(\operatorname*{div}\mathbf{w}_0,\operatorname*{div}\mathbf{v})_{L^2(\widehat{K})}  +  (\mathbf{v},\operatorname*{\mathbf{curl}}\boldsymbol{\varphi})_{L^2(\widehat{K})} & =  (\operatorname*{div}(\boldsymbol{\mathcal{E}}^{\operatorname*{div}}z),\operatorname*{div}\mathbf{v})_{L^2(\widehat{K})} \quad  \forall\mathbf{v}\in \mathring{\mathbf{V}}_p(\widehat{K}) \\
\label{eq:lemma:saddle-point-div-b}
(\mathbf{w}_0,\operatorname*{\mathbf{curl}}\boldsymbol{\mu})_{L^2(\widehat{K})}  & =  (\boldsymbol{\mathcal{E}}^{\operatorname*{div}}z,\operatorname*{\mathbf{curl}}\boldsymbol{\mu})_{L^2(\widehat{K})} \qquad \forall\boldsymbol{\mu}\in \mathring{\mathbf{Q}}_{p,\perp}(\widehat{K}).
\end{align}
\end{subequations}
Unique solvability of Problem~\eqref{eq:lemma:saddle-point-div} is seen as follows: 
Define the bilinear forms $a(\mathbf{w},\mathbf{q}):=(\operatorname*{div}\mathbf{w},\operatorname*{div}\mathbf{q})_{L^2(\widehat{K})}$ and $b(\mathbf{w},\boldsymbol{\varphi}):=(\mathbf{w},\operatorname*{\mathbf{curl}}\boldsymbol{\varphi})_{L^2(\widehat{K})}$ for $\mathbf{w},\mathbf{q}\in\mathring{\mathbf{V}}_p(\widehat{K})$ and $\boldsymbol{\varphi}\in\mathring{\mathbf{Q}}_{p,\perp}(\widehat{K})$.
Coercivity of $a$ on the kernel of $b$, $\operatorname*{ker}b=\{\mathbf{v}\in\mathring{\mathbf{V}}_p(\widehat{K}): (\mathbf{v},\operatorname*{\mathbf{curl}} \boldsymbol{\mu})_{L^2(\widehat{K})} = 0 \, \forall \boldsymbol{\mu}\in \mathring{\mathbf{Q}}_{p,\perp}(\widehat{K})\}$, follows from the Friedrichs inequality 
for the divergence operator (cf.~Lemma~\ref{lemma:discrete-friedrichs-div}) 
since for 
$ {\mathbf v} \in \operatorname*{ker} b$ one has 
\begin{align*}
a(\mathbf{v},\mathbf{v})&=\Vert\operatorname*{div}\mathbf{v}\Vert_{L^2(\widehat{K})}^2 \geq \frac{1}{2C^2} \Vert\mathbf{v}\Vert_{L^2(\widehat{K})}^2 + \frac{1}{2}\Vert\operatorname*{div}\mathbf{v}\Vert_{L^2(\widehat{K})}^2 \\
&\geq \operatorname*{min}\{\frac{1}{2C^2},\frac{1}{2}\}\Vert\mathbf{v}\Vert_{\mathbf{H}(\widehat{K},\operatorname*{div})}^2.
\end{align*}
Next, the inf-sup condition for $b$ follows easily by considering, 
for given 
$\boldsymbol{\varphi}\in\mathring{\mathbf{Q}}_{p,\perp}(\widehat{K})$, the function 
$\mathbf{w}=\operatorname*{\mathbf{curl}}\boldsymbol{\varphi}\in\mathring{\mathbf{V}}_p(\widehat{K})$ 
in $b({\mathbf w},{\boldsymbol{\varphi}})$ and using 
the Friedrichs inequality for the $\operatorname*{\mathbf{curl}}$ (Lemma~\ref{lemma:discrete-friedrichs-3d}),
\begin{align*}
\frac{b(\mathbf{w},\boldsymbol{\varphi})}{\Vert\mathbf{w}\Vert_{\mathbf{H}(\widehat{K},\operatorname*{div})} \Vert\boldsymbol{\varphi}\Vert_{\mathbf{H}(\widehat{K},\operatorname*{\mathbf{curl}})}} = \frac{\Vert\operatorname*{\mathbf{curl}}\boldsymbol{\varphi}\Vert_{L^2(\widehat{K})}^2}{\Vert\operatorname*{\mathbf{curl}}\boldsymbol{\varphi}\Vert_{L^2(\widehat{K})} \Vert\boldsymbol{\varphi}\Vert_{\mathbf{H}(\widehat{K},\operatorname*{\mathbf{curl}})}} 
\stackrel{\text{Lem.~\ref{lemma:discrete-friedrichs-3d}}}{\geq} C. 
\end{align*}
Thus, the saddle point problem \eqref{eq:lemma:saddle-point-div} has a unique solution 
$(\mathbf{w}_0, {\boldsymbol\varphi}) \in 
\mathring{\mathbf{V}}_p(\widehat{K}) \times \mathring{\mathbf{Q}}_{p,\perp}(\widehat K)$. 
In fact, selecting ${\mathbf v} = \operatorname{\mathbf{curl}} {\boldsymbol \varphi}$ in 
(\ref{eq:lemma:saddle-point-div-a}) shows ${\boldsymbol \varphi} = 0$.
The lifting operator $\boldsymbol{\mathcal{L}}^{\operatorname*{div},3d}_p$ 
now obviously satisfies (\ref{item:lemma:Hdiv-lifting-i}) %, (\ref{item:lemma:Hcurl-lifting-ii}), 
and (\ref{item:lemma:Hdiv-lifting-iv}) 
by construction, cf. \cite[Theorem~7.1]{demkowicz-gopalakrishnan-schoeberl-III} for the properties of 
the operator $\boldsymbol{\mathcal{E}}^{\operatorname*{div}}$. 
For (\ref{item:lemma:Hdiv-lifting-iii}) note that the solution $\mathbf{w}_0$ satisfies 
the estimate $\Vert\mathbf{w}_0\Vert_{\mathbf{H}(\widehat{K},\operatorname*{div})} \lesssim \Vert f\Vert + \Vert g\Vert$, where $f(\mathbf{v})=(\operatorname*{div}(\boldsymbol{\mathcal{E}}^{\operatorname*{div}}z),\operatorname*{div}\mathbf{v})_{L^2(\widehat{K})}$, $g(v)=(\boldsymbol{\mathcal{E}}^{\operatorname*{div}}z,\operatorname*{\mathbf{curl}} \mathbf{v})_{L^2(\widehat{K})}$, and $\Vert \cdot \Vert$ denotes the operator norm. Thus,
\begin{align*}
\Vert f\Vert = \operatorname*{sup}_{\Vert\mathbf{v}\Vert_{\mathbf{H}(\widehat{K},\operatorname*{div})} \leq 1} |(\operatorname*{div}(\boldsymbol{\mathcal{E}}^{\operatorname*{div}}z),\operatorname*{div}\mathbf{v})_{L^2(\widehat{K})}| \leq \Vert\operatorname*{div}(\boldsymbol{\mathcal{E}}^{\operatorname*{div}}z)\Vert_{L^2(\widehat{K})} \lesssim \Vert z\Vert_{\widetilde{H}^{-1/2}(\partial\widehat{K})}.
\end{align*}
The estimate
$\displaystyle 
\Vert g\Vert \lesssim \Vert z\Vert_{\widetilde{H}^{-1/2}(\partial\widehat{K})}
$
is shown similarly. Hence, (\ref{item:lemma:Hdiv-lifting-iii}) follows from
\begin{align*}
\Vert\boldsymbol{\mathcal{L}}^{\operatorname*{div},3d}_p z\Vert_{\mathbf{H}(\widehat{K},\operatorname*{div})} \leq \Vert\boldsymbol{\mathcal{E}}^{\operatorname*{div}}z\Vert_{\mathbf{H}(\widehat{K},\operatorname*{div})} + \Vert\mathbf{w}_0\Vert_{\mathbf{H}(\widehat{K},\operatorname*{div})} \lesssim \Vert z\Vert_{\widetilde{H}^{-1/2}(\partial\widehat{K})}.
\qquad 
\qedhere
\end{align*}
\end{proof}

\begin{theorem}
\label{thm:H1div-approximation} 
Let $\widehat K$ be a fixed tetrahedron.
There exists $C>0$ independent of $p$ such
that for all ${\mathbf{u}}\in{\mathbf{H}}^{1/2}(\widehat{K},\operatorname{div}%
)$
\begin{equation}
\Vert{\mathbf{u}}-\hatPidivcom{\mathbf{u}}%
\Vert_{{\mathbf{H}}(\widehat{K},\operatorname{div})}\leq Cp^{-1/2}%
\inf_{{\mathbf v} \in {\mathbf V}_p(\widehat K)} 
\Vert{\mathbf{u} - \mathbf{v}}\Vert_{{\mathbf{H}}^{1/2}(\widehat{K},\operatorname{div})}.
\label{eq:thmH1div-approximation-10}%
\end{equation}
\end{theorem}

\begin{proof}
\emph{1st~step:} By the projection property of $\hatPidivcom$, it suffices to show 
(\ref{eq:thmH1div-approximation-10}) for ${\mathbf v} = 0$. 

\emph{2nd~step:} As shown in Lemma~\ref{lemma:Pidiv-face}, $\mathbf{u}\cdot \mathbf{n}_f \in L^2(f)$ on each face $f\in \mathcal{F}(\widehat{K})$. Thus we get 
from Lemma~\ref{lemma:Pidiv-face}
\begin{align}
\label{eq:thm:H1div-approximation-12}
\Vert({\mathbf{u}}-\hatPidivcom{\mathbf{u}})\cdot{\mathbf{n}}_{f}\Vert_{\widetilde{H}^{-1/2}(f)}\lesssim p^{-1/2}\Vert{\mathbf{u}}\cdot{\mathbf{n}}_{f}\Vert_{L^2(f)} 
\lesssim p^{-1/2} \|{\mathbf u}\|_{{\mathbf H}^{1/2}(\widehat K,\operatorname{div})}. 
\end{align}

\emph{3rd~step:} 
The difference ${\mathbf u} - \hatPidivcom {\mathbf u}$ is estimated using 
the approximation $P^{\operatorname{div},3d} {\mathbf u}$ of Lemma~\ref{lemma:Pdiv3d}. 
We abbreviate 
$\mathbf{E}:=\hatPidivcom\mathbf{u}-P^{\operatorname*{div},3d}\mathbf{u}\in \mathbf{V}_p(\widehat{K})$. 
Since $\hatPidivcom\mathbf{u}$ satisfies the orthogonality conditions \eqref{eq:Pi_div-b} 
and \eqref{eq:Pi_div-a}, and $P^{\operatorname*{div},3d}\mathbf{u}$ satisfies 
the conditions \eqref{eq:lemma:Pdiv3d} and \eqref{eq:lemma:Pdiv3d-20}, we  have the two orthogonality conditions 
\begin{subequations}
\label{eq:thm:H1div-approximation-20}
\begin{align}
\label{eq:thm:H1div-approximation-20-a}
(\operatorname*{div}\mathbf{E},\operatorname*{div}\mathbf{v})_{L^2(\widehat{K})} &= 0 \qquad \forall \mathbf{v}\in \mathring{\mathbf{V}}_p(\widehat{K}),
\\ 
\label{eq:thm:H1div-approximation-20-b}
(\mathbf{E},\operatorname*{\mathbf{curl}}\mathbf{v})_{L^2(\widehat{K})} 
&= 0 \qquad \forall \mathbf{v}\in \mathring{\mathbf{Q}}_p(\widehat{K}).
\end{align}
\end{subequations}
By Lemma~\ref{lemma:lifting-operator-div}, the orthogonality condition
\[
(\boldsymbol{\mathcal{L}}^{\operatorname*{div},3d}_p(\mathbf{E}\cdot\mathbf{n}),\operatorname*{\mathbf{curl}}\mathbf{v})_{L^2(\widehat{K})} = 0 \qquad \forall \mathbf{v}\in \mathring{\mathbf{Q}}_p(\widehat{K})
\]
holds; hence the discrete Friedrichs inequality 
(Lemma~\ref{lemma:discrete-friedrichs-div}, (\ref{item:lemma:discrete-friedrichs-div-ii})) 
can be applied to $\mathbf{E}-\boldsymbol{\mathcal{L}}^{\operatorname*{div},3d}_p(\mathbf{E}\cdot\mathbf{n}) \in \mathring{\mathbf{V}}_p(\widehat{K})$. Thus, we obtain 
\begin{align}
\label{eq:thm:H1div-approximation-30}
\begin{split}
\Vert\mathbf{E}\Vert_{L^2(\widehat{K})} &\leq \Vert\boldsymbol{\mathcal{L}}^{\operatorname*{div},3d}_p(\mathbf{E}\cdot\mathbf{n})\Vert_{L^2(\widehat{K})} + \Vert\mathbf{E}-\boldsymbol{\mathcal{L}}^{\operatorname*{div},3d}_p
(\mathbf{E}\cdot\mathbf{n})\Vert_{L^2(\widehat{K})} \\
&\lesssim \Vert\mathbf{E}\cdot\mathbf{n}\Vert_{H^{-1/2}(\partial\widehat{K})} + \Vert\operatorname*{div}(\mathbf{E}-\boldsymbol{\mathcal{L}}^{\operatorname*{div},3d}_p(\mathbf{E}\cdot\mathbf{n}))\Vert_{L^2(\widehat{K})} \\
&\lesssim \Vert\mathbf{E}\cdot\mathbf{n}\Vert_{H^{-1/2}(\partial\widehat{K})} + \Vert\operatorname*{div}\mathbf{E}\Vert_{L^2(\widehat{K})}.
\end{split}
\end{align}
\emph{4th~step:} 
Using \eqref{eq:thm:H1div-approximation-20-a}, we get 
\begin{align}
\label{eq:thm:H1div-approximation-40}
\Vert\operatorname*{div}\mathbf{E}\Vert_{L^2(\widehat{K})}^2 = (\operatorname*{div}\mathbf{E},\operatorname*{div}\boldsymbol{\mathcal{L}}^{\operatorname*{div},3d}_p(\mathbf{E}\cdot\mathbf{n}))_{L^2(\widehat{K})} \lesssim \Vert\operatorname*{div}\mathbf{E}\Vert_{L^2(\widehat{K})} \Vert\mathbf{E}\cdot\mathbf{n}\Vert_{H^{-1/2}(\partial\widehat{K})}.
\end{align}
Combining 
(\ref{eq:thm:H1div-approximation-30}),
(\ref{eq:thm:H1div-approximation-40}) 
we arrive at 
\begin{equation}
\label{eq:thm:H1div-approximation-100}
\|{\mathbf E}\|_{{\mathbf H}(\widehat K ,\operatorname{div})} 
\lesssim \|{\mathbf E} \cdot {\mathbf n}\|_{H^{-1/2}(\partial \widehat K)}. 
\end{equation}
\emph{5th~step:} The triangle inequality and the continuity of the normal trace operator give 
\begin{align*}
\Vert \mathbf{u}-\hatPidivcom\mathbf{u}&\Vert_{\mathbf{H}(\widehat{K},\operatorname*{div})} 
\leq \Vert\mathbf{u} - P^{\operatorname{div,3d}}{\mathbf u}\Vert_{\mathbf{H}(\widehat{K},\operatorname*{div})} 
+ \Vert\mathbf{E}\Vert_{\mathbf{H}(\widehat{K},\operatorname*{div})} \\
&\stackrel{(\ref{eq:thm:H1div-approximation-100})}{\lesssim} 
\Vert\mathbf{u} - P^{\operatorname{div,3d}}{\mathbf u}\Vert_{\mathbf{H}(\widehat{K},\operatorname*{div})} 
+ \Vert\mathbf{E}\cdot{\mathbf n}\Vert_{H^{-1/2}(\partial \widehat{K})} \\
&\lesssim \Vert\mathbf{u} - P^{\operatorname{div},3d} \mathbf{u} \Vert_{\mathbf{H}(\widehat{K},\operatorname*{div})} 
+ \sum_{f\in\mathcal{F}(\widehat{K})}\Vert(\mathbf{u}-\hatPidivcom\mathbf{u})\cdot\mathbf{n}_f\Vert_{\widetilde{H}^{-1/2}(f)} \\
&\overset{\eqref{eq:thm:H1div-approximation-12},\text{Lem.~\ref{lemma:Pdiv3d}}}{\lesssim} p^{-1/2}\Vert\mathbf{u}\Vert_{\mathbf{H}^{1/2}(\widehat{K},\operatorname*{div})}.
\qedhere
\end{align*}
\end{proof}

Considering the approximation error in negative Sobolev norms is the next step.

\begin{theorem}
\label{thm:duality-again-div}
Let $\widehat K$ be a fixed tetrahedron.
For $s \in [0,1]$ and for all $\mathbf{u}\in \mathbf{H}^{1/2}(\widehat{K},\operatorname{div})$ 
there holds the estimate
\begin{align*}
\Vert\mathbf{u}-\hatPidivcom\mathbf{u}\Vert_{\widetilde{\mathbf{H}}^{-s}(\widehat{K},\operatorname{div})} 
\leq C_s  p^{-1/2-s} \inf_{{\mathbf v} \in {\mathbf V}_p(\widehat K)} \Vert\mathbf{u} - \mathbf{v}\Vert_{\mathbf{H}^{1/2}(\widehat{K},\operatorname{div})}.
\end{align*}
\end{theorem}

\begin{proof}
In view of the projection property of $\hatPidivcom$, we restrict to showing the estimate with ${\mathbf v} = 0$. 
The case $s = 0$ is shown in Theorem~\ref{thm:H1div-approximation}. We will therefore merely focus 
on the case $s = 1$ as the case $s \in (0,1)$ follows by interpolation. 

We write $\mathbf{E}:=\mathbf{u}-\hatPidivcom\mathbf{u}$ for simplicity. 
By definition we have
\begin{align}
\label{eq:lemma:duality-again-div-10}
\Vert\mathbf{E}\Vert_{\widetilde{\mathbf{H}}^{-1}(\widehat{K},\operatorname{div})} &\sim 
\Vert\mathbf{E}\Vert_{\widetilde{\mathbf{H}}^{-1}(\widehat{K})} + \Vert\operatorname{div}\mathbf{E}\Vert_{\widetilde{H}^{-1}(\widehat{K})} 
\\ 
\nonumber 
&= 
\operatorname*{sup}_{\mathbf{v}\in\mathbf{H}^1(\widehat{K})} \frac{(\mathbf{E},\mathbf{v})_{L^2(\widehat{K})}}{\Vert\mathbf{v}\Vert_{\mathbf{H}^1(\widehat{K})}} + 
\operatorname*{sup}_{{v}\in{H}^1(\widehat{K})} \frac{(\operatorname{div} \mathbf{E},{v})_{L^2(\widehat{K})}}{\Vert{v}\Vert_{{H}^1(\widehat{K})}}.
\end{align}
We start with estimating the first supremum in (\ref{eq:lemma:duality-again-div-10}). 
%\begin{align*}
%\Vert\mathbf{E}\Vert_{\widetilde{\mathbf{H}}^{-1}(\widehat{K})} = \operatorname*{sup}_{\mathbf{v}\in\mathbf{H}^1(\widehat{K})} \frac{(\mathbf{E},\mathbf{v})_{L^2(\widehat{K})}}{\Vert\mathbf{v}\Vert_{\mathbf{H}^1(\widehat{K})}}.
%\end{align*}
We write $\mathbf{v}\in\mathbf{H}^1(\widehat{K})$ as
\begin{align*}
\mathbf{v}=\nabla\varphi + \operatorname{\mathbf{curl}}\mathbf{z}
\end{align*}
with $\varphi\in H^2(\widehat{K})$ and $\mathbf{z}\in\mathbf{H}^1(\widehat{K},\operatorname{\mathbf{curl}}) \cap \mathbf{H}_0(\widehat{K},\operatorname{\mathbf{curl}})$ according to Lemma~\ref{lemma:helmholtz-3d}
and have to bound the two terms in $(\mathbf{E},\mathbf{v})_{L^2(\widehat{K})} = (\mathbf{E},\operatorname{\mathbf{curl}}\mathbf{z})_{L^2(\widehat{K})} + (\mathbf{E},\nabla\varphi)_{L^2(\widehat{K})}$. For the first term, 
Theorem~\ref{thm:H1div-approximation} yields 
\begin{align*}
\bigl|(\mathbf{E},\operatorname{\mathbf{curl}}\mathbf{z})_{L^2(\widehat{K})}\bigr| 
&= \bigl|\operatorname*{inf}_{\mathbf{w}\in\mathring{\mathbf{Q}}_p(\widehat{K})} (\mathbf{E},\operatorname{\mathbf{curl}}(\mathbf{z}-\mathbf{w}))_{L^2(\widehat{K})} \bigr| \lesssim p^{-1} \Vert\mathbf{E}\Vert_{L^2(\widehat{K})} \Vert\mathbf{z}\Vert_{\mathbf{H}^1(\widehat{K},\operatorname{\mathbf{curl}})} \\
&\lesssim p^{-3/2} \Vert\mathbf{u}\Vert_{\mathbf{H}^{1/2}(\widehat{K},\operatorname{div})} \Vert\mathbf{v}\Vert_{\mathbf{H}^1(\widehat{K})},
\end{align*}
where the infimum is estimated as in Lemma~\ref{lemma:picurl-negative-I}, without repeating the arguments here. For the second term, we employ integration by parts to get
\begin{align}
\label{eq:lemma:duality-again-div-40}
(\mathbf{E},\nabla\varphi)_{L^2(\widehat{K})} = -(\operatorname{div}\mathbf{E},\varphi)_{L^2(\widehat{K})} + \sum_{f\in\mathcal{F}(\widehat{K})} (\mathbf{E}\cdot\mathbf{n}_f,\varphi)_{L^2(f)}
\end{align}
Denote by $\overline{\varphi}:=(\int_{\widehat{K}}\varphi)/|\widehat{K}|$ the average of $\varphi$. Integration by parts gives
\begin{align}
\label{eq:lemma:duality-again-div-60}
(\operatorname{div}\mathbf{E},\varphi)_{L^2(\widehat{K})} = (\operatorname{div}\mathbf{E},\varphi-\overline{\varphi})_{L^2(\widehat{K})} + \overline{\varphi}(\mathbf{E}\cdot\mathbf{n},1)_{L^2(\partial\widehat{K})} \!\overset{\eqref{eq:Pi_div-d}}{=}\! (\operatorname{div}\mathbf{E},\varphi-\overline{\varphi})_{L^2(\widehat{K})}.
\end{align}
We then define the auxiliary function $\psi$ by
\begin{align*}
\Delta\psi=\varphi-\overline{\varphi}, \qquad \partial_n\psi=0 \text{ on }\partial\widehat{K}
\end{align*}
and set $\boldsymbol{\Phi}:=\nabla\psi$. Since $\operatorname{div}\boldsymbol{\Phi}=\Delta\psi=\varphi-\overline{\varphi}$, we get
\begin{align}
\label{eq:lemma:duality-again-div-80}
\big| (\operatorname{div}\mathbf{E}&,\varphi-\overline{\varphi})_{L^2(\widehat{K})} \big| 
= \big| (\operatorname{div}\mathbf{E},\operatorname{div}\boldsymbol{\Phi})_{L^2(\widehat{K})} \big| \\
&\overset{\eqref{eq:Pi_div-a}}{=} 
\bigl| 
\operatorname*{inf}_{\mathbf{w}\in\mathring{\mathbf{V}}_p(\widehat{K})} (\operatorname{div}\mathbf{E},\operatorname{div}(\boldsymbol{\Phi}-\mathbf{w}))_{L^2(\widehat{K})} \bigr|
\label{eq:lemma:duality-again-div-82}
\lesssim p^{-1} \Vert\mathbf{E}\Vert_{\mathbf{H}(\widehat{K},\operatorname{div})} 
                 \Vert\boldsymbol{\Phi}\Vert_{\mathbf{H}^1(\widehat{K},\operatorname{div})} \\
&\lesssim p^{-3/2} \Vert\mathbf{u}\Vert_{\mathbf{H}^{1/2}(\widehat{K},\operatorname{div})} 
                \Vert\varphi \Vert_{H^1(\widehat{K})}
\nonumber 
\lesssim p^{-3/2} \Vert\mathbf{u}\Vert_{\mathbf{H}^{1/2}(\widehat{K},\operatorname{div})} 
                \Vert{\mathbf v}\Vert_{{\mathbf H}^1(\widehat{K})}.
\end{align}
Thus, only estimates for the boundary terms in \eqref{eq:lemma:duality-again-div-40} are missing. 
The orthogonality properties \eqref{eq:Pi_div-d} and \eqref{eq:Pi_div-c} as well as Lemma~\ref{lemma:Pidiv-face} lead to
\begin{align*}
\bigl|(\mathbf{E}\cdot\mathbf{n},\varphi)_{L^2(f)}\bigr| 
&= \bigl| \operatorname*{inf}_{w\in V_p(f)} (\mathbf{E}\cdot\mathbf{n},\varphi-w)_{L^2(f)} \bigr|
\lesssim p^{-1} \Vert\mathbf{E}\cdot\mathbf{n}\Vert_{\widetilde{H}^{-1/2}(f)} \Vert\varphi\Vert_{H^{3/2}(f)} \\
&\!\!\!\!\!\!\!\!\stackrel{\text{Lem.~\ref{lemma:Pidiv-face}}}{\lesssim}\!\! p^{-3/2} \Vert\mathbf{u}\cdot\mathbf{n}\Vert_{L^2(f)} \Vert\varphi\Vert_{H^2(\widehat{K})} \lesssim p^{-3/2} \Vert\mathbf{u}\Vert_{\mathbf{H}^{1/2}(\widehat{K},\operatorname{div})} \Vert\mathbf{v}\Vert_{\mathbf{H}^1(\widehat{K})}.
\end{align*}
Thus, we have estimated the first term of \eqref{eq:lemma:duality-again-div-10}.

We now handle the second supremum in (\ref{eq:lemma:duality-again-div-10}). 
Such estimates have already been derived  in \eqref{eq:lemma:duality-again-div-60} 
and \eqref{eq:lemma:duality-again-div-80}; we merely have to note  
that the function $\varphi$ in these lines satisfied $\varphi\in H^2(\widehat{K})$, 
but $H^1(\widehat{K})$-regularity is indeed sufficient as is visible in (\ref{eq:lemma:duality-again-div-82}). 
\end{proof}

For functions whose divergence is a polynomial, we get the following result similar to Lemma~\ref{lemma:better-regularity}.

\begin{lemma}
\label{lemma:better-regularity-div}
Assume that all interior angles of the $4$ faces of $\widehat  K$ are smaller than $2\pi/3$.
For all $k\geq1$, $s \in [0,1]$, and all
${\mathbf{u}}\in{\mathbf{H}}^{k}(\widehat{K})$ with $\operatorname*{div}%
{\mathbf{u}}\in {\mathcal{P}}_{p}(\widehat{K})$ there holds
\begin{equation}
\Vert{\mathbf{u}}-\hatPidivcom{\mathbf{u}}%
\Vert_{\widetilde{\mathbf{H}}^{-s}(\widehat{K},\operatorname{div})}\leq C_{s,k}p^{-(k+s)}\Vert{\mathbf{u}}\Vert
_{\mathbf{H}^{k}(\widehat{K})}.
\label{eq:proposition-better-regularity-div}%
\end{equation}
If $p\geq k-1$, then $\Vert{\mathbf{u}}\Vert_{{\mathbf{H}}%
^{k}(\widehat{K})}$ can be replaced with the seminorm $|{\mathbf{u}%
}|_{{\mathbf{H}}^{k}(\widehat{K})}$.
Moreover, \eqref{eq:proposition-better-regularity-div} holds for  $s = 0$ without the conditions on the 
angles of the faces of $\widehat K$. 
\end{lemma}
\begin{proof}
We write, using the decomposition of Lemma~\ref{lemma:helmholtz-decomposition-div},
$\displaystyle 
{\mathbf{u}}=\operatorname*{\mathbf{curl}} \mathbf{R}^{\operatorname*{curl}}({\mathbf{u}}-{\mathbf{R}%
}^{\operatorname*{div}}\operatorname*{div}{\mathbf{u}})+{\mathbf{R}}^{\operatorname*{div}}\operatorname*{div}{\mathbf{u}}=:\operatorname*{\mathbf{curl}}
\boldsymbol{\varphi}+{\mathbf{z}}%
$
with $\boldsymbol{\varphi}\in \mathbf{H}^{k+1}(\widehat{K})$ and ${\mathbf{z}}\in{\mathbf{H}}%
^{k}(\widehat{K})$ together with
\begin{equation}
\Vert\boldsymbol{\varphi}\Vert_{\mathbf{H}^{k+1}(\widehat{K})}+\Vert{\mathbf{z}}\Vert_{{\mathbf{H}}^{k}(\widehat{K})} \lesssim \Vert{\mathbf{u}}\Vert_{{\mathbf{H}}^{k}(\widehat{K})}+\Vert\operatorname*{div}{\mathbf{u}}\Vert_{H^{k-1}(\widehat{K})} \leq C\Vert{\mathbf{u}}\Vert_{{\mathbf{H}}^{k}(\widehat{K})}.
\label{eq:lemma:projection-based-interpolation-approximation-200}%
\end{equation}
The assumption $\operatorname*{div}{\mathbf{u}}\in {\mathcal{P}}_{p}(\widehat K)$
and %\cite[(3.19)]{hiptmair08} 
Lemma~\ref{lemma:mcintosh}, (\ref{item:lemma:mcintosh-vi})
imply ${\mathbf{z}}={\mathbf{R}}%
^{\operatorname*{div}}\operatorname*{div}{\mathbf{u}}\in
\mathbf{V}_p(\widehat{K})$; furthermore, since
$\hatPidivcom$ is a projection, we conclude
${\mathbf{z}}-\hatPidivcom{\mathbf{z}}=0$. Thus,
we get from the commuting diagram
%Corollary~\ref{cor:thm:projection-based-interpolation}
\begin{align*}
& \Vert(\operatorname{I}-\hatPidivcom){\mathbf{u}}\Vert
_{\widetilde{\mathbf{H}}^{-s}(\widehat{K},\operatorname{div})}=\Vert(\operatorname{I}-\hatPidivcom)\operatorname*{\mathbf{curl}}\boldsymbol{\varphi}+\underbrace{(\operatorname{I}-\hatPidivcom){\mathbf{z}}}_{=0}\Vert_{\widetilde{\mathbf{H}}^{-s}(\widehat{K},\operatorname{div})}\\
&\qquad = \Vert\operatorname{\mathbf{curl}}(\operatorname{I}-\hatPicurlcom)\boldsymbol{\varphi}\Vert_{\widetilde{\mathbf{H}}^{-s}(\widehat{K},\operatorname{div})} 
\leq 
\Vert(\operatorname{I}-\hatPicurlcom)\boldsymbol{\varphi}\Vert_{\widetilde{\mathbf{H}}^{-s}(\widehat{K},\operatorname{\mathbf{curl}})}  \\
& \qquad 
\stackrel{\text{Thm.~{\ref{thm:duality-again}}}}{\lesssim}
p^{-(1+s)}
\inf_{{\mathbf v} \in {\mathbf Q}_p(\widehat K)} \Vert\boldsymbol{\varphi} -{\mathbf v}\Vert_{\mathbf{H}^{1}(\widehat{K},\operatorname{\mathbf{curl}})} 
\stackrel{\text{Lem.~\ref{lemma:Pgrad1d}},\eqref{eq:lemma:projection-based-interpolation-approximation-200}}
{\lesssim} p^{-(k+s)} \Vert\mathbf{u}\Vert_{\mathbf{H}^k(\widehat{K})}.
\end{align*} 
%\begin{align*}
%& \Vert(\operatorname{I}-\hatPidivcom){\mathbf{u}}\Vert
%_{\widetilde{\mathbf{H}}^{-s}(\widehat{K},\operatorname{div})}=\Vert(\operatorname{I}-\hatPidivcom)\operatorname*{\mathbf{curl}}\boldsymbol{\varphi}+\underbrace{(\operatorname{I}-\hatPidivcom){\mathbf{z}}}_{=0}\Vert_{\widetilde{\mathbf{H}}^{-s}(\widehat{K},\operatorname{div})}\\
%&\qquad = \Vert\operatorname{\mathbf{curl}}(\operatorname{I}-\hatPicurlcom)\boldsymbol{\varphi}\Vert_{\widetilde{\mathbf{H}}^{-s}(\widehat{K},\operatorname{div})} 
%\leq 
%\Vert(\operatorname{I}-\hatPicurlcom)\boldsymbol{\varphi}\Vert_{\widetilde{\mathbf{H}}^{-s}(\widehat{K},\operatorname{\mathbf{curl}})}  \\
%& \qquad 
%\stackrel{\eqref{eq:cor:thm:projection-based-interpolation-2}}{\lesssim}
%p^{-(k+s)}\Vert\boldsymbol{\varphi}\Vert_{\mathbf{H}^{k}(\widehat{K},\operatorname{\mathbf{curl}})} 
%\stackrel{\eqref{eq:lemma:projection-based-interpolation-approximation-200}}
%{\lesssim} p^{-(k+s)} \Vert\mathbf{u}\Vert_{\mathbf{H}^k(\widehat{K})}.
%\end{align*} 
Replacing
$\Vert{\mathbf{u}}\Vert_{{\mathbf{H}}^{k}(\widehat{K})}$ with $|{\mathbf{u}%
}|_{{\mathbf{H}}^{k}(\widehat{K})}$ follows from the observation that the
projector $\hatPidivcom$ reproduces polynomials
of degree $p$.
\end{proof}
\subsection*{Acknowledgement} JMM is grateful to his colleague 
Joachim Sch\"oberl (TU Wien) for inspiring discussions on the topic 
of the paper and, in particular, for pointing out the arguments of 
Theorem~\ref{lemma:demkowicz-grad-2D}. CR acknowledges the support 
of the Austrian Science Fund (FWF) under grant P 28367-N35. 
%---------------------------------------------------------------------------
\bibliographystyle{plain}
\bibliography{maxwell}

\newcommand{\noopsort}[1]{} \newcommand{\printfirst}[2]{#1}
  \newcommand{\singleletter}[1]{#1} \newcommand{\switchargs}[2]{#2#1}
  \def\cprime{$'$} \def\cprime{$'$} \def\cprime{$'$}
\begin{thebibliography}{10}

\bibitem{ainsworth-demkowicz09}
Mark Ainsworth and Leszek Demkowicz.
\newblock Explicit polynomial preserving trace liftings on a triangle.
\newblock {\em Math. Nachr.}, 282(5):640--658, 2009.

\bibitem{apel-melenk17}
T.~Apel and J.M. Melenk.
\newblock Interpolation and quasi-interpolation in $h$- and $hp$-version finite
  element spaces.
\newblock In E.~Stein, R.~de~Borst, and T.J.R. Hughes, editors, {\em
  Encyclopedia of Computational Mechanics}, volume~1, pages 1--33. Wiley,
  second edition, 2017.

\bibitem{babuska-craig-mandel-pitkaranta91}
I.~Babu{\v s}ka, A.~Craig, J.~Mandel, and J.~Pitk\"aranta.
\newblock Efficient preconditioning for the $p$ version finite element method
  in two dimensions.
\newblock {\em SIAM J. Numer. Anal.}, 28(3):624--661, 1991.

\bibitem{babuska-suri94}
I.~Babu{\v s}ka and M.~Suri.
\newblock The $p$ and $h$-$p$ versions of the finite element method, basic
  principles and properties.
\newblock {\em SIAM review}, 36(4):578--632, 1994.

\bibitem{bespalov-heuer09}
Alexei Bespalov and Norbert Heuer.
\newblock Optimal error estimation for {H}(curl)-conforming {$p$}-interpolation
  in two dimensions.
\newblock {\em SIAM J. Numer. Anal.}, 47(5):3977--3989, 2009.

\bibitem{birman-solomyak87}
M.~Sh. Birman and M.~Z. Solomyak.
\newblock {$L_2$}-theory of the {M}axwell operator in arbitrary domains.
\newblock {\em Uspekhi Mat. Nauk}, 42(6(258)):61--76, 247, 1987.

\bibitem{boffi-costabel-dauge-demkowicz-hiptmair11}
D.~Boffi, M.~Costabel, M.~Dauge, L.~Demkowicz, and R.~Hiptmair.
\newblock Discrete compactness for the $p$-version of discrete differential
  forms.
\newblock {\em SIAM J. Numer. Anal.}, 49(1):135--158, 2011.

\bibitem{BuffaCiarlet2001b}
A.~Buffa and P.~Ciarlet, Jr.
\newblock On traces for functional spaces related to {M}axwell's equations.
  {II}. {H}odge decompositions on the boundary of {L}ipschitz polyhedra and
  applications.
\newblock {\em Math. Methods Appl. Sci.}, 24(1):31--48, 2001.

\bibitem{BuffaCiarlet2001}
A.~Buffa and P.~Ciarle{t, Jr.}
\newblock On traces for functional spaces related to {M}axwell's equations.
  {P}art {I}: {A}n integration by parts formula in {L}ipschitz polyhedra.
\newblock {\em Math. Meth. Appl. Sci.}, 21:9--30, 2001.

\bibitem{demkowicz-cao05}
W.~Cao and L.~Demkowicz.
\newblock Optimal error estimate of a projection based interpolation for the
  {$p$}-version approximation in three dimensions.
\newblock {\em Comput. Math. Appl.}, 50(3-4):359--366, 2005.

\bibitem{costabel-dauge-demkowicz07}
M.~Costabel, M.~Dauge, and L.~Demkowicz.
\newblock Polynomial extension operators for {$H^1$}, {$H(\rm curl)$} and
  {$H(\rm div)$}-spaces on a cube.
\newblock {\em Math. Comp.}, 77(264):1967--1999, 2008.

\bibitem{costabel-mcintosh10}
Martin Costabel and Alan McIntosh.
\newblock On {B}ogovski\u\i \ and regularized {P}oincar\'e integral operators
  for de {R}ham complexes on {L}ipschitz domains.
\newblock {\em Math. Z.}, 265(2):297--320, 2010.

\bibitem{dauge88}
M.~Dauge.
\newblock {\em Elliptic boundary value problems on corner domains}, volume 1341
  of {\em Lecture Notes in Mathematics}.
\newblock Springer Verlag, 1988.

\bibitem{demkowicz08}
L.~Demkowicz.
\newblock Polynomial exact sequences and projection-based interpolation with
  applications to {M}axwell's equations.
\newblock In D.~Boffi, F.~Brezzi, L.~Demkowicz, L.F. Dur{\'a}n, R.~Falk, and
  M.~Fortin, editors, {\em Mixed Finite Elements, Compatibility Conditions, and
  Applications}, volume 1939 of {\em Lectures Notes in Mathematics}. Springer
  Verlag, 2008.

\bibitem{demkowicz-babuska03}
L.~Demkowicz and I.~Babu{\v s}ka.
\newblock {$p$} interpolation error estimates for edge finite elements of
  variable order in two dimensions.
\newblock {\em SIAM J. Numer. Anal.}, 41(4):1195--1208, 2003.

\bibitem{demkowicz-buffa05}
L.~Demkowicz and A.~Buffa.
\newblock {$H^1$}, {$H({\rm curl})$} and {$H({\rm div})$}-conforming
  projection-based interpolation in three dimensions. {Q}uasi-optimal
  {$p$}-interpolation estimates.
\newblock {\em Comput. Methods Appl. Mech. Engrg.}, 194(2-5):267--296, 2005.

\bibitem{demkowicz-monk-vardapetyan-rachowicz00}
L.~Demkowicz, P.~Monk, L.~Vardapetyan, and W.~Rachowicz.
\newblock de {R}ham diagram for {$hp$} finite element spaces.
\newblock {\em Comput. Math. Appl.}, 39(7-8):29--38, 2000.

\bibitem{demkowicz-gopalakrishnan-schoeberl-I}
Leszek Demkowicz, Jayadeep Gopalakrishnan, and Joachim Sch\"oberl.
\newblock Polynomial extension operators. {I}.
\newblock {\em SIAM J. Numer. Anal.}, 46(6):3006--3031, 2008.

\bibitem{demkowicz-gopalakrishnan-schoeberl-II}
Leszek Demkowicz, Jayadeep Gopalakrishnan, and Joachim Sch\"oberl.
\newblock Polynomial extension operators. {II}.
\newblock {\em SIAM J. Numer. Anal.}, 47(5):3293--3324, 2009.

\bibitem{demkowicz-gopalakrishnan-schoeberl-III}
Leszek Demkowicz, Jayadeep Gopalakrishnan, and Joachim Sch\"oberl.
\newblock Polynomial extension operators. {P}art {III}.
\newblock {\em Math. Comp.}, 81(279):1289--1326, 2012.

\bibitem{grisvard92}
P.~Grisvard.
\newblock {\em Singularities in Boundary Value Problems}.
\newblock Springer Verlag/Masson, 1992.

\bibitem{Grisvard85}
P.G. Grisvard.
\newblock {\em Elliptic {P}roblems in {N}onsmooth {D}omains}.
\newblock Pitman, Boston, 1985.

\bibitem{hiptmair-acta}
R.~Hiptmair.
\newblock Finite elements in computational electromagnetism.
\newblock {\em Acta Numer.}, 11:237--339, 2002.

\bibitem{hiptmair08}
Ralf Hiptmair.
\newblock Discrete compactness for p-version of tetrahedral edge elements,
  2009.
\newblock arXiv:0901.0761.

\bibitem{Mclean00}
W.~McLean.
\newblock {\em Strongly Elliptic Systems and Boundary Integral Equations}.
\newblock Cambridge, Univ. Press, 2000.

\bibitem{melenk-sauter18}
J.~M. {Melenk} and S.~{Sauter}.
\newblock {Wavenumber-explicit $hp$-FEM analysis for Maxwell's equations with
  transparent boundary conditions}, 2018.
\newblock arXiv:1803.01619.

\bibitem{melenk_nshpinterpolation_article}
J.M. Melenk.
\newblock {$hp$}-interpolation of nonsmooth functions and an application to
  {$hp$}-a posteriori error estimation.
\newblock {\em SIAM J. Numer. Anal.}, 43(1):127--155, 2005.

\bibitem{Monkbook}
Peter Monk.
\newblock {\em Finite element methods for {M}axwell's equations}.
\newblock Oxford University Press, New York, 2003.

\bibitem{munoz-sola97}
R.~Mu{\~n}oz-Sola.
\newblock Polynomial liftings on a tetrahedron and applications to the
  $hp$-version of the finite element method in three dimensions.
\newblock {\em SIAM J. Numer. Anal.}, 34(1):282--314, 1997.

\bibitem{nedelec80}
J.-C. N\'ed\'elec.
\newblock Mixed finite elements in {${\bf R}^{3}$}.
\newblock {\em Numer. Math.}, 35(3):315--341, 1980.

\bibitem{rojikdiss}
C.~Rojik.
\newblock {\em $p$-version projection based interpolation}.
\newblock PhD thesis, Institut f\"ur Analysis und Scientific Computing,
  Technische Universit{\"{a}}t Wien, (in prep.).

\bibitem{saranen82}
Jukka Saranen.
\newblock On an inequality of {F}riedrichs.
\newblock {\em Math. Scand.}, 51(2):310--322 (1983), 1982.

\bibitem{SchwabhpBook}
Ch. Schwab.
\newblock {\em {$p$}- and {$hp$}-finite element methods}.
\newblock The Clarendon Press Oxford University Press, New York, 1998.
\newblock Theory and applications in solid and fluid mechanics.

\end{thebibliography}
%---------------------------------------------------------------------------
\end{document}